\documentclass[10pt,twoside, a4paper, english, reqno]{amsart}
\usepackage[dvips]{epsfig}
\usepackage{amscd,cancel}
\usepackage{a4wide}
\usepackage{amssymb}
\usepackage{amsthm}
\usepackage{amsmath}
\usepackage{amsmath,amsthm,amssymb,enumerate}
\usepackage{euscript,mathrsfs}
\usepackage{bbm}
\usepackage{bm}
\usepackage{latexsym}
\usepackage{mathtools,dsfont}

\usepackage{upref}
\usepackage[colorlinks,backref]{hyperref}
\usepackage{color}
\usepackage[usenames,dvipsnames]{xcolor}
\usepackage[toc,page]{appendix}
\usepackage{pdfsync}

\theoremstyle{plain}
\newtheorem{thm}{Theorem}[section]
\theoremstyle{plain}
\newtheorem{lem}[thm]{Lemma}
\newtheorem{proposition}[thm]{Proposition}
\newtheorem{cor}[thm]{Corollary}

\theoremstyle{definition}
\newtheorem{definition}{Definition}[section]
\newtheorem{remark}{Remark}[section]

\newtheorem*{maintheorem*}{Main Theorem}
\newtheorem*{maincorollary*}{Main Corollary}

\newenvironment{Assumptions}
{
\setcounter{enumi}{0}

\begin{enumerate}}
{\end{enumerate} }

\newenvironment{Assumptions2}
{
\setcounter{enumi}{0}

\begin{enumerate}}
{\end{enumerate} }

\newcommand{\R}{\ensuremath{\mathbb{R}}}

\def\e{{\text{e}}}

\def\d#1{\mathrm{ d#1}}

\numberwithin{equation}{section} \allowdisplaybreaks
\begin{document}
	\title[Stochastic Fractional Conservation Laws]
	{Stochastic fractional conservation laws }
	\author{Abhishek Chaudhary}
	\date{\today}
	
	\maketitle
	
	\medskip
	\centerline{ Centre for Applicable Mathematics, Tata Institute of Fundamental Research}
	\centerline{P.O. Box 6503, GKVK Post Office, Bangalore 560065, India}\centerline{abhi@tifrbng.res.in}
	\medskip
	\begin{abstract}
	In this paper, we consider the Cauchy problem for the nonlinear fractional conservation laws driven by a multiplicative noise. In particular, we are concerned with the well-posedness theory and the study of the long-time behavior of solutions for such equations. We show the existence of desired kinetic solution by using the vanishing viscosity method. In fact, we establish strong convergence of the approximate viscous solutions to a kinetic solution. Moreover, under a nonlinearity-diffusivity condition, we prove the existence of an invariant measure using the well-known Krylov-Bogoliubov theorem. Finally, we show the uniqueness and ergodicity of the invariant measure.
\end{abstract}
\section{Introduction} 
Nonlinear nonlocal integro-PDEs have great demand in mathematics due to their applications in several different areas such as mathematical finance \cite{mathmatical finance}, flow in porous media \cite{porous media}, radiation hydrodynamics \cite{hyd}, and overdriven gas detonations \cite{overdriven}. The addition of a Brownian noise to this type of physical model is completely natural, as it represents exterior perturbations or a lack of knowledge of certain substantial parameters. We are interested in the well-posedness theory and long-time behavior of solutions for the fractional conservation laws driven by Brownian noise. In this article,
we consider the following stochastic fractional conservation laws 
\begin{equation}\label{1.1}
\begin{cases}
	d u(x,t)+\mbox{div}(F(u(x,t)))d t +(-\Delta)^\alpha[A(u(x,t))]dt=\Psi(x,u(x,t))\,dB(t),\,\,\, &x \in \mathbb{T}^N,\,\, t \in(0,T),\\
	u(x,0)=u_0(x),\,&\,x\in\mathbb{T}^N
\end{cases}
\end{equation}
where $N\ge 1$, $B$ is a cylinderical Wiener process, $u_0$ is the given initial function, $F:\mathbb{R}\mapsto\mathbb{R}^N,$ $A:\mathbb{R}\mapsto\mathbb{R}$ are given (sufficiently smooth) functions (see Section \ref{section 2} for the complete list of assumptions). Here $(-\Delta)^\alpha$ denotes the fractional Laplacian operator of order $\alpha\in(0,1)$, defined pointwise as follows
\begin{align}\label{2.4}
(-\Delta)^{\alpha}[\kappa](x):=-P.V.\int_{\mathbb{R}^N}(\kappa(x+z)-\kappa(x))\lambda(z)\,dz,\,\,\, \forall\,\,x\in\mathbb{T}^N,
\end{align}
where $\lambda(z)=\frac{1}{|z|^{N+2\alpha}},\,\, z\,\ne\,0$, $\lambda(0)=0$, and $\kappa$ is a sufficiently regular function. Finally, the function $ \Psi : L^2(\mathbb{T}^N )\to L_2(\mathfrak{O}; L^2(\mathbb{T}^N ))$ is a $L_2(\mathfrak{O}; L^2(\mathbb{T}^N ))$-valued function, where $L_2(\mathfrak{O};L^2(\mathbb{T}^N))$ represents the collection of Hilbert-Schmidt operators from $\mathfrak{O}$ to $L^2(\mathbb{T}^N )$.
\subsection{Earlier works}
Over the last decade, there have been many contributions to the larger area of stochastic partial differential equations that are driven by Brownian noise.
The equation \eqref{1.1} could be viewed as a stochastic perturbation of the nonlocal hyperbolic equation. In the absence of nonlocal term (fractional Laplacian) along with $\Psi=0$, the equation \eqref{1.1} becomes the well-known conservation laws.

Kruzhkov \cite{Kruzhkov} first introduced the concept of entropy solution for conservation laws, and established well-posedness theory for the same in the $L^\infty$ framework. The entropy formulation of the elliptic-parabolic-hyperbolic problem has been developed by Carrillo \cite{Carrillo}. We mention also the work by Bendahmane $\&$ Karlsen \cite{Karlsen} for the anisotropic diffusion case. On the other hand, we refer to the works of Alibaud \cite{Alibaud}, Cifani et.al. \cite{Cifani, Alibaud 2} for deterministic fractional conservation laws and deterministic fractional degenerate convection-diffusion equations respectively. The concept of the kinetic solution was first introduced in the paper of Lions et. al \cite{lions} for the conservation laws. Then Chen $\&$ Perthame \cite{chen} developed a well-posedness theory for general degenerate parabolic-hyperbolic equations with non-isotropic nonlinearity.  Moreover, the well-posedness theory for the nonlocal conservation laws was recently studied by Alibaud et al. in \cite{N-2}.

In the stochastic set-up, the work of Kim \cite{Kim} first established the concept of entropy solutions for stochastic conservation laws and developed the well-posedness theory for one-dimensional conservation laws driven by additive Brownian noise. We refer to the work of Vallet $\&$ Wittbold \cite{Vallet} for the multi-dimensional Dirichlet problem. However, when the noise is multiplicative in nature, then situation differs significantly from that of additive noise case, and in this direction we mention the work of Feng $\&$ Nualart \cite{Feng} for one-dimensional balance laws. For multi-dimensional case, Debussche and Vovelle \cite{deb} extended the well-posedness theory for the stochastic scalar conservation laws. Moreover, Dotti and Vovelle \cite{sylvain} showed convergence of approximations to stochastic scalar conservation laws. For the more delicate case of degenerate parabolic equations, we refer to the works of Hofmanova \cite{hofmanova}, Debussche et al. \cite{vovelle}. For more works in this direction, we refer to \cite{BaVaWitParab,BisKoleyMaj,KMV,KMV1,ujjwal,koley2013multilevel,Koley3}. Furthermore, well-posedness theory for stochastic fractional degenerate problem has recently been studied in \cite{neeraj,Neeraj2}. On the other hand, the analysis of the long-time behavior of global solutions is the second pillar in the theory of SPDEs, after the analysis of the well-posedness.
The first paper in this direction is by E, Khanin, Mazel, Sinani \cite{EKAS} for the analysis of the invariant measure related to the periodic inviscid Burgers equation with stochastic forcing in space dimension one. We also mention the work by Bakhtin \cite{Bak}, which deals with the scalar conservation laws with Poisson random forcing on the whole line. Moreover, Debussche and Vovelle \cite{vovelle 2} established the existence of an invariant measure for a scalar first-order conservation laws with stochastic forcing under a hypothesis of non-degeneracy on the flux. Furthermore, Chen and Pang \cite{chen.2} also have proved existence of an invariant measure for the stochastic anisotropic parabolic-hyperbolic equation. A number of authors have contributed since then, and we mention the works of \cite{B-2,B-3}.

\subsection{Aim and outline of this paper} Equation \eqref{1.1} is a fractional convection-diffusion equation, and this class of equations have received considerable interest recently thanks to the wide variety of applications. Due to the nonlinearity involves in equation \eqref{1.1}, classical solutions to \eqref{1.1} are hard to find, and weak solutions must be sought. We adapt the notion of kinetic formulation for solutions to \eqref{1.1}, which is only weak in space variable but strong in time variable, as proposed in \cite{sylvain}. To sum up, we aim at developing the following results related to \eqref{1.1}:
\begin{itemize}
\item[(i)] Our first aim is to establish the well-posedness theory for solutions to the Cauchy problem \eqref{1.1}. The proof of well-posedness theory is rather classical but still quite technical, the hardest part is probably the uniqueness part (see step 3 in the proof of Theorem \ref{comparison}). The proof of existence is based on the vanishing viscosity argument, while the uniqueness proof is settled by extending Kruzkov’s doubling of variable technique in the presence of multiplicative noise. Note that the noise coefficient $\Psi$ has an explicit dependency on the spatial variable $x$ as well. In this case, the proof of contraction principle (see Theorem \ref{th3.6}) requires a change in hierarchy while computing limits with respect to various parameters involved (for details, see \eqref{final1}). Indeed, this technical hurdle forces us to analyze the general equation with values of $\alpha$ in $(0,\frac{1}{2})$. However, when the noise coefficient $\Psi=\Psi(u)$, we can extend well-posedness theory for $\alpha\in (0, 1)$ (see Theorem \ref{main result}).

\item[(ii)] The second main contribution of this paper is the analysis of the long-time behavior of the kinetic solution. In fact, we obtain existence of an invariant measure through the Krylov-Bogoliubov theorem under a nonlinearity-diffusivity condition \eqref{non} (see Theorem \ref{Invariant measure}). The proof of the existence of invariant measure itself is much more delicate than in the hyperbolic case \cite{vovelle 2}. Moreover, we show that the law of kinetic solution $u(t)$ converges to an unique invariant measure as $t\to\,\infty$ (see \eqref{convergence of law}). We also establish a well-posedness theory for initial data in $L^1(\mathbb{T}^N)$ which is needed for the framework of invariant measure (see Theorem \ref{main result 2}).
\end{itemize}
The paper is organized as follows. In Section \ref{section 2}, we give the details of assumptions and introduce the basic setting, define the notion of a kinetic solution, and state our main results, Theorem \ref{main result}, Theorem \ref{main result 2}, Theorem \ref{Invariant measure}. Section \ref{section 3} is devoted to well-posedness theory in $L^p$-setting. Subsection \ref{subsection 3.1} is devoted to the proof of uniqueness together with the proof of $L^1$-contraction principle. In Subsection \ref{subsection 3.2}, we prove the existence part of Theorem \ref{main result} which is divided into two parts. First, we prove the existence for the smooth initial condition. Second, we relax the hypothesis upon initial data and prove the existence for general initial data. In Section \ref{section 4}, we prove the existence and uniqueness of invariant measure, Theorem \ref{Invariant measure}. In Appendix \ref{A}, we briefly derive the kinetic formulation of equation \eqref{1.1}. In Appendix \ref{B}, we give the proof of Theorem \ref{main result 2}.

\section{Technical framework and statement of main result}\label{section 2}
\subsection{Hypothesis}
In this subsection, we give the precise assumptions on each of the term appearing in the equation \eqref{1.1}. We work on a finite-time interval $[0,T]$ for the well-posedness theory and consider periodic boundary conditions: $x\in\mathbb{T}^N$, where $\mathbb{T}^N $ is the N-dimensional torus. Suppose that $(\Omega,\mathcal{F},\mathbb{P},({\mathcal{F}}_t),(\beta_k(t))_{k\ge1})$ is a stochastic basis with a complete, right continuous filtration. Let $\mathcal{P}$ indicates the predictable $\sigma$-algebra on $\Omega\times[0,T]$ associated to $(\mathcal{F}_t)_{t\ge0}$. $B$ is a cylindrical Wiener process  defined as $B=\sum_{n\ge1}\beta_n l_n$ , where the coefficients $\beta_n$ are independent Brownian processes and $(l_n)_{n\ge1}$ is a complete orthonormal basis in a Hilbert space $\mathfrak{O}$. Here, we introduce the canonical space $\mathfrak{O}\subset\mathfrak{O}_0$ via
$$\mathfrak{O}_0=\bigg\{v=\sum_{n\ge1}\theta_n \mathfrak{\gamma}_n;\, \sum_{n\ge1}\frac{\theta_n^2}{n^2}\,\textless\,\infty\bigg\}$$
with the norm 
$$\| v\|_{\mathfrak{O}_0}^2=\sum_{n\ge1}\frac{\theta_n^2}{n^2},\,\,\, v=\sum_{n\ge1}\theta_n \mathfrak{\gamma}_n .$$
Notice that the embedding $\mathfrak{O}\hookrightarrow \mathfrak{O}_0$ is Hilbert-Schmidt. Moreover, $B$ has $\mathbb{P}$-almost surely trajectories in $C([0,T];\mathfrak{O}_0)$. For all $w\in L^2(\mathbb{T}^N)$ we consider a mapping $\Psi: \mathfrak{O}\to L^2(\mathbb{T}^N)$ intruduced by $\Psi(w)\mathfrak{\gamma}_n = h_n(\cdot, w(\cdot))$. Then we define
$$\Psi(x,w)=\sum_{n\ge1} h_n(x,w)\mathfrak{\gamma}_n ,$$
For details related to the It\^o integral $\int_0^t\Psi(x,u){\rm dB}(t)$, we refer to \cite{prato}. Now we separate the rest of the assumptions in two parts; first part for the well-posedness theory and second part for the existence and uniqueness of invariant measure.
\subsection*{Assumptions for the well-posedness theory} Here, we list the assumptions which are necessary to state the Theorems \ref{main result} \& \ref{main result 2}. 
\begin{Assumptions}
\item\label{A1} $F$ is a $C^2(\mathbb{R};\mathbb{R}^N)$-function with a polynomial growth of its derivative, in the following sense: there exists $r\,\ge\,1$  such that
\begin{align}\label{F.2}
	\sup_{|\zeta|\,\le\,\delta}|F'(\xi)-F'(\xi+\zeta)|\,\le\,C_F (1+|\xi|^{r-1})\,\delta.
\end{align}
\item\label{A2} $A:\mathbb{R}\to\mathbb{R}$ is a non-decreasing Lipschitz continuous function.
\item\label{A3} $h_k\in C(\mathbb{T}^N \times\mathbb{R})$ with the following bounds:
\begin{align}\label{2.2}
	H^2(x,w)=\sum_{k\ge1}|h_k(x,w)|^2 \le C_0 (1+|w|^2),
\end{align}
\begin{align}\label{2.3} \sum_{k\ge1}|h_k(x,w)-h_k(y,z)|^2\le C_{\Psi}(|x-y|^2+|w-z|g(|w-z|)),
\end{align}
where $x,y\in\mathbb{T}^N$ and $w,z\in\mathbb{R}$ and $g$ is a continuous non-decreasing function on $\mathbb{R}_+$  satisfying, $g(0)=0$, and we assume also $0\le g(\mathfrak{\zeta})\le1$ for all $\mathfrak{\zeta}\in\mathbb{R}_+$.

\item \label{A4} If noise is additive in nature, then $h_k\in C(\mathbb{T}^N)$, with the following bounds
\begin{align}\label{noise1}
	H^2(x)=\sum_{k\ge1}|h_k(x)|^2,\,\,\sum_{k\ge\,1} \|h_k\|_{C(\mathbb{T}^N)}^2 \le C_0,
\end{align}
\begin{align}\label{noise2} \sum_{k\ge1}|h_k(x)-h_k(y)|^2\,\le\,C_{\Psi}(|x-y|^2),
\end{align}
where $x,y\in\mathbb{T}^N$.
\end{Assumptions}
\subsection*{Assumptions for the proof of invariant measure} Here, we list the assumptions which are necessary to state the Theorem \ref{Invariant measure}.
\begin{Assumptions2}
\item\label{H1}  $A\in C^2(\mathbb{R};\mathbb{R})$ is a non-decreasing Lipschitz continuous function with the following bound:
\begin{align}
	|A''(\mathfrak{\zeta})|\,\le\,C\,(|\mathfrak{\zeta}|+1)\,\qquad\,\forall\mathfrak{\zeta}\in\mathbb{R}.
\end{align}
\item \label{H2} $F\in C^2(\mathbb{R};\mathbb{R}^N)$ with the following bound:
\begin{align}\label{flux existance}
	|F''(\mathfrak{\zeta})|\,\le\,C\,(|\mathfrak{\zeta}|+1).\,\,\,\qquad\, 
\end{align}
\item\label{H3} The noise is additive in nature, and $h_k\in C^2(\mathbb{T}^N)$, with the following bounds
\begin{align}\label{noise1}
	H^2(x)=\sum_{k\ge1}|h_k(x)|^2,\,\,\sum_{k\ge\,1} \|h_k\|_{C^2(\mathbb{T}^N)}^2 \le C_0,
\end{align}
\begin{align}\label{noise2} \sum_{k\ge1}|h_k(x)-h_k(y)|^2\,\le\,C_{\Psi}(|x-y|^2),
\end{align}
where $x,y\in\mathbb{T}^N$.
\item\label{H4} The functions $h_k$ satisfies the cancellation condition
\begin{align}\label{cancellation} \int_{\mathbb{T}^N}h_k(x)dx=0.
\end{align}
\end{Assumptions2}
\subsection{Definitions for $L^p$-setting:} \label{section 3.1}
Here, we introduce the formulation of kinetic solution to \eqref{1.1} as well as the basic definitions concerning the notion of kinetic solution. Let $\mathcal{M}^+\big([0,T]\times\mathbb{T}^N\times\mathbb{R}\big)$ be the collection of non-negative Radon measures over $[0,T]\times\mathbb{T}^N\times\mathbb{R}$.
\begin{definition}[\textbf{Kinetic measure}]\label{kinetic measure 1}
A mapping $\mathfrak{m}$ from $\Omega$ to $\mathcal{M}^+\big([0,T]\times\mathbb{T}^N\times\mathbb{R}\big)$ is said to be kinetic measure provided the following conditions hold: 
\begin{enumerate}
	\item[(i)] $\mathfrak{m}$ is measurable in the following sense: for each $\psi\in C_0(\mathbb{T}^N\times[0,T]\times\mathbb{R})$ the mapping $\mathfrak{m}(\psi):\Omega\to\mathbb{R}$ is measurable,
	\item[(ii)] if $\mathbb{B}_{L}^c=\{\mathfrak{\zeta}\in\mathbb{R}; |\mathfrak{\zeta}|\ge L\}$, then $\mathfrak{m}$ vanishes for large $\mathfrak{\zeta}$ in the sense: 
	$$\lim_{L\to\infty}\mathbb{E} \mathfrak{m}(\mathbb{T}^N\times[0,T]\times \mathbb{B}_{L}^c)=0.$$
\end{enumerate}	
\end{definition}
\begin{definition}[\textbf{Kinetic solution}]\label{kinetic solution in lp setting} Let $u_0\in L^p(\Omega\times\mathbb{T}^N)$ for all $p\in[1,+\infty)$. A $L^1(\mathbb{T}^N)$- valued stochastic process $(u(t))_{t\in[0,T]}$ is said to be a solution to \eqref{1.1} with initial datum $u_0$, if $(u(t))_{t\in[0,T]}$ and
$f(t):=\mathbbm{1}_{u(t)\textgreater\mathfrak{\zeta}}$ have the following properties:
\begin{enumerate}
	
	\item[1.] $u\in L_{\mathcal{P}}^p(\mathbb{T}^N\times[0,T]\times\Omega),\,\,\,\,\forall\,\, p\in[1,+\infty)$ ,
	\item[2.] for all $\varphi\in C_c(\mathbb{T}^N\times\mathbb{R}),$ $\mathbb{P}$-almost surely, $t\to \langle f(t),\varphi\rangle$ is c\'adl\'ag,
	\item[3.]for all $p\in[1,+\infty),$ there exists $C_p\ge0$ such that 
	\begin{align}\label{2.5}
		\mathbb{E}(\sup_{0\le t\le T}\| u(t)\|_{L^p(\mathbb{T}^N)}^p)\le C_p,
	\end{align}
	\item[4.] 
	Let $\mathfrak{\mathfrak{\eta}}_{1}:\Omega\to\,\mathcal{M}^{+}\big(\mathbb{T}^N\times[0,T]\times\mathbb{R}\big)$ be defined as follows:  
	$$\mathfrak{\eta}_1(x,t,\xi)=\int_{\mathbb{R}^N}|A(u(x+z,t))-A(\mathfrak{\zeta})|\mathbbm{1}_{\mbox{Conv}\{u(x,t),u(x+z,t)\}}(\mathfrak{\zeta})\lambda(z)dz.$$
	There exists a random kinetic measure $\mathfrak{m}$ in sense of Definition \ref{kinetic measure 1} such that $\mathbb{P}$-almost surely, $\mathfrak{m}\ge\mathfrak{\mathfrak{\eta}}_1$, and the pair $(f,\mathfrak{m})$ satisfies the following formulation: for all $\varphi \in C_c^2(\mathbb{T}^N\times\mathbb{R})$, $t\in[0,T]$,
	\begin{align}\label{2.6}
		\langle f(t),\varphi \rangle &= \langle f_0, \varphi \rangle + \int_0^t\langle f(s),F'(\mathfrak{\zeta})\cdot\nabla\varphi\rangle ds -\int_0^t\langle f(s), \, A'(\mathfrak{\zeta})\,(-\Delta)^\alpha[\varphi]\rangle ds\notag\\
		&\qquad+\sum_{k=1}^\infty \int_0^t \int_{\mathbb{T}^N} h_k(x,u(x,s))\varphi(x,u(x,s) dx d\beta_k(s)\notag\\
		&\qquad+\frac{1}{2}\int_0^t\int_{\mathbb{T}^N}\partial_{\mathfrak{\zeta}}\varphi(x,u(x,s)) H^2 (x,u(x,s))dxds -\mathfrak{m}(\partial_{\mathfrak{\zeta}}\varphi)([0,t]),
	\end{align}
	$\mathbb{P}$-almost surely,
	where, $f_0(x,\xi)=\mathbbm{1}_{u_0\textgreater\xi}$, $H^2(x, \xi) := \sum_{k\ge1} |h_k(x,\xi)|^2$.
\end{enumerate}
\end{definition}
\noindent
Here, we use the brackets $\langle.,.\rangle$, to indicate the duality between $C_c^\infty(\mathbb{T}^N \times\mathbb{R})$ and the space of distributions over $\mathbb{T}^N\times\mathbb{R}$, we use $\mathfrak{m}(\varphi)$ to denote the Borel measure on $[0,T]$ defined by
$$\mathfrak{m}(\varphi):B\mapsto\int_{\mathbb{T}^N\times B\times \mathbb{R}}\varphi(x,\mathfrak{\zeta})d\mathfrak{m}(x,t,\mathfrak{\zeta}),\,\,\, \varphi \in C_b(\mathbb{T}^N\times\mathbb{R})$$
for all $B$ Borel subset of $[0,T]$, and for all $c,d \in \R$
$$\text{Conv}\{c, d\} := (\text{min}\{c, d\},\text{max}\{c, d\}).$$
	\subsection{Definitions for $L^1$-setting:}
Here, we recall the definition of kinetic solution as well as the related definitions used for $L^1$-setting. This is a generalization of the concept of kinetic solution as defined in Definition \ref{kinetic solution in lp setting}. In this context, the corresponding kinetic measure is not finite and one can show only suitable decay at infinity. 
\begin{definition}\textbf{(Kinetic measure)}\label{definition kinetic measure}\label{kinetic measure 2}
	A mapping $\mathfrak{m}_1$ from $\Omega$ to $\mathcal{M}^+\big([0,T]\times\mathbb{T}^N\times\mathbb{R}\big)$ is said to be a kinetic measure provided the following conditions hold:
	\begin{enumerate}
		\item[(i)] $\mathfrak{m}_1$ is measurable in the following sense: for each $\kappa\in C_0(\mathbb{T}^N\times[0,T]\times\mathbb{R})$ the mapping $\mathfrak{m}_1(\kappa):\Omega\to\mathbb{R}$ is measurable,
		\item[(ii)] if $\widetilde{\mathbb{B}}_{L}=\{\xi\in\mathbb{R}; (L+1)\,\ge\,|\mathfrak{\zeta}|\ge L\}$, then 
		$$\lim_{L\to\infty}\mathbb{E} \mathfrak{m}_1(\mathbb{T}^N\times[0,T]\times\widetilde{\mathbb{B}}_{L})=0.$$
	\end{enumerate}		
\end{definition}
\begin{definition}\textbf{$\big($Kinetic solution$).$}\label{definition kinetic solution in l1 setting}
	Let $u_0\,\in\,L^1(\mathbb{T}^N).$ A $L^1(\mathbb{T}^N)$- valued stochastic process $(u(t))_{t\in[0,T]}$ is said to be a solution to \eqref{1.1} with initial datum $u_0$, if the following conditions are satisfied,
	\begin{enumerate}
		
		\item[1.]$u\in L_{\mathcal{P}}^1(\mathbb{T}^N\times[0,T]\times\Omega)$ and satisfying
		$$\mathbb{E}\bigg(\sup_{t\in[0,T]}\|u(t)\|_{L^1(\mathbb{T}^N)}\bigg)\,\textless\,+\infty,$$
		\item[2.] for all $\varphi\in C_c^2(\mathbb{T}^N\times\mathbb{R}),$ $\mathbb{P}$-almost surely, $t\to \langle f(t),\varphi\rangle$ is c\'adl\'ag,
		\item[3.] let $\mathfrak{\mathfrak{\eta}}_{1}:\Omega\to\,\mathcal{M}^{+}\big(\mathbb{T}^N\times[0,T]\times\mathbb{R}\big)$ be defined as follows:  
		$$\mathfrak{\mathfrak{\eta}}_1(x,t,\xi)=\int_{\mathbb{R}^N}|A(u(x+z,t))-A(\mathfrak{\zeta})|\mathbbm{1}_{\mbox{Conv}\{u(x,t),u(x+z,t)\}}(\mathfrak{\zeta})\lambda(z)dz.$$
		There exists a kinetic measure ${\mathfrak{m}}$ in sense of Definition \ref{kinetic measure 2} such that $\mathfrak{m}_1\,\ge\,\mathfrak{\mathfrak{\eta}}_1$, $\mathbb{P}$-almost surely, and the pair $\big(f,\, \mathfrak{m}_1 \big)$  satisfies the following formulation: for all $\varphi\,\in\,C_c^2(\mathbb{T}^N\times\mathbb{R})$, for all $t\in[0,T],$		
		\begin{align}\label{formulation}
			\langle {f}(t),\varphi \rangle &= \langle f_0, \varphi \rangle + \int_0^t\langle f(s),F'(\mathfrak{\zeta})\cdot\nabla\varphi\rangle ds -\int_0^t\langle f(s), A'(\mathfrak{\zeta})\,(-\Delta)^\alpha[\varphi]\rangle ds\notag\\
			&\qquad+\sum_{k=1}^\infty \int_0^t \int_{\mathbb{T}^N} h_k(x)\varphi(x,u(x,s) dx d\beta_k(s)\notag\\
			&\qquad+\frac{1}{2}\int_0^t\int_{\mathbb{T}^N}\partial_{\mathfrak{\zeta}}\varphi(x,u(x,s)) H^2 (x)dxds -\mathfrak{m}_1(\partial_{\mathfrak{\zeta}}\varphi)([0,t]),
		\end{align}
		$\mathbb{P}$-almost surely, where, $f_0(x,\mathfrak{\zeta})=\mathbbm{1}_{u_0\textgreater\mathfrak{\zeta}}$, $H^2(x) := \sum_{k\ge1} |h_k(x)|^2$.
	\end{enumerate}
\end{definition}
\subsection{Invariant measure} Suppose that noise is only additive in nature. Let $u(t,0,u_0)$ denote the solution at time $t$ from starting at time $0$ (for existence of solution, see Theorem \ref{main result 2}). Then, we can define the operator $\mathcal{Q}_t:B_b(L^1(\mathbb{T}^N))\to B_b(L^1(\mathbb{T}^N))$ by
$$\big(\mathcal{Q}_t\kappa\big)(v)=\mathbb{E}[ \kappa(u(t,0,v)) ]$$
Suppose that $u(t,t_0,v)$ denotes the solution to \eqref{1.1} starting at time $t_0$ from an $\mathcal{F}_{t_0}$-measurable initial condition $v$. By contraction principal (see Theorem \ref{main result 2}), it holds true that $\mathbb{P}$-almost surely,
$$\|u(s,t,u_1)-u(s,t,u_2)\|_{L^1(\mathbb{T}^N)}\,\le\,\|u_1-u_2\|_{L^1(\mathbb{T}^N)}$$
It implies that $\mathcal{Q}_t\kappa\,\in\,C_b(L^1(\mathbb{T}^N))$ for all $\kappa\,\in\,C_b(L^1(\mathbb{T}^N))$. The equation \eqref{1.1} defines a Markov process in the following sense:
$$\mathbb{E}[\kappa(u(t+s,0,v))|\mathcal{F}_t]=\mathcal{Q}_s\kappa\big(u(t,0,v)\big),\,\,\,\forall\,\,\kappa\,\in \,C_b(L^1(\mathbb{T}^N))\,\,\,\forall\,t,s\,\textgreater\,0,\,\,\,\mathbb{P}-\text{almost surely},$$
then the semigroup property $\mathcal{Q}_{t+s}=\mathcal{Q}_t\circ \mathcal{Q}_s$ holds. The equation $\eqref{1.1}$ defines a Feller Markov process. The semigroup $(\mathcal{Q}_t)_{t\,\ge\,0}$ is called Feller.
We now denote the law of $u(t,0,v)$ by $\lambda_{t,v}$,  then
$$\mathcal{Q}_t\kappa(v)=\mathbb{E}[\kappa(u(t,0,v))]=\int_{L^1({\mathbb{T}^N})}\kappa(y)\lambda_{t,v}(dy),$$
and $\langle\cdot,\cdot\rangle$ denotes the duality product between bounded Borel functions and probability measures, we obtain 
$$\mathcal{Q}_t\kappa(v)=\langle \kappa, \lambda_{t,v}\rangle=\langle \mathcal{Q}_t\kappa,\delta_v\rangle.$$ 
It follows that $\lambda_{t,v}=\mathcal{Q}_t^*\delta_v.$ More generally, if we consider a solution to \eqref{1.1} with initial condition $u_0$ having the initial law $\lambda$, we have $\lambda_{t,u_0}=\mathcal{Q}_t^*\lambda$.
\begin{definition}\textbf{(Invariant measure)}
	We say that a probability measure $\lambda$ on $L^1(\mathbb{T}^N)$ is an invariant measure if 
	$$\mathcal{Q}_t^*\lambda=\lambda,\,\,\,\,\,\,\text{for all}\,t\,\ge\,0.$$
\end{definition}
Now, we have all in hand to formulate the Krylov-Bogoliubov Theorem \cite[Proposition 11.3]{prato}.
\begin{thm}
	Suppose that the semigroup $(\mathcal{Q}_t)$ is Feller. Suppose that there exists a random variable $u_0$, a sequence $(T_n)$ increasing to $\infty$ and a probability measure $\lambda$ such that
	$$\frac{1}{T_n}\int_0^{T_n}\lambda_{t,u_0} dt\,\to\,\lambda\,\,\text{weak-*}. $$
	Then $\lambda$ is an invariant measure for $(\mathcal{Q}_t)$.
\end{thm}

\subsection{Nonlinearity-Diffusivity condition and recent developments}\label{non-linearity diffusivity condition}
Here, we will discuss some recent developments in the analysis of long-time behavior of solutions of nonlinear deterministic, stochastic PDEs and some comments about assumptions on flux function and non-linearity. Let us introduce the nonlinearity-diffusivity condition on which we will work on as follows: suppose that there exist $s\,\in\,(0,1)$, and $C \textgreater\,0$, independent of $\gamma$, such that 
\begin{align}\label{non}
	\sup_{\tau\in\mathbb{R}, k\in\,Z^n}\int_{\mathbb{R}}\frac{\gamma\big(A'(\mathfrak{\zeta})|k|^{2\alpha-1}+\gamma\big)}{(\gamma+A'(\mathfrak{\zeta})|k|^{2\alpha-1})^2+|F'(\mathfrak{\zeta})\cdot\frac{k}{|k|}+\tau|^2}d\mathfrak{\zeta} =:\mathfrak{\mathfrak{\eta}}(\gamma)\,\le\,C\,\gamma^s\,\,\to\,0\,\,\text{as}\,\gamma\,\to\,0.
\end{align}
This condition implies that no interval of $\mathfrak{\zeta}$ on which, flux function $F(\mathfrak{\zeta})$ is affine and nonlinear diffusive function $A(\mathfrak{\zeta})$ is degenerate. We can write this condition in the simple and more standard setting: for any $\tau\in\mathbb{R}\,\,\hat{k}\,\in\,\mathbb{S}^{N-1}$, we have 
$$\mathcal{L}\{\mathfrak{\zeta}\in\mathbb{R};\,F'(\mathfrak{\zeta})\cdot\hat{k}+\tau=0,\, A'(\mathfrak{\zeta})=0\}=0.$$
We mention some results under similar type of conditions on flux functions: hyperbolic conservation laws \cite{vovelle2,vovelle 2},  and diffusion matrix in anisotropic degenerate parabolic-hyperbolic equations \cite{chen.2,chen3}. These results can summarized as follows: Consider
\begin{equation}\label{hyperbolic}
	\begin{cases}
		d u(x,t)+\mbox{div}(F(u(x,t)))dt =\Psi(x)\,dB(t),\,\,\,& x \in \mathbb{T}^N,\,\, t \in(0,T),\\
		u(x,0)=u_0(x) & x\in\mathbb{T}^N.
	\end{cases}
\end{equation}
The purely hyperbolic conservation laws is special case of \eqref{1.1} with $A(\mathfrak{\zeta})=0$ and $\Psi=0$. If put $A(\mathfrak{\zeta})=0$ in condition \eqref{non}, it gives that for any $\tau\in\mathbb{R},\,\,\hat{k}\,\in\,\mathbb{S}^{N-1}$, 
\begin{align}\label{non_degeneracy}\mathcal{L}\{\mathfrak{\zeta}\in\mathbb{R};\,F'(\mathfrak{\zeta})\cdot\hat{k}+\tau=0\}=0,
\end{align}
\noindent
then all solutions of the deterministic scalar conservation laws converges to $\bar{u}_0=\int_{\mathbb{T}^N}u_0(x)dx$  (see\,\cite{vovelle2}).\\
There is also an article by Chen and Perthame \cite{chen3}, where authors studied the large time behaviour of periodic solutions in $L^\infty$ to the deterministic nonlinear anisotropic degenerate parabolic-hyperbolic equations of second order.
\begin{align}\label{aniso}
	\begin{cases}
		\partial _t u(x,t)+\mbox{div}_x(F(u(x,t))) =\nabla_x(A(u)\cdot\nabla_x u),\,\,\, &x \in \mathbb{T}^N,\,\, t \in(0,T),\\
		u(x,0)=u_0 &x\in\mathbb{T}^N.
	\end{cases}
\end{align}
\begin{thm}[\textbf{Anisotropic degenerate parabolic-hyperbolic equation}]
	Let $u\,\in\,L^\infty([0,\infty)\times\mathbb{R}^N)$ be the unique periodic entropy solution to \eqref{aniso}. Suppose that the flux function $F$ and the diffusion matrix $A$ satisfy the following non-linearity diffusivity condition: for any $\delta\,\textgreater\,0$
	\begin{align}
		\sup_{|\tau|+|k|\,\ge\,\delta}\int_{|\mathfrak{\zeta}|\,\le\,\|u_0\|_{\infty}}\frac{\gamma}{\gamma+|F'(\mathfrak{\zeta})\cdot k+\tau|^2+(k^TA(\mathfrak{\zeta})k)^2}d\mathfrak{\zeta}:=\omega_\delta(\gamma)\,\to\,0\,\,\text{as}\,\,\gamma\,\to\,0.
	\end{align}
	Then we have
	$$\|u(t)-\bar{ u}_0\|_{\infty}\,\to0\,\,\, \text{as}\,\,t\,\to\,\infty,$$
	where
	$$\bar{u}_0=\int_{\mathbb{T}^N}u_0(x)dx.$$
\end{thm}
However, in stochastic setting, Debussche et. al \cite{deb} used the following non-degeneracy condition on the flux function,
\begin{align}\label{non-degeneracy flux}
	i(\varepsilon)=\,\sup_{\tau\in\mathbb{R}\,\hat{k}\in\mathbb{S}^{N-1}}|\{\mathfrak{\zeta}\in\,\mathbb{R};|\tau+\hat{k}\cdot F'(\mathfrak{\zeta})|\,\textless\,\varepsilon\}|\,\le\,C_1\,\varepsilon^b
\end{align}
for some $C_1\,\textgreater\,0$ and $b\,\textgreater\,0$.
\begin{thm}[\textbf{The stochastic conservation laws}] Assumptions (H.2)-(H.4) hold. Then there exists an invariant measure for \eqref{hyperbolic} in $L^1(\mathbb{T}^N)$. If the condition \eqref{H2} is strengthened into the hypothesis that $F$ is sub-quadratic in the following sense:
	\begin{align}
		|F''(\mathfrak{\zeta})|\,\le\,C,\,\,\forall\,\,\,\mathfrak{\zeta}\in\mathbb{R},
	\end{align}
	then the invariant measure is unique.
\end{thm}

The existence of invariant measure is also established for the stochastic anisotropic parabolic-hyperbolic equation in \cite{chen.2}.
It is clear from the above discussion that our assumption \eqref{non} on the flux $F$ and the diffusive function $A$ is in line with the assumptions in recently developed works.
\subsection{Some basic facts about Young measure}\label{section 3.3}
Roughly speaking a Young measure is a parametrized family of probability measures where the parameters are drawn from a finite measure space. Let $\mathcal{P}(\mathbb{R})$ be the space of probability measure on $\mathbb{R}$. 
\begin{definition}[\textbf{Young measure}] Suppose that $(\mathcal{X},\lambda)$ is a finite measure space. A mapping $\mathfrak{\mathcal{V}}$ from $\mathcal{X}$ to $\mathcal{P}(\mathbb{R})$ is said to be a Young measure if, for all $\kappa \in C_b(\mathbb{R})$, the map $w\to\mathfrak{\mathcal{V}}_w(\kappa)$ from $\mathcal{X}$ into $\mathbb{R}$ is measurable. We say that a Young measure $\mathfrak{\mathcal{V}}$ vanishes at infinity if , for all $p\ge1$,
	$$\int_\mathcal{X} \int_{\mathbb{R}}|\mathfrak{\zeta}|^p d\mathfrak{\mathcal{V}}_w(\mathfrak{\zeta})d\lambda(w)\textless \infty.$$
\end{definition}
\begin{definition}[\textbf{Kinetic function}]Suppose that $(\mathcal{X}$,$\lambda$) is a finite measure space. A measurable function $f:\mathcal{X}\times\mathbb{R}\to[0,1]$ is said to be a kinetic function  if there exists a Young measure $\mathfrak{\mathcal{V}}$ on $\mathcal{X}$ vanishing at infinity such that, for $\lambda$-a.e. $w\in \mathcal{X}$, for all $\mathfrak{\zeta}\in\mathbb{R}$,
	$$f(w,\mathfrak{\zeta})=\mathfrak{\mathcal{V}}_w(\mathfrak{\zeta},\infty).$$
\end{definition}
\begin{definition}[\textbf{Equilibrium }]
	A measurable function $f:\mathcal{X}\times\mathbb{R}\to[0,1]$ is said to be an equilibrium if there exists a measurable function $v:\mathcal{X}\to\mathbb{R}$ such that $f(w,\mathfrak{\zeta})=\mathbbm{1}_{v(w)\textgreater\mathfrak{\zeta}}$ almost every $w\in \mathcal{X}$.
\end{definition}
\begin{definition}[\textbf{Convergence of Young measure}]
	Suppose that $(\mathcal{X},\lambda)$ is a finite measure space. A sequence of Young measures $\mathcal{V}^n$ on $\mathcal{X}$ said to conveges to a Young measure $\mathcal{V}$ on $\mathcal{X}$ provided the following convergence holds: for all $h\in L^1(\mathcal{X})$, for all $g\in C_b(\mathbb{R})$,
	\begin{align}\label{2.12}
		\lim_{n\to +\infty}\int_{\mathcal{X}} h(y)\int_{\mathbb{R}}g(\zeta)d\mathcal{V}_z^n(\zeta)d\lambda(y)&=\int_{\mathcal{X}} h(y)\int_{\mathbb{R}}g(\zeta)d\mathcal{V}_y(\zeta)d\lambda(y).
	\end{align}
\end{definition}
\begin{remark}[\textbf{Kinetic formulation in terms of Young measure}]Suppose that u is a kinetic solution of \eqref{1.1}. If we define  $f(x,t,\zeta)=\mathbbm{1}_{u(x,t)\textgreater\mathfrak{\zeta}}$, then we have $\partial_{\mathfrak{\zeta}}f(x,t,\zeta)=-\delta_{u(x,t)=\mathfrak{\zeta}}$, where $\mathfrak{\mathcal{V}}=\delta_{u=\mathfrak{\zeta}}$ is a Young measure on $\Omega\times[0,T]\times\mathbb{T}^N $, therefore we can write \eqref{2.6} as follows:
	for all $\varphi \in C_c^2(\mathbb{T}^N\times\mathbb{R})$, $\mathbb{P}$-almost surely, for all $t\in[0,T]$
	\begin{align}\label{2.10}
		\langle f(t),\varphi \rangle &= \langle f_0, \varphi \rangle + \int_0^t\langle f(s),F'(\mathfrak{\zeta})\cdot\nabla\varphi\rangle ds -\int_0^t\langle f(s),A'(\mathfrak{\zeta}) (-\Delta)^\alpha[\varphi]\rangle ds\notag\\
		&\qquad+\sum_{k=1}^\infty \int_0^t \int_{\mathbb{T}^N}\int_{\mathbb{R}} h_k(x,\mathfrak{\zeta})\varphi(x,\mathfrak{\zeta})d\mathfrak{\mathcal{V}}_{s,x}(\mathfrak{\zeta})dx d\beta_k(s)\notag\\
		&\qquad+\frac{1}{2}\int_0^t\int_{\mathbb{T}^N}\int_{\mathbb{R}}\partial_{\mathfrak{\zeta}}\varphi(x,\mathfrak{\zeta}) H^2 (x,\mathfrak{\zeta})d\mathfrak{\mathcal{V}}_{s,x}(\mathfrak{\zeta})dxds -\mathfrak{m}(\partial_{\mathfrak{\zeta}}\varphi)([0,t]).
	\end{align}
\end{remark} 
\subsection{The main results.}\label{section 3.4}
To conclude this section, we state our main results.
\begin{thm}[\textbf{Well-posedness in $L^p$-setting}]\label{th2.10}\label{main result}\label{main result 1}
	Let the assumptions \eqref{A1}-\eqref{A3}  be true. Let $\alpha\in (0, \frac{1}{2})$ and $u_0\in L^p(\Omega\times\mathbb{T}^N)$ for all $p\in[1,\infty)$. Then there exists a unique kinetic solution to \eqref{1.1} with initial data in the sense of Definition \ref{kinetic solution in lp setting}  and it has almost surely continuous trajectories in $L^p(\mathbb{T}^N)$, for all $p\in[1,\infty)$. Let $u_1, u_2$ be kinetic solutions to \eqref{1.1} with initial data $u_{1,0}$ and $u_{2,0}$, respectively, then for all $t\in[0,T]$,
	\begin{align}
		\mathbb{E}\left\Vert u_1(t)-u_2(t)\right\Vert_{L^1(\mathbb{T}^N)}\le\mathbb{E}\left\Vert u_{1,0}-u_{2,0}\right\Vert_{L^1(\mathbb{T}^N)}.\notag
	\end{align}
	Moreover, if assume that noise is only multiplicative, $\Psi=\Psi(u)$, then the same result also holds for $\alpha\in (0, 1)$.
\end{thm}

In the next theorem, we also establish well-posedness theory for initial data in $L^1(\mathbb{T}^N)$ which is required for the framework of invariant measure.
\begin{thm}[\textbf{Well-posedness in $L^1$-setting}]\label{existance and uniqueness}\label{main result 2} Let the assumptions \eqref{A1}-\eqref{A2} $\&$ \eqref{A4} be true. Let $\alpha\in (0, \frac{1}{2})$ and $u_0\in L^1(\mathbb{T}^N)$. Then there exists a solution $u$ to \eqref{1.1} with initial data $u_0$ in the sense of Definiton \ref{definition kinetic solution in l1 setting}. Besides, u has almost surely continuous trajectories in $L^1(\mathbb{T}^N)$. Suppose that $F$ is sub-quadratic in the following sense:
	$|F''(\mathfrak{\zeta})|\,\le\,C,\,\forall\mathfrak{\zeta}\in\mathbb{R}.$
	Let $u_1, u_2$ be kinetic solutions to \eqref{1.1} with initial data $u_0^1$ and $u_0^2\,\in\,L^1(\mathbb{T}^N)$, respectively. Then the following holds: $\mathbb{P}$-almost surely, for all $t\in[0,T]$
	\begin{align}\label{contraction principal}
		\|u_1(t)-u_2(t)\|_{L^1(\mathbb{T}^N)}\,\le\,\|u_0^1-u_0^2\|_{L^1(\mathbb{T}^N)}.
	\end{align}
\end{thm}
\begin{thm}[\textbf{Invariant measure}]\label{Invariant measure} Let initial data $u_0\in L^1(\mathbb{T}^N)$ such that $\int_{\mathbb{T}^N}u_0(x)dx=0$. Let the assumptions \eqref{H1}-\eqref{H4}  be true. Let $\alpha\in (0, \frac{1}{2})$.
	Then there exists an invariant measure $\lambda$ for \eqref{1.1} in $L^1(\mathbb{T}^N)$.
	Let $\lambda_{t,u_0}$ be law of solution $u(t)$ with initial data $u_0$. Suppose that, if the Asumptions \eqref{H1}-\eqref{H2} are strengthened into the hypothesis that $F$ and $A$ are sub-quadratic in the following sense:
	\begin{align}\label{flux uniqueness}
		|F''(\mathfrak{\zeta})|\,\le\,C,\,\,\,\qquad\,\,|A''(\mathfrak{\zeta})|\,\le\,C,\,\,\,\,\qquad\forall\mathfrak{\zeta}\in\mathbb{R},
	\end{align}
	then the invariant measure $\lambda$ is unique and $\lambda_{t,u_0}$ converges to an unique invariant measure $\lambda$ as $t\to\infty$.
\end{thm}
\textbf{Comment:}
In Section \ref{section 3}, one remembers that we will work with two cases. In first case, we have general noise and $\alpha\in (0, \frac{1}{2})$. In the second case, we have only multiplicative noise and $\alpha \in (0,1)$. A large amount of proof are same for both the cases. However, we mention and provide separate proofs when it is different. 
\section{Well-posednees in $L^p$-setting: Proof of Theorem \ref{main result 1}}\label{section 3}
\subsection{Comparison principle:}\label{subsection 3.1}
Here, we follow the approach of \cite{sylvain} and obtain a property of kinetic solution, which will help in the proof of continuity of trajectories of solutions. In the next proposition, we see that $\mathbb{P}$-almost surely c\'adl\'ag property is independent of test function $\varphi$ and that the limit from the left at any time $t\textgreater\,0$ is defined by a kinetic function. 
\begin{proposition}\label{Proposition 3.1} Suppose that if $(u(t))_{t\in[0,T]}$ is a solution to \eqref{1.1} with initial data $u_0$, then we have the following two properties,
	\begin{enumerate}
		\item[1.] there exists a measurable subset $\Omega_1\subset\Omega$ of full probability such that, for all $\omega\in\Omega_1$, for all $\varphi\in C_c^2(\mathbb{T}^N\times\mathbb{R})$, $t\mapsto\langle f(\omega,t),\varphi\rangle$ is c\'adl\'ag.
		\item[2.] there exists an $L^\infty(\mathbb{T}^N\times\mathbb{R};[0,1])$-valued process $(f^{-}(t))_{t\in(0,T]}$ such that: for each $t\in(0,T]$, for all $\omega\in\Omega_1$, for all $\varphi\in C_c^2(\mathbb{T}^N\times\mathbb{R})$, $f^{-}(t)$ is a kinetic function on $\mathbb{T}^N$ which defines the left limit of $s\mapsto\langle f(s),\varphi \rangle$ at t as follows:
		\begin{align}\label{3.1}
			\langle f^{-}(t),\varphi\rangle=\lim_{s_n\to t^{-}}\langle f(s_n),\varphi\rangle .
		\end{align}
	\end{enumerate}
	\begin{proof}
		For a proof, we refer to \cite[Proposition 2.10]{sylvain}.
	\end{proof}
\end{proposition}

\begin{remark}[Left and right limits] In Proposition \ref{Proposition 3.1} we prove something more than what in state. Indeed, for $\omega\in{\Omega}_1$, we have $f(s_n)\to f^-(t)$ in $L^\infty(\mathbb{T}^N\times\mathbb{R})$ for the weak-* topology, when $s_n\uparrow t$, which implies \eqref{3.1}. By similar arguments, we can show that $f(s_n)\to f(t)$ in $L^\infty(\mathbb{T}^N\times\mathbb{R})$ weak-* when $s_n\downarrow t.$
\end{remark}
\begin{remark}By using Fatou's lemma we obtain the following bounds: for all $\omega\in\Omega_1$,
	\begin{align}\label{3.3}\sup_{t\in[0,T]}\int_{\mathbb{T}^N}\int_{\mathbb{R}}|\mathfrak{\zeta}|^p d\mathfrak{\mathcal{V}}_{x,t}^{-}(\mathfrak{\zeta})dx\le C_p(\omega),\,\,\,\mathbb{E}(\sup_{t\in[0,T]}\int_{\mathbb{T}^N}\int_{\mathbb{R}}|\mathfrak{\zeta}|^p d\mathfrak{\mathcal{V}}_{x,t}^{-}(\mathfrak{\zeta})dx)\le C_p.
	\end{align}
\end{remark}
\begin{remark}[Equation for $f^{-}$] Passing to the limit in \eqref{2.6} for an increasing sequence $t_n$ to $t$, we obtain the following equation on $f^{-}$:
	\begin{align}\label{3.4}
		\langle f^{-}(t),\varphi \rangle &= \langle f(0),\varphi\rangle +\int_0^t\langle f(s),F'(\mathfrak{\zeta})\cdot\nabla_x\varphi\rangle ds-\int_0^t\langle f(s),A'(\mathfrak{\zeta}) (-\Delta)^\alpha[\varphi]\rangle ds\notag\\
		&\qquad+\sum_{k\ge 1}\int_0^t \int_{\mathbb{T}^N}\int_{\mathbb{R}}h_k(x,\mathfrak{\zeta})\varphi(x,\mathfrak{\zeta})d\mathfrak{\mathcal{V}}_{x,s}(\mathfrak{\zeta})dxd\beta_k(s)\notag\\
		&\qquad+\frac{1}{2}\int_0^t\int_{\mathbb{T}^N}\int_{\mathbb{R}}H^2(x,\mathfrak{\zeta})\partial_{\mathfrak{\zeta}}\varphi(x,\mathfrak{\zeta})d\mathfrak{\mathcal{V}}_{x,s}(\mathfrak{\zeta})dx ds-\mathfrak{m}(\partial_{\mathfrak{\zeta}}\varphi)([0,t)).
	\end{align}
	In particular, we obtain
	\begin{align}\label{3.5}\langle f(t)-f^{-}(t), \varphi\rangle=-\mathfrak{m}(\partial_\mathfrak{\zeta} \varphi)(\{t\}).
	\end{align}
	Outside the set of atomic points (at most countable) of $A\mapsto \mathfrak{m}(\partial_{\mathfrak{\zeta}}\varphi)(A)$, we have$\langle f(t),\varphi\rangle=\langle f^{-}(t),\varphi\rangle.$ It implies that $\mathbb{P}$ almost surely, $f(t)=f^{-}(t)$ almost every $t\in[0,T]$. It is clear that equation \eqref{3.4} gives us the following equation for $f^{-}(t)$: $\mathbb{P}$-almost surely, for all $t\in[0,T]$
	\begin{align}\label{3.6}
		\langle f^{-}(t),\varphi \rangle &= \langle f(0),\varphi\rangle +\int_0^t\langle f^{-}(s),F'(\mathfrak{\zeta})\cdot\nabla_x\varphi\rangle ds-\int_0^t\langle f^-(s),A'(\mathfrak{\zeta})(-\Delta)^\alpha[\varphi]\rangle ds\notag\\&\qquad+\sum_{k\ge 1}\int_0^t \int_{\mathbb{T}^N}\int_{\mathbb{R}}h_k(x,\mathfrak{\zeta})\varphi(x,\mathfrak{\zeta})d\mathfrak{\mathcal{V}}_{x,s}^{-}(\mathfrak{\zeta})dxd\beta_k(s)\notag\\
		&\qquad+\frac{1}{2}\int_0^t\int_{\mathbb{T}^N}\int_{\mathbb{R}}H^2(x,\mathfrak{\zeta})\partial_{\mathfrak{\zeta}}\varphi(x,\mathfrak{\zeta})d\mathfrak{\mathcal{V}}_{x,s}^{-}(\mathfrak{\zeta})dx ds-\mathfrak{m}(\partial_{\mathfrak{\zeta}}\varphi)([0,t)).
	\end{align}
\end{remark}
\begin{remark}
	We will use the following notation: If $f:X\times\mathbb{R}\to[0,1]$ is kinetic function, we define the conjugate function $\bar{f}$ of $f$ as $\bar{f}(x,t,\zeta)=1-f(x,t,\zeta)$. We also define $f^+$ by $f^+:=f$. We can take any of them in integral with respect to time or in a stochastic integral. 
\end{remark}
\noindent
\textbf{Doubling of variables}: The proof of the contraction principle follows a line argument that suitably Kruzkov's method of doubling the variables to the stochastic case. In this context, we approximate $\|\big(u_1(t)-u_2(t)\big)^+\|_{L^1(\mathbb{T}^N)}$ by $\int_{(\mathbb{T}^N)^2}\int_{\mathbb{R}^2}\Upsilon(x-y)\kappa(\xi-\mathfrak{\zeta})f_1(x,t,\xi)\bar {f}_2(y,t,\mathfrak{\zeta})d\xi d\mathfrak{\zeta} dx dy$, where test functions is a suitable smooth approximation. The main idea of the proof is to analyze the value of $\|u_1(t)-u_2(t)\|_{L^1(\mathbb{T}^N)}$ as a random variable, and then  we arrive at the conclusion that $\mathbb{E}\|u_1(t)-u_2(t)\|_{L^1(\mathbb{T}^N)}$ is decreasing fucntion of time. Here, we follow similar lines as proposed in the proof of  \cite[Proposition 3.1]{sylvain} for the proof of following proposition. We give the details of the proof for the sake of completeness.

\begin{proposition}[\textbf{Doubling of variables}]\label{th3.5}
	Let $(u_1(t))_{t\in[0,T]}$ and $(u_2(t))_{t\in[0,T]}$ be solution solution to \eqref{1.1} with initial data $u_{1,0}$ and $u_{2,0}$ respectively. Let us denote $f_1(t)=\mathbbm{1}_{u_1(t)\textgreater\,\xi},\,\& \, f_2(t)= \mathbbm{1}_{u_2(t)\,\textgreater\,\xi}$. Then, for all \,\,$ 0\le t\leq T$ and non-negative test functions $\Upsilon\in{C}^\infty(\mathbb{T}^N)$,\, $\kappa\in{C_{c}^\infty(\mathbb{R})}$ we have 
	\begin{align}\label{tech}
	\mathbb{E}\bigg[\int_{(\mathbb{T}^N)^2}&\int_{\mathbb{R}^2}\Upsilon(x-y)\kappa(\xi-\mathfrak{\zeta})f_1^{\pm}(x,t,\xi)\bar {f}_2^{\pm}(y,t,\mathfrak{\zeta})d\xi d\mathfrak{\zeta} dx dy\bigg]\notag\\
	&\le\mathbb{E}\bigg[ \int_{(\mathbb{T}^N)}\int_{\mathbb{R}^2}\Upsilon(x-y)\kappa(\xi-\mathfrak{\zeta}) f_{1,0} (x,\xi)\bar f_{2,0} (y,\mathfrak{\zeta})d\xi d\mathfrak{\zeta} dx dy+ \mathcal{E}_{\Upsilon}+ \mathcal{E}_{\kappa}+J\bigg], 
	\end{align}
	where
	\begin{align}
	\mathcal{E}_{\Upsilon}\notag=\int_0^t\int_{(\mathbb{T}^N)^2}\int_{\mathbb{R}^2}f_1(x,s,\xi)\bar{f}_2(y,s,\mathfrak{\zeta})(F'(\xi)-F'(\mathfrak{\zeta}))&\kappa(\xi-\mathfrak{\zeta})d\xi d\mathfrak{\zeta}\cdot\nabla\Upsilon(x-y)dx dy ds,\notag
	\end{align}
	\begin{align}
	\mathcal{E}_{\kappa}\notag=\frac{1}{2}\int_{(\mathbb{T}^N)^2}\Upsilon(x-y)\int_0^t\int_{\mathbb{R}^2}\kappa(\xi-\mathfrak{\zeta})\sum_{k\ge1}|h_k(x,\xi)-&h_k(y,\mathfrak{\zeta})|^2d\mathfrak{\mathcal{V}}_{x,s}^{1}\oplus\mathfrak{\mathcal{V}}_{y,s}^{2}(\xi,\mathfrak{\zeta})dx dy ds,\notag
	\end{align}	
	and
	\begin{align}
	J&=-\int_0^t\int_{(\mathbb{T}^N)^2}\int_{\mathbb{R}^2}f_1 (x,s,\xi)\bar {f}_2(y,s,\mathfrak{\zeta})\kappa(\xi-\mathfrak{\zeta})A'(\xi)(-\Delta)_x^\alpha(\Upsilon(x-y)))d\xi d\mathfrak{\zeta} dx dy ds\notag\\
	&\qquad-\int_0^t\int_{(\mathbb{T}^N)^2}\int_{\mathbb{R}^2}f_1 (x,s,\xi)\bar {f}_2(y,s,\mathfrak{\zeta})\kappa(\xi-\mathfrak{\zeta})A'(\mathfrak{\zeta})(-\Delta)_y^\alpha(\Upsilon(x-y)))d\xi d\mathfrak{\zeta} dx dy ds\notag\\
	&\qquad-\int_0^t\int_{(\mathbb{T}^N)^2}\int_{\mathbb{R}^2}f_1(x,s,\xi)\partial_{\xi}\kappa(\xi-\mathfrak{\zeta})\Upsilon(x-y)d\mathfrak{\mathfrak{\eta}}_{2,2}(y,s,\mathfrak{\zeta})dxd\xi \notag\\
	&\qquad+\int_0^t\int_{(\mathbb{T}^N)^2}\int_{\mathbb{R}^2}\bar{f}_2(y,s,\mathfrak{\zeta})\partial_{\mathfrak{\zeta}}\kappa(\xi-\mathfrak{\zeta})\Upsilon(x-y)d\mathfrak{\mathfrak{\eta}}_{1,2}(x,s,\xi)dy d\mathfrak{\zeta}.\notag
	\end{align}
\end{proposition}
\begin{remark}
	Here, we  need to fix some notations for kinetic measures corresponding to kinetic solutions $(u_1(t))_{t\in[0,T]}$ and $(u_2(t))_{t\in[0,T]}$. Suppose that $\mathfrak{m}_1$ and $\mathfrak{m}_2$ are kinetic measures for $(u_1(t))_{t\in[0,T]}$ and $(u_2(t))_{t\in[0,T]}$ respectively, satisfying $\mathbb{P}$-almost surely, $\mathfrak{m}_1 \ge \mathfrak{\mathfrak{\eta}}_{1,2}$ and $\mathfrak{m}_2\ge \mathfrak{\eta}_{2,2}$, where 
	$$\mathfrak{\eta}_{1,2}(x,t,\xi)=\int_{\mathbb{R}^N}|A(u_1(x+z))-A(\xi)|\mathbbm{1}_{\mbox{Conv}\{u_1(x,t),u_1(x+z)\}}(\xi)\mu(z)dz$$
	and 
	$$\mathfrak{\eta}_{2,2}(y,t,\mathfrak{\zeta})=\int_{\mathbb{R}^N}|A(u_2(y+z))-A(\mathfrak{\zeta})|\mathbbm{1}_{\mbox{Conv}\{u_2(x,t),u_2(x+z)\}}(\mathfrak{\zeta})\mu(z)dz.$$
	 These measures can be  written as $\mathfrak{m}_1=\mathfrak{m}_{1,1}+\mathfrak{\eta}_{1,2}$ and $\mathfrak{m}_2=\mathfrak{m}_{2,1}+\mathfrak{\eta}_{2,2}$ for some non negative measure $\mathfrak{m}_{1,1}$ and $\mathfrak{m}_{2,1}$ respectively.
	\end{remark}
\begin{proof}
	 Le us define $H_1^2(x,\xi)= \sum_{k\ge1}|h_k(x,\xi)|^2$ and $H_2^2(y,\mathfrak{\zeta})=\sum_{k\ge1}|h_k(y,\mathfrak{\zeta})|^2$. Let $\varphi_1 \in C_c^\infty(\mathbb{T}_x^N\times\mathbb{R}_{\xi})$ and $\varphi_2\in C_c^\infty(\mathbb{T}_y^N\times\mathbb{R}_{\mathfrak{\zeta}})$.  By equation \eqref{2.10} for $f_1=f_1^+$ we have
	 $$\langle f_1^+(t),\varphi \rangle = \langle \Lambda_1^*, \partial_\xi \varphi_1 \rangle+\mathcal{M}_1(t)$$
	 with
	 $$\mathcal{M}_1(t)=\sum_{k\ge1} \int_0^t\int_{\mathbb{T}^N}\int_{\mathbb{R}} h_k(x,\xi)\varphi_1(x,\xi) d\mathfrak{\mathcal{V}}_{x,s}^1(\xi) dx d\beta_k (s)$$
	 and 
	 \begin{align*}
	 \langle \Lambda_1^*, \varphi _1 \rangle ([0,t])& =\langle f_{1,0}, ,\varphi_1 \rangle \delta_0([0,t])+\int_0^t \langle f_1, F'\cdot\nabla \varphi_1 \rangle ds-\int_0^t \langle f_1 , A'(\xi)(-\Delta)_x^\alpha[\varphi_1]\rangle ds\\
	 &\qquad+\frac{1}{2} \int_0^t \int_{\mathbb{T}^N}\int_{\mathbb{R}} \partial_{\xi} \varphi_1 H_1^2(x,\xi) d\mathfrak{\mathcal{V}}_{x,s}^1(\xi) dx ds -\mathfrak{m}_1(\partial_{\xi} \varphi_1)([0,t]).
	 \end{align*}
	 Notice that $\mathfrak{m}_1(\partial_{\xi}\varphi_1)(\{0\})=0$ and the value of $\langle \Lambda_1^* ,\partial_{\xi} \varphi_1\rangle (\{0\})$ is $\langle f_{1,0}, \varphi_1 \rangle.$ Similarly 
	 $$\langle \bar{f}_2^+(t),\varphi_2 \rangle=\langle \bar{\Lambda}_2^*, \partial_{\mathfrak{\zeta}} \varphi_2 \rangle([0,t])+\bar{\mathcal{M}}_2(t)$$
	 with 
	 $$\bar{\mathcal{M}}_2(t)=\sum_{k\ge1} \int_0^t\int_{\mathbb{T}^N}\int_{\mathbb{R}} h_k(y,\mathfrak{\zeta})\varphi_2(y,\mathfrak{\zeta}) d\mathfrak{\mathcal{V}}_{y,s}^1(\mathfrak{\zeta}) dx d\beta_k (s)$$
	 and
	 \begin{align*}
	 \langle \bar{\Lambda}_2^*, \varphi _2 \rangle ([0,t])& =\langle \bar{f}_{2,0}, ,\varphi_1 \rangle \delta_0([0,t])+\int_0^t \langle\bar{ f}_2, F'\cdot\nabla \varphi_2 \rangle ds-\int_0^t \langle \bar{f}_2 , A'(\mathfrak{\zeta}) (-\Delta)_y^\alpha[\varphi_2]\rangle ds\\
	 &\qquad-\frac{1}{2} \int_0^t \int_{\mathbb{T}^N}\int_{\mathbb{R}} \partial_{\mathfrak{\zeta}} \varphi_2 H_2^2 d\mathfrak{\mathcal{V}}_{y,s}^2(\mathfrak{\zeta}) dx ds +\mathfrak{m}_2(\partial_{\mathfrak{\zeta}} \varphi_2)([0,t]),	
	 \end{align*}
	 where $\langle \bar{\Lambda}_2^* ,\partial_{\mathfrak{\zeta}}\varphi_2 \rangle (\{0\})=\langle \bar{f}_2, \varphi_2 \rangle$. Let $\theta(x,\xi, y, \mathfrak{\zeta})= \varphi_1(x,\xi)\varphi_2(y,\mathfrak{\zeta}).$ Making use of It\^o formula for $\mathcal{M}_1(t)\bar{\mathcal{M}}_2(t)$, and integration by parts formula for functions of finite variation, $\langle \Lambda_1^* , \partial_{\xi} \varphi_1 \rangle\,\langle \bar{\Lambda}_2^*,\partial_{\mathfrak{\zeta}} \varphi_2 \rangle ([0,t])$, (see \cite[Chapter 0]{Rev}), we conclude that
	 \begin{align*}
	\langle \Lambda_1^*, \partial_{\xi} \varphi_1 ([0,t])\rangle\,&\langle \bar{\Lambda}_2^* ,\partial_{\mathfrak{\zeta}}\varphi_2 \rangle ([0,t])=\langle \Lambda_1^* , \partial_{\xi} \varphi_1 \rangle (\{0\})\,\langle \bar{\Lambda}_2^*, \partial_{\mathfrak{\zeta}} \varphi_2 \rangle (\{0\})\\&+\int_{(0,t]} \langle \mathfrak{m}_1^* ,\partial_{\xi} \varphi_1 \rangle([0,s)) d \langle \bar{\Lambda}_2^*, \partial_{\mathfrak{\zeta}}\varphi_2 \rangle (s) +\int_{(0,t]}\langle \bar{\Lambda}_2^* , \partial_{\mathfrak{\zeta}} \varphi_2 ([0,s]) d\langle \Lambda_1^*, \partial_{\xi} \varphi_1 \rangle (s).
	 \end{align*}
	 We also have the following identity
	 \begin{align*}
	 \langle \Lambda_1^* , \partial_{\xi} \varphi_1 \rangle ([0,t]) \bar{\mathcal{M}}_2(t)=\int_0^t \langle \Lambda_1^*, \partial_{\xi} \varphi_1\rangle ([0,s]) d\bar{\mathcal{M}}_2(s) + \int_0^t \bar{\mathcal{M}}_2(s) \langle \Lambda_1^* , \partial_{\xi} \varphi_1 \rangle (ds).
	 \end{align*}
	 It is easy to obtain since $\bar{\mathcal{M}}_2$ is continuous and a similar formula for $\langle \bar{\Lambda}_2^* ,\partial_{\mathfrak{\zeta}}\varphi_2 \rangle {\mathcal{M}}_1(t)$. These identities give 
	 $$\langle f_1^+(t), \varphi_1 \rangle\,\langle \bar{f}_2^+(t),\varphi_2 \rangle =\langle \langle f_1^{+}(t)\,\bar{f}_2^{+}(t), \alpha \rangle \rangle$$ 
	 It implies that
	  \begin{align}\label{un}
	  \mathbb{E}\langle \langle f_1^+(t)\,\bar{f}_2^+(t), \theta \rangle\rangle&= \mathbb{E}\langle \langle f_{1,0} \bar{f}_{2,0}, \theta\rangle\rangle\notag\\
	  &+\mathbb{E}\int_0^t\int_{(\mathbb{T}^N)^2}\int_{\mathbb{R}^2} f_1 \bar{f}_2 (F'(\xi)\cdot\nabla_x + F'(\mathfrak{\zeta})\cdot\nabla_y)\theta d\xi d\mathfrak{\zeta} dx dy ds\notag\\
	  &-\mathbb{E}\int_0^t \int_{(\mathbb{T}^N)^2}\int_{\mathbb{R}^2} f_1 \bar{f}_2(A'(\xi)(-\Delta)_x^\alpha+A'(\mathfrak{\zeta}) (-\Delta)_y^\alpha) [\theta]\, d\xi d\mathfrak{\zeta} dx dy ds\notag\\
	  &+\frac{1}{2}\mathbb{E}\int_0^t\int_{(\mathbb{T}^N)^2}\int_{\mathbb{R}^2}\partial_{\xi} \theta \bar{f}_2(s) H_1^2\,d\mathfrak{\mathcal{V}}_{x,s}^1(\xi) d\mathfrak{\zeta} dx dy ds\notag\\
	  &-\frac{1}{2} \mathbb{E} \int_0^t \int_{(\mathbb{T}^N)^2}\int_{\mathbb{R}^2}\partial_{\mathfrak{\zeta}} \theta f_1(s) H_2^2\, d\mathfrak{\mathcal{V}}_{y,s}^2(\mathfrak{\zeta})\,d\xi dy dx ds\notag\\
	  &-\mathbb{E}\int_0^t \int_{(\mathbb{T}^N)^2}\int_{\mathbb{R}^2} H_{1,2}\theta\,d\mathfrak{\mathcal{V}}_{x,s}^1(\xi)\,d\mathfrak{\mathcal{V}}_{y,s}^2(\mathfrak{\zeta}) dx dy ds\notag\\
	  &-\mathbb{E}\int_{(0,t]}\int_{(\mathbb{T}^N)^2}\int_{\mathbb{R}^2} \bar{ f}_2^{+}(s)\partial_{\xi} \theta d\mathfrak{m}_1(x,s,\xi) d\mathfrak{\zeta} dy\notag\\
	  &+\mathbb{E}\int_{(0,t]}\int_{(\mathbb{T}^N)^2}\int_{\mathbb{R}^2}f_1^-(s) \partial_{\mathfrak{\zeta}} \theta  d\mathfrak{m}_2(y,s,\mathfrak{\zeta}) d\xi dx	  
	  \end{align}
	  where $H_{1,2}(x,y;\xi,\mathfrak{\zeta}):=\sum_{k\ge1 } h_k(x,\xi)h_k(y,\mathfrak{\zeta})$ and $\langle\langle\cdot,\cdot \rangle\rangle$ refer to the duality distribution over $\mathbb{T}_x^N\times\mathbb{R}_{\xi}\times\mathbb{T}_y^N\times\mathbb{R}_{\mathfrak{\zeta}}$. By a density argument of approximation, we can easily prove that \eqref{un} is valid for any test function $\theta \in C_c^\infty(\mathbb{T}_x^N\times\mathbb{R}_{\xi}\times\mathbb{T}_y^N\times\mathbb{R}_{\mathfrak{\zeta}})$. The assumption that $\theta$ is compactly supported can be relaxed by the help of the condition at infinity on $\mathfrak{m}_j$ and $\mathfrak{\mathcal{V}}^j$, $j=1,2$. By help of truncation and approximation argument for $\theta$, we can prove that \eqref{un} is also valid if $\theta \in C_b^\infty (\mathbb{T}_x^N\times\mathbb{R}_{\xi}\times\mathbb{T}_y^N\times\mathbb{R}_{\mathfrak{\zeta}})$ is compactly supported in a neighbourhood of the set
	  $\big\{(x,\mathfrak{\zeta},x,\mathfrak{\zeta}); x\in\mathbb{T}^N, \mathfrak{\zeta} \in\mathbb{R}\big\}$,
	  then we can take $\theta(x,\xi, y,\mathfrak{\zeta})=\Upsilon(x-y) \kappa(\xi-\mathfrak{\zeta})$ where $\Upsilon\in C^{\infty}(\mathbb{T}^N), \kappa\in C^{\infty}_c(\mathbb{R})$.
	  the following identities,
	  $(\nabla_x+\nabla_y)\theta=0,\,\,\,\,\, (\partial_{\xi}+\partial_\mathfrak{\zeta})\theta=0,$ gives that
	  \begin{align*}
	  &\mathbbm{E}\bigg[\int_{(\mathbbm{T}^N)^2}\int_{\mathbb{R}^2}\Upsilon(x-y)\kappa(\xi-\mathfrak{\zeta})f_1^{+}(x,s,\xi)\bar{f}_2^{+}(y,t,\mathfrak{\zeta})d\xi d\mathfrak{\zeta} dx dy\bigg]\\
	  &=\mathbb{E}\bigg[\int_{(\mathbb{T}^N)^2}\int_{\mathbb{R}^2}\Upsilon(x-y)\kappa(\xi-\mathfrak{\zeta}) f_{1,0}(x,\xi)\bar{f}_{2,0}(y,\mathfrak{\zeta}) d\xi d\mathfrak{\zeta} dx dy+J+K+\mathcal{E}_{\Upsilon}+\mathcal{E}_{\kappa}\bigg],
	  \end{align*} 	  
	  where
	  \begin{align*}
	  K&=\int_{(0,t]}\int_{\mathbb{T}^N}\int_{\mathbb{R}}f_1^{-}(x,s,\xi)\partial_\mathfrak{\zeta}\theta\,\,\, dm_{2,1}(y,s,\mathfrak{\zeta})-\int_{(0,t]}\int_{\mathbb{T}^N}\int_{\mathbb{R}}\bar{f}_2^{+}(y,s,\mathfrak{\zeta})\,\,\partial_\xi \theta\,\,\, d\mathfrak{m}_{1,1}(x,s,\xi)\\
	  &=-\int_{(0,t]}\int_{\mathbb{T}^N}\int_{\mathbb{R}}f^{-}_1(x,s,\xi)\partial_{\xi} \theta d\mathfrak{m}_{2,1}(y,s,\mathfrak{\zeta})d\xi dx-\int_{(0,t]}\int_{\mathbb{T}^N}\int_{\mathbb{R}}f_2(y,s,\mathfrak{\zeta})\partial_{\mathfrak{\zeta}} \theta d\mathfrak{m}_{1,1}(x,s,\xi)dy d\mathfrak{\zeta}\\
	  &=-\int_{(0,t]}\int_{\mathbb{T}^N}\int_{\mathbb{R}}\theta d\mathfrak{\mathcal{V}}_{x,s}^{1,-}(\xi) d\mathfrak{m}_{2,1}(y,s,\mathfrak{\zeta}) dx-\int_{(0,t]}\int_{\mathbb{T}^N}\int_{\mathbb{R}}\theta d\mathfrak{\mathcal{V}}_{y,s}^{2,+}(\mathfrak{\zeta}) d\mathfrak{m}_{1,1}(x,s,\xi)dy\\
	  &\le\,0.
	  \end{align*}
	  Consequently we deduce the following estimate, $\mathbb{P}$-almost surely, for all $t\in[0,T]$
	  \begin{align}
	  \mathbb{E}\bigg[\int_{(\mathbb{T}^N)^2}\int_{(\mathbb{R})^2}&\Upsilon(x-y)\kappa(\xi-\mathfrak{\zeta})f_1^{\pm}(x,t,\xi)\bar {f}_2^{\pm}(y,t,\mathfrak{\zeta})d\xi d\mathfrak{\zeta} dx dy\bigg]\notag\\
	 &\qquad\leq\mathbb{E}\bigg[\int_{(\mathbb{T}^N)}\int_{(\mathbb{R}^2)}\Upsilon(x-y)\kappa(\xi-\mathfrak{\zeta})f_{1,0}(x,\xi)\bar f_{2,0} (y,\mathfrak{\zeta})d\xi d\mathfrak{\zeta} dx dy+ \mathcal{E}_{\Upsilon}+\mathcal{E}_{\kappa}+J\bigg].
	  \end{align}
 \end{proof}
\begin{remark}
	We can easily conclude that, if $f_1^\pm=f_2^\pm$, then inequality \eqref{tech} holds pathwise, i.e.,
	\begin{align}\label{app inequality1}
		\begin{aligned}
			\int_{(\mathbb{T}^N)^2}&\int_{\mathbb{R}^2}\Upsilon(x-y)\kappa(\xi-\mathfrak{\zeta})f_1^{\pm}(x,t,\xi)\bar {f}_2^{\pm}(y,t,\mathfrak{\zeta})d\xi d\mathfrak{\zeta} dx dy\\
			&\le\int_{(\mathbb{T}^N)}\int_{\mathbb{R}^2}\Upsilon(x-y)\kappa(\xi-\mathfrak{\zeta}) f_{1,0} (x,\xi)\bar f_{2,0} (y,\mathfrak{\zeta})d\xi d\mathfrak{\zeta} dx dy+ \mathcal{E}_{\Upsilon}+ \mathcal{E}_{\kappa}+J,
		\end{aligned}
	\end{align}
for all $t\in[0,T]$, $\mathbb{P}$-almost surely, where $\mathcal{E}_{\Upsilon}, \mathcal{E}_{\kappa},$ and $J$ are introduced as in previous Proposition \ref{th3.5}. For the proof of it, there is no need to take expectation after use of It\^o and integration by part formula in the proof of Proposition \ref{th3.5}, since after the approximation of test functions $\Upsilon(x-y)$, and $\kappa(\xi-\mathfrak{\zeta})$, the contributed martingale terms cancel to each other. 
\end{remark}
\begin{thm}[\textbf{Contraction principle}]\label{th3.6}\label{comparison} Let $(u(t))_{t\in[0,T]}$ be a kinetic solution to \eqref{1.1}, then there exists a $L^1(\mathbb{T}^N)$-valued process $(u^{-}(t))_{t\in[0,T]}$ such that $\mathbb{P}$-almost surely, for all $t\in[0,T]$, $f^{-}(t)=\mathbbm{1}_{u^{-}(t)\textgreater\xi}$. Moreover, if $(u_1(t))_{t\in[0,T]}$,\,\,$(u_2(t))_{t\in[0,T]}$\,are kinetic solutions to \eqref{1.1} with initial data $u_{1,0} $ and $u_{2,0}$ respectively, then we have for all $t\in[0,T]$,
	\begin{align}\label{contraction}\mathbb{E}\left\Vert u_1(t)-u_2(t)\right\Vert_{L^{1}(\mathbb{T}^N)}\le\mathbb{E}\left\Vert u_{1,0}-u_{2,0}\right\Vert_{\mathbb{L}^1(\mathbb{T}^N)}.
	\end{align}
\end{thm}
\begin{proof} We will first show contraction principle \eqref{contraction} in the following several steps.
	
	\noindent
	\textbf{Step 1:}
	Suppose that $(\Upsilon_{\varepsilon}),\,(\kappa_{\delta})$ are the approximations to the identity on $\mathbb{T}^N$ and $\mathbb{R}$, respectively, that is, $\Upsilon\in C^\infty(\mathbb{T}^N)$ and $\kappa\in C_c^\infty (\mathbb{R})$ be symmetric nonnegative functions such that $\int_{\mathbb{T}^N} \Upsilon(x) dx = 1 ,\,\, \int_{\mathbb{R}} \kappa(\xi) d\xi =1$ and supp$\kappa\subset(-1,1)$ with
	$\Upsilon_{\varepsilon}=\frac{1}{\varepsilon^N}\Upsilon(\frac{x}{\varepsilon}),$ and
	$\kappa_{\delta}(\xi)=\frac{1}{\delta}\kappa(\frac{\xi}{\delta}).$
	Let put 
	$\kappa=\kappa_{\delta}$ ,  and $\Upsilon= \Upsilon_\varepsilon $ in inequality \eqref{tech}. Then, we follow \cite[Equation 3.16]{sylvain} to conclude for all $t\in[0,T]$
	\begin{align}
		&\mathbb{E}\int_{\mathbb{T}^N}\int_{\mathbb{R}}f_1(x,t,\xi)\bar{f}_2(x,t,\xi)d\xi dx=\mathbb{E}\int_{(\mathbb{T}^N)^2}\int_{\mathbb{R}^2}\Upsilon_{\varepsilon}(x-y)\kappa_{\delta}(\xi-\mathfrak{\zeta})f_1(x,t,\xi)\bar{f}_2(y,t,\mathfrak{\zeta})d\xi d\mathfrak{\zeta} dxdy+\mathfrak{\eta}_{t}(\varepsilon,\delta)
	\end{align}
	where $ \lim_{\varepsilon,\delta\to 0}\mathfrak{\eta}_t(\varepsilon,\delta)=0.$\\
	\textbf{Step 2:} We follow \cite{deb} to estimate $\mathcal{E}_\Upsilon$ and $\mathcal{E}_\kappa$. We have $\mathbb{P}$-almost surely, for all $t\in[0,T]$
	\begin{align}
		|\mathcal{E}_{\kappa}|&\le  C_{\Psi}\,t\,(\delta^{-1}\varepsilon^2 + g(\delta))\label{estimate 1}\\
		|\mathcal{E}_\Upsilon|&\leq C(\omega) \varepsilon^{-1}\delta,\label{estimate 2}
	\end{align}
	\textbf{Step 3:}
	Let $\mathfrak{\eta}_\delta(\xi)=\delta\,\,\mathfrak{\eta}_(\frac{x}{\delta})$, be convex approximation of absolute value function, where $\mathfrak{\eta}$  be a $C^\infty(R)$ function satisfying,
	$\mathfrak{\eta}(0)=0$,\,\,\,\,\, $\mathfrak{\eta}(r)=\mathfrak{\eta}(-r)$,\,\,\,\, $\mathfrak{\eta}'(r)=-\mathfrak{\eta}'(-r)$, $\mathfrak{\eta}''_\delta=\kappa_\delta,$ and
	\[\mathfrak{\eta}'(r)=\begin{cases}
		-1, & r\, \le\,-1\\
		\in[-1,1] ,& |r|\,\textless\,1\\
		1, & r\,\ge\,1
	\end{cases}
	\]
	In order to estimate $J$, we proceed as follow
	\begin{align*}
		J&=-\int_0^t\int_{(\mathbb{T}^N)^2}\int_{\mathbb{R}^2}f_1(x,s,\xi)\bar{ f}_2(y,s,\mathfrak{\zeta})\kappa_{\delta}(\xi-\mathfrak{\zeta})A'(\xi)(-\Delta)_x^\alpha(\Upsilon_{\varepsilon}(x-y)))d\xi d\mathfrak{\zeta} dx dy \notag\\
		&\quad-\int_0^t\int_{(\mathbb{T}^N)^2}\int_{\mathbb{R}^2}f_1(x,s,\xi)\bar{f}_2(y,s,\mathfrak{\zeta})\kappa_{\delta}(\xi-\mathfrak{\zeta})A'(\mathfrak{\zeta})(-\Delta)_y^\alpha(\Upsilon_{\varepsilon}(x-y)))d\xi d\mathfrak{\zeta} dx dy ds\notag\\
		&\quad+\int_0^t\int_{(\mathbb{T}^N)^2}\int_{\mathbb{R}^2}\partial_\mathfrak{\zeta}\kappa_{\delta}(\xi-\mathfrak{\zeta})\Upsilon_{\varepsilon}(x-y)f_1(x,s,\xi)d\mathfrak{\eta}_{2}(y,s,\mathfrak{\zeta})dx d\xi \notag\\
		&\quad-\int_0^t\int_{(\mathbb{T}^N)^2}\int_{\mathbb{R}^2}\partial_\xi\kappa_{\delta}(\xi-\mathfrak{\zeta})\Upsilon_{\varepsilon}(x-y)\bar{f}_2(y,s,\mathfrak{\zeta})d\mathfrak{\eta}_{1}(x,s,\xi)dy d\mathfrak{\zeta}\\
		&:=J_1+J_2,
	\end{align*}
	where
	\noindent
	\begin{align*}
		J_1&=-\int_0^t\int_{(\mathbb{T}^N)^2}\int_{\mathbb{R}^2}f_1(x,s,\xi)\bar{ f}_2(y,s,\mathfrak{\zeta})\kappa_{\delta}(\xi-\mathfrak{\zeta})A'(\xi)(-\Delta)_x^\alpha(\Upsilon_{\varepsilon}(x-y)))d\xi d\mathfrak{\zeta} dx dy \notag\\
		&\quad-\int_0^t\int_{(\mathbb{T}^N)^2}\int_{\mathbb{R}^2}f_1(x,s,\xi)\bar{ f}_2(y,s,\mathfrak{\zeta})\kappa_{\delta}(\xi-\mathfrak{\zeta})A'(\mathfrak{\zeta})(-\Delta)_y^\alpha(\Upsilon_{\varepsilon}(x-y)))d\xi d\mathfrak{\zeta} dx dy ds,\notag\\
		J_2&=\int_0^t\int_{(\mathbb{T}^N)^2}\int_{\mathbb{R}^2}\partial_\mathfrak{\zeta}\kappa_{\delta}(\xi-\mathfrak{\zeta})\Upsilon_{\varepsilon}(x-y)f_1(x,s,\xi)d\mathfrak{\eta}_{2}(y,s,\mathfrak{\zeta})dx d\xi\\
		&\quad-\int_0^t\int_{(\mathbb{R}^N)^2}\int_{\mathbb{R}^2}\partial_\xi\kappa_{\delta}(\xi-\mathfrak{\zeta})\Upsilon_{\varepsilon}(x-y)\bar{f}_2(y,s,\mathfrak{\zeta})d\mathfrak{\eta}_{1}(x,s,\xi)dy d\mathfrak{\zeta}.
	\end{align*}
	Now we will try to write $J_2$ in terms of $-J_1$ and plus some small term as follows:
	\begin{align*}
		&\int\limits_0^t\int\limits_{(\mathbb{T}^N)^2}\int\limits_{\mathbb{R}^2}\partial_\xi\kappa_{\delta}(\xi-\mathfrak{\zeta})\Upsilon_{\varepsilon}(x-y)\bar{f}_2(y,s,\mathfrak{\zeta})d\mathfrak{\eta}_{1}(x,s,\xi)dy d\mathfrak{\zeta}\\
		&\quad=\int_0^t\int\limits_{({\mathbb{T}^N})^2}\int\limits_{\mathbb{R}^2}\int\limits_{\mathbb{R}^N}|\tau_z A(u_1(x,s))-A(\xi)|\mathbbm{1}_{Con\{u_1(x,s),\tau_z u_1(x,s)\}}(\xi) \partial_{\xi} \kappa_\delta(\xi-\mathfrak{\zeta}) \\
		&\qquad\qquad\qquad\qquad\qquad\qquad\qquad\qquad\qquad\qquad\qquad\Upsilon_{\varepsilon}(x-y)\bar{f}_2(y,s,\mathfrak{\zeta}) \mu(z) dz d\xi d\mathfrak{\zeta} dx dy ds\\
		&\quad=\int\limits_0^t\int\limits_{(\mathbb{T}^N)^2}\int\limits_{\mathbb{R}^{N+1}} \bigg\{\int\limits_{u_1(x,s)}^{\tau_z u_1(x,s)}(\tau_z A (u_1(x,s))-A(\xi)) \partial_{\xi} \kappa_\delta(\xi-\mathfrak{\zeta}) d\xi \\
		&\qquad+\int\limits_{\tau_z u_1(x,s)}^{u_1(x,s)}(A(\xi)-\tau_zA(u_1(x,s)))\partial_{\xi} \kappa_\delta(\xi-\mathfrak{\zeta}) d\xi \bigg\}\Upsilon_\varepsilon(x-y)\bar{f}_2(y,s,\mathfrak{\zeta}) \mu(z) dz d\mathfrak{\zeta} dx dy ds\\
		&\quad=\int\limits_ 0^t\int\limits_{\mathbb{R}^{N+1}}\int\limits_{\mathbb{T}^N}\Bigg\{ \int\limits_{\mathbb{T}^N\,\cap\,\{u_1(x,s)\,\le\,\tau_z u_1(x,s)\} }\bigg\{\big[(\tau_z A (u_1(x,s))- A(\xi))\kappa_\delta(\xi-\mathfrak{\zeta})\big]_{u_1(x,s)}^{\tau_z u_1(x,s)}\\&\qquad+\int\limits_{u_1(x,s)}^{\tau_z u_1(x,s)}\kappa_\delta(\xi-\mathfrak{\zeta})A'(\xi) d\xi \bigg\}+\int\limits_{\mathbb{T}^N\cap\{\tau_zu_1(x,s)\le u_1(x,s)\}}\bigg\{\big[A(\xi)-\tau_zA(u_1(x,s)) \kappa_\delta(\xi-\mathfrak{\zeta})\big]_{\tau_z u_1(x,s)}^{u_1(x,s)}\\&\qquad-\int\limits_{\tau_z u_1(x,s)}^{u_1(x,s)} \kappa_\delta (\xi-\mathfrak{\zeta})A'(\xi) d\xi \bigg\}\Bigg\} \Upsilon_\varepsilon(x-y)\bar{f}_2(y,s,\mathfrak{\zeta}) \mu(z) dz d\mathfrak{\zeta} dx dy ds\\
		&\quad=-\int\limits_0^t\int\limits_{({\mathbb{T}^N})^2}\int\limits_{\mathbb{R}^{N+1}}\bigg\{(\tau_z A(u_1(x,s))- A(u_1(x,s)))\kappa_{\delta}(u_1(x,s)-\mathfrak{\zeta})+\int\limits_{u_1(x,s)}^{\tau_z u_1(x,s)}\kappa_{\delta}(\xi-\mathfrak{\zeta})A'(\xi) d\xi \bigg\}\\
		&\qquad\qquad\qquad\qquad\qquad\qquad\qquad\qquad\qquad\qquad\qquad\qquad\qquad\qquad\Upsilon_\varepsilon(x-y)\bar{f}_2(y,s,\mathfrak{\zeta}) \mu(z) dz d\mathfrak{\zeta} dx dy ds\\
		&\quad=-\int\limits_0^t\int_{(\mathbb{T}^N)^2}\int\limits_{\mathbb{R}^{N+1}}(\tau_z A(u_1(x,s))-A(u_1(x,s)))\kappa_{\delta}(u_1(x,s)-\mathfrak{\zeta})\Upsilon_\varepsilon(x-y)\bar{ f}_2(y,s,\mathfrak{\zeta})\mu(z) dz d\mathfrak{\zeta} dx dy ds\\
		&\quad\quad+\int\limits_0^t\int\limits_{(\mathbb{T}^N)^2}\int\limits_{\mathbb{R}^{N+1}}\int\limits_{\mathbb{R}}\bigg\{\int\limits_{-\infty}^{\tau_z u_1(x,s)}\kappa_{\delta}(\xi-\mathfrak{\zeta})A'(\xi)d\xi -\int_{-\infty}^{u_1(x,s)}\kappa_{\delta}(\xi-\mathfrak{\zeta})A'(\xi) d\xi \bigg\}\\
		&\qquad\qquad\qquad\qquad\qquad\qquad\qquad\qquad\qquad\qquad\qquad\qquad\qquad\qquad\Upsilon_\varepsilon(x-y)\bar{ f}_2(y,s,\mathfrak{\zeta}) \mu(z) dz d\mathfrak{\zeta} dx dy ds\\
		&\quad=-\int\limits_0^t\int\limits_{(\mathbb{T}^N)^2}\int\limits_{\mathbb{R}^N}\int\limits_{\mathbb{R}}(\tau_z A(u_1(x,s))-A(u_1(x,s)))\kappa_{\delta}(u_1(x,s)-\mathfrak{\zeta})\Upsilon_\varepsilon(x-y)\bar{ f}_2(y,s,\mathfrak{\zeta})\mu(z) dz d\mathfrak{\zeta} dx dy ds\\
		&\qquad+\int\limits_0^t\int\limits_{(\mathbb{T}^N)^2}\int\limits_{\mathbb{R}^N}\int\limits_{\mathbb{R}}\bigg(\int_{-\infty}^{ u_1(x,s)}\kappa_{\delta}(\xi-\mathfrak{\zeta})A'(\xi)d\xi\bigg) (\tau_z(\Upsilon_\varepsilon(x-y)-\Upsilon_\varepsilon(x-y)\bar{ f}_2(y,s,\mathfrak{\zeta}) \mu(z) dz d\mathfrak{\zeta} dx dy ds\\
		&\quad=-\int\limits_0^t\int\limits_{(\mathbb{T}^N)^2}\int\limits_{\mathbb{R}^N}\int\limits_{\mathbb{R}}(\tau_z A(u_1(x,s))-A(u_1(x,s)))\kappa_{\delta}(u_1(x,s)-\mathfrak{\zeta})\Upsilon_\varepsilon(x-y)\bar{ f}_2(y,s,\mathfrak{\zeta})\mu(z) dz d\mathfrak{\zeta} dx dy ds\\
		&\qquad-\int\limits_0^t\int\limits_{(\mathbb{T}^N)^2}\int\limits_{\mathbb{R}^2} f_1(x,s,\xi)\kappa_{\delta}(\xi-\mathfrak{\zeta})A'(\xi) (-\Delta)_x^\alpha(\Upsilon_\varepsilon)(x-y) \bar{ f}_2(y,s,\mathfrak{\zeta})\mu(z) dz d\xi d\mathfrak{\zeta} dx dy ds.
	\end{align*}
	Similarly, we can compute for remaining part of $J_2$ as follows:
	\begin{align*}
		&\int\limits_0^t\int\limits_{({\mathbb{T}^N})^2}\int\limits_{\mathbb{R}^2} \partial_{\mathfrak{\zeta}} \kappa_{\delta}(\xi-\mathfrak{\zeta}) \Upsilon_\varepsilon(x-y)f_1(x,s,\xi) d\mathfrak{\eta}_{2}(y,s,\mathfrak{\zeta}) dx d\xi\\
		&\qquad=-\int\limits_0^t\int\limits_{(\mathbb{T}^N)^2}\int\limits_{\mathbb{R}^N}\int\limits_{\mathbb{R}}(\tau_z A(u_2(y,s))-A(u_2(y,s)))\kappa_{\delta}(\xi-u_2(x,s))\Upsilon_\varepsilon(x-y) f_1(x,s,\xi)\mu(z) dz d\xi dx dy ds\\
		&\qquad\quad+\int\limits_0^t\int\limits_{(\mathbb{T}^N)^2}\int\limits_{\mathbb{R}^2} f_1(x,s,\xi)\kappa_{\delta}(\xi-\mathfrak{\zeta}) (-\Delta)_x^\alpha(\Upsilon_\varepsilon(x-y)) \bar{ f}_2(y,s,\mathfrak{\zeta}) \mu(z) dz d\xi d\mathfrak{\zeta} dx dy ds.
	\end{align*}
	\begin{align*}
		&J_2=-\int_0^t\int\limits_{(\mathbb{T}^N)^2}\int\limits_{\mathbb{R}}\int\limits_{\mathbb{R}^N}\big(\tau_z A(u_2(y,s))-A(u_2(y,s)))\big)(\kappa_{\delta} (\xi-u_2(y,s))\Upsilon_{\varepsilon}(x-y)) f_1(x,s,\xi)\mu(z)dz dx dy d\xi ds\notag\\
		&\qquad\qquad+\int_0^t\int\limits_{(\mathbb{T}^N)^2}\int\limits_{\mathbb{R}^2}\bar{f}_2(y,s,\mathfrak{\zeta}) \kappa_{\delta}(\xi-\mathfrak{\zeta})(-\Delta)_x^\alpha(\Upsilon_{\varepsilon}(x-y))f_1(x,s,\xi)\mu(z)dz dx dy d\mathfrak{\zeta} d\xi ds
		\\&\qquad\qquad+\int\limits_0^t\int\limits_{(\mathbb{T}^N)^2}\int\limits_{\mathbb{R}}\int\limits_{\mathbb{R}^N}\big(\tau_z A(u_1(x,t))-A(u_1(x,s)))\big)\kappa_{\delta} (u_1(x,s)-\mathfrak{\zeta})\Upsilon_{\varepsilon}(x-y)\bar{f}_2(y,t,\mathfrak{\zeta})\mu(z)dz dx dy d\mathfrak{\zeta}\notag\\
		&\qquad\qquad+\int\limits_0^t\int\limits_{(\mathbb{T}^N)^2}\int\limits_{\mathbb{R}}\int\limits_{\mathbb{R}^N}f_1(x,s,\xi)\kappa_{\delta}(\xi-\mathfrak{\zeta})A'(\mathfrak{\zeta}) (-\Delta)_y^\alpha(\Upsilon_\varepsilon(x-y))\bar{f}_2(y,s,\mathfrak{\zeta})\mu(z)dz dx dy d\mathfrak{\zeta} d\xi ds \\
		&:=-J_1+I_1
	\end{align*}
	where
	\begin{align*}
		I_1&=-\int\limits_0^t\int\limits_{(\mathbb{T}^N)^2}\int\limits_{\mathbb{R}}\int\limits_{\mathbb{R}^N}\big(\tau_z A(u_2(y,s))-A(u_2(y,s)))\big)(\kappa_{\delta} (\xi-u_2(y,s))\Upsilon_{\varepsilon}(x-y)) f_1(x,s,\xi)\mu(z)dz dx dy d\xi ds\notag\\
		&+\int\limits_0^t\int\limits_{(\mathbb{T}^N)^2}\int\limits_{\mathbb{R}}\int\limits_{\mathbb{R}^N}\big(\tau_z A(u_1(x,t))-A(u_1(x,s)))\big)\kappa_{\delta} (u_1(x,s)-\mathfrak{\zeta})\Upsilon_{\varepsilon}(x-y)\bar{f}_2(y,t,\mathfrak{\zeta})\mu(z)dz dx dy d\mathfrak{\zeta} ds\notag\\
		&=\int\limits_0^t\int\limits_{(\mathbb{T}^N)^2}\int\limits_{\mathbb{R}}\int_{\mathbb{R}^N}(\tau_z A(u_1(x,s))-A(u_1(x,s)))\kappa_{\delta}(u_1(x,s)-\mathfrak{\zeta})\Upsilon_{\varepsilon}(x-y)\bar{f}_2(y,s,\mathfrak{\zeta})\mu(z)dz d\mathfrak{\zeta} dx dy ds\notag\\
		&-\int\limits_0^t\int\limits_{(\mathbb{T}^N)^2}\int\limits_{\mathbb{R}}\int\limits_{\mathbb{R}^N}(\tau_z A(u_2(y,s))-A(u_2(y,s)))\kappa_{\delta}(\xi-u_2(y,s))\Upsilon_{\varepsilon}(x-y f_1(x,s,\xi)\mu(z) dz d\xi dx dy ds\notag\\
		&=\int\limits_0^t\int\limits_{(\mathbb{T}^N)^2}\int\limits_{\mathbb{R}^N}(\int\limits_{u_2(y,s)}^{+\infty}\kappa_{\delta}(u_1(x,s)-\mathfrak{\zeta})d\mathfrak{\zeta})(\tau_z A(u_1(x,s))-A(u_1(x,s)))\Upsilon_{\varepsilon}(x-y)\mu(z)dz dx dy ds\notag\\
		&-\int\limits_0^t\int\limits_{(\mathbb{T}^N)^2}\int\limits_{\mathbb{R}^N}(\int\limits_{-\infty}^{u_1(x,s)}\kappa_{\delta}(\xi-u_2(y,s))d\xi)(\tau_z A(u_2(y,s))-A(u_2(y,s)))\Upsilon_{\varepsilon}(x-y)\mu(z) dz dx dy ds\notag\\
		&=\int\limits_0^t\int\limits_{(\mathbb{T}^N)^2}\int\limits_{\mathbb{R}^N}\mathfrak{\eta}_{\delta}'(u_1(x,s)-u_2(y,s))[(\tau_z(A(u_1(x,s))\\
		&\qquad\qquad\qquad\qquad\qquad\qquad-A(u_2(y,s)))-(A(u_1(x,s))-A(u_2(y,s)))]\Upsilon_{\varepsilon}(x-y)\mu(z)dz dx dy ds\notag\\
		&=:I_1^r+I_r^1
	\end{align*}
	where
	\begin{align*}
		I_1^r&=\int_0^t\int\limits_{(\mathbb{T}^N)^2}\int\limits_{|z|\textgreater\,r}\mathfrak{\eta}_{\delta}'(u_1(x,s)-u_2(y,s))[(\tau_z(A(u_1(x,s))\\
		&\qquad\qquad\qquad\qquad\qquad-A(u_2(y,s)))-(A(u_1(x,s))-A(u_2(y,s)))]\Upsilon_{\varepsilon}(x-y)\mu(z)dz dx dy ds,\\
		I_r^1&=\int\limits_0^t\int\limits_{(\mathbb{T}^N)^2}\int\limits_{|z|\le\,r}\mathfrak{\eta}_{\delta}'(u_1(x,s)-u_2(y,s))[(\tau_z(A(u_1(x,s))\\
		&\qquad\qquad\qquad\qquad\qquad\qquad-A(u_2(y,s)))-(A(u_1(x,s))-A(u_2(y,s)))]\Upsilon_{\varepsilon}(x-y)\mu(z)dz dx dy ds,\\
	\end{align*}
	We will estimate $I_1^r$ by the help of the Lipschitz continuity property of nonlinear function $A$ as follows:
	\begin{align*}
		I_1^r&=-\frac{1}{2}\int\limits_0^t\int\limits_{(\mathbb{T}^N)^2}\int\limits_{|z|\textgreater\,r}\big[\mathfrak{\eta}_\delta'(\tau_z(u_1(x,s)-u_2(y,s)))-\mathfrak{\eta}_\delta'(u_1(x,s)-u_2(y,s))\big]\\
		&\qquad\qquad\times\big[\tau_z(A(u_1(x,s))-A(u_2(y,s)))-(A(u_1(x,s))-A(u_2(y,s)))\big]\Upsilon_\varepsilon(x-y)\mu(z)dz dx dy ds.
	\end{align*}We observe that
	\begin{align}
		H_1&=\int\limits_0^t\int\limits_{(\mathbb{T}^N)^2}\int\limits_{|z|\textgreater\,r}\mathfrak{\eta}_\delta'(u_1(x,s)-u_2(y,s))[(\tau_z(A(u_1(x,s))-A(u_2(y,s)))\notag\\
		&\qquad\qquad\qquad-sgn(u_1(x,s)-u_2(y,s))(A(u_1(x,s))-A(u_2(y,s)))]\Upsilon_{\varepsilon}(x-y)\mu(z)dz dx dy\notag\\
		&\le\,\int\limits_0^t\int\limits_{(\mathbb{T}^N)^2}\int\limits_{|z|\,\textgreater\,r}\bigg(|\tau_z(A(u_1(x,s))-A(u_2(y,s)))|-|A(u_1(x,s))-A(u_2(y,s))|\bigg)\Upsilon_\varepsilon(x-y)\mu(z)dz dx dy ds\notag\\
		&=0,\notag\\
	\end{align}
	With the help of above inequality we conclude that
	\begin{align}
		H_2&=I_1^r-H_1\notag\\
		&=\int\limits_0^t\int\limits_{(\mathbb{T}^N)^2}\int\limits_{|z|\le\,r}[sgn(u_1(x,s)-u_2(y,s))-\mathfrak{\eta}_{\delta}'(u_1(x,s)-u_2(y,s))](A(u_1(x,s))-A(u_2(y,s))\Upsilon_{\varepsilon}(x-y)\notag\\&\qquad\qquad\qquad\qquad\qquad\qquad\qquad\qquad\qquad\qquad\qquad\qquad\qquad\qquad\qquad\mu(z)dz dx dy ds\notag\\
		&\le\,C(\omega)\,r^{-2\alpha}\,\delta.\label{estimate 3}
	\end{align}
	We will estimate $I_r^r$ by the help of non-decreasing property of nonlinear function $A$. Since $A$ is non decreasing Lipschitz continuous function, so we have the following inequality, for $a, b, k \in \mathbb{R}$,
	\begin{align*}
		\mathfrak{\eta}_{\delta}'(b-k)[A(a)-A(b)]\,\le\,\int_{k}^a\mathfrak{\eta}_{\delta}'(\sigma-k)A'(\sigma)d\sigma-\int_{k}^b\mathfrak{\eta}_{\delta}'(\sigma-k) A'(\sigma) d\sigma.
	\end{align*}
	By using above inequality, we estimate $I_r^1$ as follows:
	\begin{align*}
		&I_r^1\\&\le\,\int\limits_0^t\int\limits_{(\mathbb{T}^N)^2}\int\limits_{|z|\le\,r}\Bigg\{\int\limits_{u_2(y,s)}^{u_1(x+z)}\mathfrak{\eta}_\delta'(\sigma-u_2(y,s))A'(\sigma) d\sigma -\int\limits_{u_2(y)}^{u_1(x)}\mathfrak{\eta}_\delta' (\sigma -u_2(y,s))A'(\sigma) d\sigma\Bigg\}\Upsilon_\varepsilon(x-y)\mu(z) dz dx dy \\&+\int\limits_0^t\int\limits_{(\mathbb{T}^N)^2}\int\limits_{|z|\le\,r}\bigg\{ \int\limits_{u_1(x)}^{u_2(y+z)} \mathfrak{\eta}_\delta'(\sigma-u_1(x,s))A'(\sigma)d\sigma-\int\limits_{u_1(x,s)}^{u_2(y,s)}\mathfrak{\eta}_\delta'(\sigma-u_1(x,s))A'(\sigma) d\sigma\bigg\}\Upsilon_\varepsilon(x-y)\mu(z)dz dx dy \\
		&:=L_1+L_2
	\end{align*}
	where 
	\begin{align}
		L_1&=\int\limits_0^t\int_{(\mathbb{T}^N)^2}\int\limits_{|z|\le\,r}\bigg(\int\limits_{u_2(y,s)}^{u_1(x,s)}\mathfrak{\eta}_\delta'(\sigma-u_2(y,s))A'(\sigma)d\sigma\bigg)\bigg(\Upsilon_\varepsilon(x+z-y)-\Upsilon_\varepsilon(x-y)\bigg)\mu(z)dz dx dy ds\notag\\
		&\,\le\, \int\limits_0^t\int\limits_{(\mathbb{T}^N)^2}\int\limits_{|z|\le\,r}|u_1(x,s)-u_2(y,s)|\int\limits_0^1(1-\tau)\,D^2\,(\Upsilon_{\varepsilon}(x-y+\tau (z))) z^2d\tau|\mu(z)dz dx dy ds\notag\\
		&\le\,C(\omega)\,\big(\frac{r^{2-2\alpha}}{\varepsilon^2}\big),\label{estimate 4}\\
		L_2&=\int\limits_0^t\int\limits_{(\mathbb{T}^N)^2}\int_{|z|\le\,r}\bigg(\int_{u_1(x,s)}^{u_2(y,s)}\mathfrak{\eta}_\delta'(\sigma-u_1(x,s))A'(\sigma)d\sigma\bigg)\big(\Upsilon_\varepsilon(x-y+z)-\Upsilon_{\varepsilon}(x)\big)\mu(z)dz dy dx ds\notag\\
		&\le\,C(\omega)\int\limits_0^t\int\limits_{(\mathbb{T}^N)^2}\int\limits_{|z|\le\,r}|u_1(x,s)-u_2(y,s)|\,|\int\limits_0^1(1-\tau)D^2(\Upsilon_\varepsilon(x-y+\tau z)) z^2 d\tau\,|\mu(z)dz dx dy ds\notag\\
		&\le\,C(\omega)\,\frac{r^{2-2\alpha}}{\varepsilon^2}.\label{estimate 5}
	\end{align}
	
	\noindent
	\textbf{Step 4:} As a consequence of previous estimates, we deduce for all $t\in[0,T]$,
	\begin{align}\label{final1}
		\mathbb{E}\int_{\mathbb{T}^N}\int_{\mathbb{R}}f_{1}(x,t,\xi)f_2(x,t,\xi)d\xi dx&\le\mathbb{E}\int_{(\mathbb{T}^N)^2}\int_{\mathbb{R}^2}\Upsilon_{\varepsilon}(x-y)\kappa_{\delta}(\xi-\mathfrak{\zeta})f_{1,0}\bar{f}_{2,0}(y,\mathfrak{\zeta})d\xi d\mathfrak{\zeta} dx dy\notag\\
		&\qquad +C_T\big(\delta\,\varepsilon^{-1}+ \varepsilon^2\,\delta^{-1}+g(\delta) + r^{-2\alpha}\delta+ r^{2-2\alpha}\varepsilon^{-2}+\mathfrak{\eta}_t(\varepsilon,\delta)\big).
	\end{align}
	\noindent
	\textbf{Case 1:} If noise is genral, $\Psi=\Psi(x,u)$ and $\alpha\in(0,\frac{1}{2})$.
	We set $\varepsilon= \delta^\beta,$ $r=\delta^{\gamma}$, and choose $\beta,\,\gamma $ such that $$\min\big\{\frac{1-\alpha}{2\alpha},1 \big\}\,\textgreater\beta\,\textgreater\, \frac{1}{2},\,\,$$
	$$\frac{1}{2\alpha}\,\textgreater\,\gamma\,\textgreater\,\frac{\beta}{1-\alpha}.$$ Then after taking $\delta\to0$, we have for all $t\in[0,T]$
	$$\mathbb{E}\int_{\mathbb{T}^N}\int_{\mathbb{R}}f_1(t)\bar{f}_2(t)d\xi d\mathfrak{\zeta}\le\mathbb{E}\int_{\mathbb{T}^N}\int_{\mathbb{R}}f_{1,0}\bar{f}_{2,0}d\xi dx.$$
	\textbf{Case 2:} If noise is only multiplicative $\Psi=\Psi(u)$ and $\alpha\in(0,1)$, then $\varepsilon^2\delta^{-1}$ does not present in right hand side of inequality \eqref{final1}. So first letting $\delta\,\to\,0$ and after that taking $r\,\to\,0$, $\varepsilon\,\to\,0$, yields
	$$\mathbb{E}\int_{\mathbb{T}^N}\int_{\mathbb{R}}f_1(t)\bar{f}_2(t)d\xi d\mathfrak{\zeta}\le\mathbb{E}\int_{\mathbb{T}^N}\int_{\mathbb{R}}f_{1,0}\bar{f}_{2,0}d\xi dx.$$
	From previous both cases it is clear that for all $t\in[0,T]$
	$$\mathbb{E}\|u_1(t)-u_2(t)\|_{L^1(\mathbb{T}^N)}\le\,\mathbb{E}\|u_{1,0}-u_{2,0}\|.$$
	Let us now prove remaing part of the this result. To conclude this we proceed as follows, 
	making use of inequality \eqref{app inequality1} and pathwise estimates \eqref{estimate 1}-\eqref{estimate 5}, we can similarly conclude  that $\mathbb{P}$-almost surely, for all $t\in[0,T]$,
	\begin{align}\label{contraction principle 2}
		\int_{\mathbb{T}^N}\int_{\mathbb{R}}f^{-}(x,t,\xi)\bar{f}^{-}(x,t,\xi)d\xi dx\,\le\,\int_{\mathbb{T}^N}\int_{\mathbb{R}}f_{0}(x,\xi)\bar{f}_{0}(x,\xi)d\xi dx.
	\end{align}
	Since $f_0(x,\xi)\bar{f}_0(x,\xi)=0$, therefore $\mathbb{P}$-almost surely, for all $t\in[0,T]$, $f^{-}(x,t,\xi)\big(1-f^{-}(x,t,\xi)\big)=0$ almost every $(x,\xi)\in \mathbb{T}^N\times\mathbb{R}$, then Fubini's theorem imply that, $\mathbb{P}$-almost surely, for all $t\in[0,T]$, there exists a set $S\subset\mathbb{T}^N$ of full measure such that, for $x\in S,f^{-}(x,t,\xi)\in\{0,1\}$ for almost every $\xi\in\mathbb{R}$. Therefore by using the fact that $f^{-}$ is a kinetic function, we conclude that for all $t\in [0,T]$, there exists $u^{-}(t):\Omega\to L^1(\mathbb{T}^N)$ such that $\mathbb{P}$-almost surely, for all $t\in[0,T]$ $f^{-}(t)=\mathbbm{1}_{u^{-}(t)\textgreater\xi}$. 
	
\end{proof}
\begin{cor}\label{th3.7} Let $u_0\in L^p(\Omega;L^p(\mathbb{T}^N))$. Then, for all $p\in[1,+\infty)$, the solution u to \eqref{1.1} has almost surely continuous trajectories in $L^p(\mathbb{T}^N)$.
\end{cor}
\begin{proof}
	We refer to the proof of \cite[Corollary 3.3]{sylvain} for a proof.
\end{proof}
\subsection{Existence:}\label{subsection 3.2}
In this section, we show the existence part of the kinetic solution. We divide proof of existence into two-parts. First, we show the existence of a solution for smooth initial data. Then with the help of the contraction principle, we show the existence of kinetic solution for general initial data.
\subsection*{Existence for smooth initial data}
We prove the existence part of Theorem \ref{th2.10} for the initial data $u_0\in L^p(\Omega;C^\infty (\mathbb{T}^N)).$
Here we use the vanishing viscosity method. Consider a truncation $(\mathcal{X}_\tau)$ on $\mathbb{R}$ and approximations $(\kappa_\tau)$, ($\psi_\tau$) to identity on $\mathbb{T}^N$ and $\mathbb{R}$, respectively. 
Then smooth approximations of $F$ and $A$ are defined as follows:
$$F_k^\tau(\mathfrak{\zeta})= \big((F_k * \kappa_\tau)\mathcal{X}_\tau\big)(\mathfrak{\zeta})\,\,\,\, k=1,...,N,$$
$$A^\tau(\mathfrak{\zeta})=(A*\psi_\tau)(\mathfrak{\zeta}).$$
Consequently, we define $F^\tau=(F_1^\tau,...,F_N^\tau)$. Note that $F^\tau$ is of class $C^\infty(\mathbb{R};\mathbb{R}^N)$ with the compact support, therefore it is Lipschitz continuous.  Clearly, the polynomial growth of $F$ remains also valid for $F^\tau$ and holds uniformly in $\tau$. We consider the following equation to approximate original equation \eqref{1.1}:
\begin{align}\label{4.1}
	d\mathfrak{u^\tau}(x,t)+\mbox{div}(F^\tau(\mathfrak{u^\tau}(x,t)))d t +(-\Delta)^\alpha[A^\tau(\mathfrak{u^\tau}(x,t))]dt 
	&=\tau\Delta \mathfrak{u^\tau}+ \Psi(x,\mathfrak{u^\tau}(x,t))dW(t)\,\, x \in \mathbb{T}^N,\,\, t \in(0,T),\notag\\
	\mathfrak{u^\tau}(0)&=u_0.
\end{align} 
There exists a unique weak solution $\mathfrak{u^\tau}$ to \eqref{4.1} such that
$$\mathfrak{u^\tau}\in L^2(\Omega;C([0,T];L^2(\mathbb{T}^N)))\cap L^2(\Omega;L^2(0,T;H^1(\mathbb{T}^N))).$$
For the existence of viscous solutions, we refer to \cite[Section 6]{Neeraj2}. The viscous equations \eqref{4.1} have weak solutions and consequent passage to the limit proves the existence of a kinetic solution to \eqref{1.1}. Nevertheless, this approximation method is quite technical and has to be done in many following steps.
\subsection*{Strong convergence}\label{section 5.2}
	Our aim of this subsection is to show that strong convergence holds true. Towards this end, we make use of the ideas developed in the subsection \ref{subsection 3.1}.
\begin{thm}\label{L1 convergence}\label{convergence} Let $\mathfrak{u^{\tau}}$ be viscous solution to \eqref{4.1} with initial data $u_0$, then there exists $u\in L^1_{\mathcal{P}}(\Omega\times[0,T]; L^1(\mathbb{T}^N))$ such that
	$$\mathfrak{u^\tau}\,\to\,u\,\,\,\text{in}\,\,\,L^1_{\mathcal{P}}(\Omega\times[0,T]; L^1(\mathbb{T}^N)).$$ 
	\begin{proof} Proof of this result is based on the computations of Theorem \ref{comparison}.
		
		\noindent
		\textbf{Step 1:} Based on the proof of Theorem \ref{comparison}, we can easily conclude that for all $t\in[0,T]$
		\begin{align*}
		&\mathbb{E}\,\int_{(\mathbb{T}^N)^2}\int_{\mathbb{R}}\Upsilon_\varepsilon(x-y) f^\tau(x,t,\mathfrak{\zeta})\bar{f}^\tau(y,t,\mathfrak{\zeta})d\mathfrak{\zeta} dx dy\\
		&\qquad\le\,\mathbb{E}\int_{(\mathbb{T}^N)^2}\int_{(\mathbb{R})^2}\Upsilon_\varepsilon(x-y)\kappa_\delta(\mathfrak{\zeta}-\mathfrak{\zeta})f^\tau (x,t,\mathfrak{\zeta})\bar{f}^\tau(y,t,\mathfrak{\zeta}) d\mathfrak{\zeta} d\mathfrak{\zeta} dx dy + \delta\\
		&\qquad\le\,\mathbb{E}\int_{(\mathbb{T}^N)^2}\int_{(\mathbb{R}^2)}\Upsilon_\varepsilon(x-y)\kappa_\delta(\mathfrak{\zeta}-\mathfrak{\zeta}) f_0(x,\mathfrak{\zeta}) \bar{ f}_0(y,\mathfrak{\zeta})d\mathfrak{\zeta} d\mathfrak{\zeta} dx dy + \delta + \mathcal{E}_\Upsilon+ \mathcal{E}_\kappa+ \mathcal{J}^\tau+J\\
		&\qquad\le\,\mathbb{E}\int_{(\mathbb{T}^N)^2}\int_{(\mathbb{R})}\Upsilon_\varepsilon(x-y) f_0(x,\mathfrak{\zeta})\bar{ f}_0(y,\mathfrak{\zeta}) d\mathfrak{\zeta} dx dy+ 2\delta+ \mathcal{E}_\Upsilon + \mathcal{E}_\kappa+ \mathcal{J}^\tau+J,
		\end{align*}
		where $\mathcal{E}_\Upsilon, \mathcal{E}_{\kappa}, J$ are introduced similarly to Theorem \ref{th3.6}, $\mathcal{J}^\tau$ is regarding to the second order term $\tau\Delta \mathfrak{u^\tau}$:
		\begin{align*}
		\mathcal{J}^\tau&=2\tau\mathbb{E}\int_0^t\int_{(\mathbb{T}^N)^2}\int_{\mathbb{R}^2} f^\tau \bar{ f}^\tau \Delta_x\Upsilon_\varepsilon(x-y)\kappa_\delta(\mathfrak{\zeta}-\mathfrak{\zeta})d\mathfrak{\zeta} d\mathfrak{\zeta} dx dy ds\\
		&\qquad-\mathbb{E}\int_0^t\int_{(\mathbb{T}^N)^2}\int_{\mathbb{R}^2}\Upsilon_\varepsilon(x-y)\kappa_\delta(\mathfrak{\zeta}-\mathfrak{\zeta}) d\mathfrak{\mathcal{V}}_{x,s}^\tau(\mathfrak{\zeta}) dx d\mathfrak{\eta}_2^\tau(y,s,\mathfrak{\zeta})\\
		&\qquad-\mathbb{E}\int_0^t\int_{(\mathbb{T}^N)^2}\int_{\mathbb{R}^2}\Upsilon_\varepsilon(x-y) \kappa_\delta(\mathfrak{\zeta}-\mathfrak{\zeta})d\mathfrak{\mathcal{V}}_{y,s}^\tau(\mathfrak{\zeta}) dy d\mathfrak{\eta}_2^\tau(x,s,\mathfrak{\zeta})\\
		&=-\tau\,\mathbb{E}\int_0^t\int_{(\mathbb{T}^N)^2}\Upsilon_\varepsilon(x-y)\kappa_\delta(\mathfrak{u^\tau}(x)-\mathfrak{u^\tau}(y))|\nabla_x \mathfrak{u^\tau}-\nabla_y \mathfrak{u^\tau}|^2 dx dy\,\le\,0.
		\end{align*}
		It shows that for all $t\in[0,T]$ 
		\begin{align}\label{final21}
		\mathbb{E}\int_{(\mathbb{T}^N)^2}\Upsilon_\varepsilon(x-y)|\mathfrak{u^\tau}(x,t)-\mathfrak{u^\tau}(y,t)|dx dy &\le\,\mathbb{E}\int_{(\mathbb{T}^N)^2}\Upsilon_\varepsilon(x-y)|u_0(x)-u_0(y)|dx dy\,\notag\\&\qquad+\,C_T( \delta+\delta\,\varepsilon^{-1}+\varepsilon^2\delta^{-1} + g(\delta)+r^{-2\alpha}\delta+r^{2-2\alpha}\varepsilon^{-2})	
		\end{align}
		\textbf{Step 2:}
		By similar techniques as in the proofs of Theorem \ref{comparison}, we obtain for any two viscous solution $\mathfrak{u^\tau},\,\,\mathfrak{u^\sigma}$
		\begin{align}\label{heading back}
		&\mathbb{E}\,\int_{\mathbb{T}}\big(\mathfrak{u^\tau}(t)-\mathfrak{u^\sigma}(t)\big)^+dx=\mathbb{E}\int_{\mathbb{T}^N}\int_{\mathbb{R}}f^\tau(x, t, \mathfrak{\zeta}) \bar{ f}^\sigma(x,t,\mathfrak{\zeta}) d\mathfrak{\zeta} dx\notag\\
		&\qquad=\mathbb{E}\int_{(\mathbb{T}^N)^2}\int_{\mathbb{R}^2}\Upsilon_\varepsilon(x-y)\kappa_{\delta}(\mathfrak{\zeta}-\mathfrak{\zeta}) f^\tau(x,t,\mathfrak{\zeta})\bar{ f}^\sigma(y,t,\mathfrak{\zeta}) d\mathfrak{\zeta} d\mathfrak{\zeta} dx dy +\mathfrak{\eta}_{1}(\tau, \sigma,\varepsilon,\delta).
		\end{align}
		We want to show that the error term $\mathfrak{\eta}_{t}(\tau,\sigma, \varepsilon, \delta)$ is independent of $\tau,\sigma$ as follows:
		\begin{align*}
		\mathfrak{\eta}_t(\tau,\sigma,\varepsilon,\delta)&=\mathbb{E}\int_{\mathbb{T}^N}\int_{\mathbb{R}}f^\tau(x,t,\mathfrak{\zeta})\bar{ f}^\sigma(x,t,\mathfrak{\zeta}) d\mathfrak{\zeta} dx\\&\qquad\qquad\qquad\qquad-\mathbb{E}\int_{(\mathbb{T}^N)^2}\int_{\mathbb{R}^2}\Upsilon_{\varepsilon}(x-y)\kappa_{\delta}(\mathfrak{\zeta}-\mathfrak{\zeta}) f^\tau(x,t,\mathfrak{\zeta}) \bar{ f}^\sigma(y,t,\mathfrak{\zeta}) d\mathfrak{\zeta} d\mathfrak{\zeta} dx dy\\
		&=\bigg(\mathbb{E}\int_{\mathbb{T}^N}\int_{\mathbb{R}}f^\tau(x,t,\mathfrak{\zeta})\bar{ f}^\sigma(x,t,\mathfrak{\zeta})d\mathfrak{\zeta} dx-\mathbb{E}\int_{\int_{()\mathbb{T}^N)^2}}\int_\mathbb{R}\Upsilon_\varepsilon(x-y)f^\tau(x,t,\mathfrak{\zeta})\bar{ f}^\sigma(y,t,\mathfrak{\zeta}) d\mathfrak{\zeta} dx dy\bigg)\\
			&\qquad+\bigg(\mathbb{E}\int_{(\mathbb{T}^N)^2}\int_{\mathbb{R}}\Upsilon_\varepsilon(x-y)f^\tau(x,t,\mathfrak{\zeta})\bar{ f}^\sigma(y,t,\mathfrak{\zeta}) d\mathfrak{\zeta} dx dy\\&\qquad\qquad\qquad\qquad-\mathbb{E}\int_{(\mathbb{T}^N)^2}\int_{\mathbb{R}^2}\Upsilon_\varepsilon(x-y)\kappa_\delta(\mathfrak{\zeta}-\mathfrak{\zeta}) f^\tau(x,t,\mathfrak{\zeta})\bar{ f}^\sigma(y,t,\mathfrak{\zeta})d\mathfrak{\zeta} d\mathfrak{\zeta} dx dy\bigg)\\
			=H_1+H_2,
		\end{align*}
		where
		\begin{align*}
		|H_1|&=\bigg|\mathbb{E}\int_{(\mathbb{T}^N)^2}\Upsilon_\varepsilon(x-y)\int_{\mathbb{R}}\mathbbm{1}_{\mathfrak{u^\tau}(x)\,\textgreater\,\mathfrak{\zeta}}\big[\mathbbm{1}_{\mathfrak{u^\sigma}(x)\,\le\,\mathfrak{\zeta}}-\mathbbm{1}_{\mathfrak{u^\sigma}(y)\,\le\,\mathfrak{\zeta}}\big]d\mathfrak{\zeta} dx dy\bigg|\\
		&=\bigg|\mathbb{E}\int_{(\mathbb{T}^N)^2}\Upsilon_\varepsilon(x-y)\big(\mathfrak{u^\sigma}(y)-\mathfrak{u^\sigma}(x)\big)dx dy\bigg|\\
		&\le\,\mathbb{E}\int_{(\mathbb{T}^N)^2}\Upsilon_\varepsilon(x-y)|u_0(x)-u_0(y)|dx dy\,+\,C_T( \delta + \delta\,\varepsilon^{-1}+\varepsilon^2\delta^{-1}\notag\\
		&\qquad+g(\delta)+r^{-2\alpha}\delta+r^{2-2\alpha}\varepsilon^{-2})
		\end{align*}
		 due to \eqref{final21} and $|H_2|\,\le\,\delta$. Come back to \eqref{heading back} and using the same computations as in Theorem \ref{comparison}, we obtain
		 \begin{align*}\mathbb{E}\int_{\mathbb{T}^N}\bigg(\mathfrak{u^\tau}(t)-\mathfrak{u^\sigma}(t)\bigg)^+dx&\le\,\mathbb{E}\int_{(\mathbb{T}^N)^2}\Upsilon_\varepsilon(x-y)|u_0(x)-u_0(y)|dx dy\,+\,C_T( \delta+\delta\,\varepsilon^{-1}+\varepsilon^2\delta^{-1}\notag\\
		 &\qquad+g(\delta)+r^{-2\alpha}\delta+r^{2-2\alpha}\varepsilon^{-2})+2\delta+ \mathcal{E}_\Upsilon+ \mathcal{E}_\kappa+ J^\#+J.
		 \end{align*}
		 The terms $\mathcal{E}_\Upsilon, \mathcal{E}_\kappa$, and $J$ are defined and can be dealt with similarly as in Theorem \ref{comparison}. The term $J^\#$ is defined as
		 \begin{align*}
		  J^\#&=(\tau+\sigma) \mathbb{E}\int_0^t\int_{(\mathbb{T}^N)^2}\int_{\mathbb{R}^2} f^\tau \bar{ f}^\sigma \Delta_x\Upsilon_{\varepsilon}(x-y)\kappa_{\delta}(\mathfrak{\zeta}-\mathfrak{\zeta}) d\mathfrak{\zeta} d\mathfrak{\zeta} dx dy ds\\
		 &\qquad-\mathbb{E}\int_0^t\int_{(\mathbb{T}^N)^2}\int_{\mathbb{R}^2}\Upsilon_\varepsilon(x-y)\kappa_{\delta}(\mathfrak{\zeta}-\mathfrak{\zeta}) d\mathfrak{\mathcal{V}}_{x,s}^\tau(\mathfrak{\zeta}) dx d\mathfrak{\eta}_{2}^\sigma(y,s,\mathfrak{\zeta}) \\
		 &\qquad-\mathbb{E}\int_0^t\int_{(\mathbb{T}^N)^2}\int_{\mathbb{R}^2}\Upsilon_{\varepsilon}(x-y)\kappa_{\delta}(\mathfrak{\zeta}-\mathfrak{\zeta}) d\mathfrak{\mathcal{V}}_{y,s}^\sigma(\mathfrak{\zeta}) dy d\mathfrak{\eta}_2(x,s,\mathfrak{\zeta}).
		 \end{align*}
		We follow \cite[Section 6]{vovelle} to estimate $J^\#$ as follows:
		 \begin{align*}
		 |J^\#|&\,\le\,C\,(\sqrt{\tau}-\sqrt{\sigma})^2\,\varepsilon^{-2},
		 \end{align*}
		  Consequently, we see that for all $t\in[0,T]$
		 \begin{align}\label{final31}
		 \mathbb{E}\int_{\mathbb{T}^N} \big(\mathfrak{u^\tau}(t)-\mathfrak{u^\sigma}(t))\big)^+dx dt\,&\le\,\bigg(\mathbb{E}\int_{(\mathbb{T}^N)^2}\Upsilon_\varepsilon(x-y)|u_0(x)-u_0(y)|dx dy\,+C_T\big(\varepsilon^\mathfrak{\zeta}+ 2\delta\notag\\&\qquad+ \varepsilon^2\delta^{-1}+g(\delta)+r^{-2\alpha}\delta+r^{2-2\alpha}\varepsilon^{-2}\big)\bigg)+C_T\big(\tau+\sigma\big)\varepsilon^{-2}
		 \end{align}
	\textbf{Case 1:} If noise is general, $\Psi=\Psi(x,u)$ and $0\,\textless\,\alpha\,\textless\frac{1}{2}$, and we 
	set $\varepsilon= \delta^\beta, r=\delta^{\gamma}$, where $\beta,\,\gamma $ such that $$\min\big\{\frac{1-\alpha}{2\alpha},1 \big\}\,\textgreater\beta\,\textgreater\, \frac{1}{2},\,\,$$
	$$\frac{1}{2\alpha}\,\textgreater\,\gamma\,\textgreater\,\frac{\beta}{1-\alpha}.$$
	Then, given $\varpropto\,\textgreater\,0$ We can fix $\delta$ small enough so that the first term on the right hand side is bounded by $\frac{\varpropto}{2}$ and then we can find $\varkappa\,\textgreater\,0$ such that the second term is also bounded by $\frac{\varpropto}{2}$ for any $\tau, \sigma\,\textless\,\varkappa$. It shows that the sequence of viscous solutions $\{\mathfrak{u^\tau}\}$ is Cauchy sequence in $L^1_{\mathcal{P}}(\Omega\times[0,T] ; L^1(\mathbb{T}^N))$, as $\tau\,\to\,0$.
	
	\noindent
	\textbf{Case 2:} If noise is only multiplicative, $\Psi=\Psi(u)$ and $0\,\textless\,\alpha\,\textless\,1$, then $\varepsilon^2\,\delta^{-1}$ will not present in \eqref{final31}. Letting $\delta\,\to\,0$, then for all $t\in[0,T]$
	\begin{align}\label{final3}
	\mathbb{E}\int_{\mathbb{T}^N} \big(\mathfrak{u^\tau}(t)-\mathfrak{u^\sigma}(t))\big)^+dx dt\le\mathbb{E}\int_{(\mathbb{T}^N)^2}\Upsilon_\varepsilon(x-y)|u_0(x)-u_0(y)|dx dy+ C_T\big(\varepsilon^\mathfrak{\zeta}+r^{2-2\alpha}\varepsilon^{-2} + \tau+\sigma\big)\varepsilon^{-2}.
	\end{align}
	Set $r=\varepsilon^a$ with $a\textgreater\frac{2}{2-2\alpha}$. Then, as previous case we can conclude sequence of viscous solutions $\{\mathfrak{u^{\tau}}\}$ is Cauchy sequence in $L^1_{\mathcal{P}}(\Omega\times[0,T] ; L^1(\mathbb{T}^N))$, as $\tau\,\to\,0$. It implies that there exists $u\in L^1_{\mathcal{P}}(\Omega\times[0,T]; L^1(\mathbb{T}^N))$ such that
	$\mathfrak{u^{\tau}}\,\to\,u\,\,\,\text{in}\,\,\,L^1_{\mathcal{P}}(\Omega\times[0,T]; L^1(\mathbb{T}^N))$ as $\tau\,\to\,0.$
    \end{proof}	\end{thm}
\subsection*{Uniform $L^p$-estimate} In the following proposition we show that viscous solutions $\mathfrak{u^\tau}$ are uniformly bounded in $L^p(\Omega\times[0,T];L^p(\mathbb{T}^N))$ and kinetic measure $m^\tau$ corresponding these viscous solutions have uniform decay in $\mathfrak{\zeta}$.
\begin{proposition}\label{p estimate}
	For all $p\in [2,\infty)$, the viscous solutions $\mathfrak{u^\tau}$ to \eqref{4.1} satisfy the following estimates:
	\begin{align}\label{lp estimate}
	\mathbb{E}\sup_{0\le t\le T}\|\mathfrak{u^\tau}(t)\|_{L^p(\mathbb{T}^N)}^p\le C (1+\mathbb{E}\|u_0\|_{L^p(\mathbb{T}^N)}^p),
	\end{align}
	\begin{align}\label{4.7}
	\mathbb{E}\sup_{0\le t\le T}\|\mathfrak{u^\tau}(t)\|_{L^p(\mathbb{T}^N)}^p+\int_0^T\int_{\mathbb{T}^N}\int_{\mathbb{R}} |\mathfrak{\zeta}|^{p-2} &d\mathfrak{m}^{\tau}(x,s,\mathfrak{\zeta})\le C (1+\mathbb{E}\|u_0\|_{L^p(\mathbb{T}^N)}^p).
	\end{align}
	Here, the constant C does not depend on $\tau$ .
\end{proposition}
\begin{proof}Proof of this proposition is an application of the generalized It\^o formula \cite[Appendix A]{vovelle}. We can not apply the generalized It\^o formula directly to $\kappa(\mathfrak{\zeta})=|\mathfrak{\zeta}|^p$, $p\in [2,\infty)$, and $\varphi(x)=1$. 
	Here, we define functions $\kappa_k\in C^2(\mathbb{R})$ that approximate $\kappa$ and have quadratic growth at infinity as necessary to apply generalized It\^o formula. We define $\kappa_k$ as
	\[\kappa_k(\mathfrak{\zeta})=\begin{cases}
	|\mathfrak{\zeta}|^p, & |\mathfrak{\zeta}|\,\le\, k,\\
	k^{p-2}\bigg[\frac{p(p-1)}{2}\mathfrak{\zeta}^2-p(p-2)k|\mathfrak{\zeta}|+\frac{(p-1)(p-2)}{2}k^2\bigg], & |\mathfrak{\zeta}|\,\textgreater \,k.
	\end{cases}
	\]
It is clear that
	\begin{align*}|\mathfrak{\zeta} \kappa_k'(\mathfrak{\zeta})| \,&\le \, \kappa_k (\mathfrak{\zeta}),\\
	|\kappa_k'(\mathfrak{\zeta})| \, &\le p(1+\kappa_k(\mathfrak{\zeta})),\\
	|\kappa_k'(\mathfrak{\zeta})| \, &\le |\mathfrak{\zeta}|\kappa''(\mathfrak{\zeta}),\\
	\mathfrak{\zeta}^2 \kappa_k ''(\mathfrak{\zeta})\, &\le p(p-1) \kappa_k (\mathfrak{\zeta}),\\
	\kappa_k''(\mathfrak{\zeta}) \,&\le p(p-1)(1+\kappa_k(\mathfrak{\zeta}))
	\end{align*}
	holds true for all $\mathfrak{\zeta}\in\mathbb{R}$, $k\in\mathbb{N}$, $p\in[2,\infty)$. Then, by generalized It\^o formula \cite[Appendix A]{vovelle}, $\mathbb{P}$-almost surely, for all $t\in[0,T]$ 
	\begin{align*}
	\int_{\mathbb{T}^N} \kappa_k(\mathfrak{u^\tau}(t)) dx &= \int_{\mathbb{T}^N} \kappa_k(u_0) dx -\int_0^t \big\langle \kappa_k'(\mathfrak{u^\tau}), \mbox{div}(F^\tau(\mathfrak{u^\tau}))\big\rangle ds-\int_0^t\big\langle \kappa_k'(\mathfrak{u^\tau}), (-\Delta)_x^\alpha(A^\tau(\mathfrak{u^{\tau}}))\big\rangle ds\\
	&\qquad+\int_0^t\big\langle \kappa_k'(\mathfrak{u^{\tau}}), \tau \Delta \mathfrak{u^{\tau}}\big\rangle ds+\sum_{k\ge1}\int_0^t\big\langle \kappa_k'(\mathfrak{u^{\tau}}), h_k(x,\mathfrak{u^{\tau}})\big\rangle d\beta_k(s)\\
	&\qquad+\frac{1}{2}\int_0^t\big\langle\kappa_k''(\mathfrak{u^{\tau}}), H^2(\mathfrak{u^{\tau}})\big\rangle ds.  
	\end{align*}
	If we define $H^\tau(\mathfrak{\zeta})=\int_0^\mathfrak{\zeta} \kappa_k''(\mathfrak{\zeta}) F^\tau(\mathfrak{\zeta}) d\mathfrak{\zeta}$, then it can be seen that the second term on the right-hand side vanishes due to the boundary conditions. The third and fourth terms are nonpositive as follows:
	$$-\int_0^t\big\langle\kappa_k'(\mathfrak{u^{\tau}}), (-\Delta)_x^\alpha(A^\tau(\mathfrak{u^{\tau}}))\big\rangle ds=-\lim_{\delta\to0}\int_0^t\int_{\mathbb{T}^N}\int_{\mathbb{R}}\kappa_k''(\mathfrak{\zeta})\mathfrak{\eta}_\delta^\tau(x,\mathfrak{\zeta},s)$$
	$$\int_0^t\big\langle\kappa_k'(\mathfrak{u^{\tau}}), \tau \Delta \mathfrak{u^{\tau}} \big\rangle ds= -\int_0^t\int_{\mathbb{T}^N}\kappa_k''(\mathfrak{u^{\tau}}) \tau |\nabla \mathfrak{u^{\tau}}|^2 dx ds,$$
	where $\mathfrak{\eta}^\delta=\int_{\mathbb{R}^N} |A^\tau(u^\delta(x+z,t))-A^\tau(\xi)|\mathbbm{1}_{\mbox{Conv}\{\mathfrak{u}^{\tau,\delta}(x+z,t),\mathfrak{u}^{\tau,\delta}(x,t)\}}(\xi)\mu(z)dz$ and $\mathfrak{u}^{\tau,\delta}$ is pathwise molification of $\mathfrak{u}^\tau$ in space variable. For discussion on this argument we refer to Appendix \ref{A}.
	After this, we can follow the proof of \cite[Proposition 4.3]{vovelle} to conclude the result.
	\end{proof}
\begin{remark}[\textbf{Approximate solutions}] 
Here, we use the computations of Appendix \ref{A} to derive the kinetic formulation to \eqref{4.1} that satisfied by $f^\tau(t)=\mathbbm{1}_{\mathfrak{u^{\tau}}(t)\,\textgreater\,\mathfrak{\zeta}}$ in the sense of $D'(\mathbb{T}^N\times\mathbb{R})$. We read as follows: $\mathbb{P}$-almost surely, for all $t\in[0,T]$
	\begin{align*}
	df^\tau(t)+{F^{\,\tau}}'\cdot\nabla f^\tau(t) dt - \tau \Delta f^\tau(t) dt+{A^\tau}'(\mathfrak{\zeta}) (-\Delta)_x^\alpha(f^\tau(t))dt&= \delta_{\mathfrak{u^{\tau}}(t)=\mathfrak{\zeta}} \varphi\, dB(t) + \partial_{\mathfrak{\zeta}}(\mathfrak{m}^\tau -\frac{1}{2}H^{2} \delta_{\mathfrak{u^{\tau}}(t)=\mathfrak{\zeta}}) dt, 
	\end{align*}
	where $d\mathfrak{m}^\tau(x,t,\mathfrak{\zeta})\ge d\mathfrak{m}_1^\tau(x,t,\mathfrak{\zeta})+d\mathfrak{\eta}_1^\tau(x,t,\mathfrak{\zeta}),$ 
	$d\mathfrak{m}_1^\tau(x,t,\mathfrak{\zeta})=\tau |\nabla \mathfrak{u^{\tau}}|^2\delta_{\mathfrak{u^{\tau}}=\mathfrak{\zeta}}dx dt,$ and
	$$d\mathfrak{\eta}_1^\tau(x,t,\mathfrak{\zeta})=\int_{\mathbb{R}^N}|A^\tau(\mathfrak{u^{\tau}}(x+z))-A^\tau(\mathfrak{\zeta})|\mathbbm{1}_{\mbox{Conv}\{\mathfrak{u^{\tau}}(x),\mathfrak{u^{\tau}}(x+z)\}}\mu(z)dz.$$
	It shows that for all $\varphi\in C_c^2(\mathbb{T}^N\times\mathbb{R})$, $\mathbb{P}$-almost surely, for all $t\in[0,T]$, 
	\begin{align}\label{approximate formulation}
	\langle f^\tau(t),\varphi\rangle &= \langle f_0^\tau,\varphi \rangle +\int_0^t \langle f^\tau(s), {F^{\,\tau}}'(\mathfrak{\zeta})\cdot\nabla_x \varphi \rangle ds -\int_0^t \langle f^\tau(s), {A^\tau}'(\mathfrak{\zeta}) (-\Delta)^\alpha(\varphi)\rangle ds\notag\\
	&\qquad+ \int_0^t \int_{\mathbb{T}^N}\int_{\mathbb{R}} h_k(x,\mathfrak{\zeta}) \varphi(x,\mathfrak{\zeta}) d\mathfrak{\mathcal{V}}_{x,s}^\tau(\mathfrak{\zeta}) dx d\beta_k(s)+ \tau\int_0^t \langle f^\tau(s),\Delta \varphi \rangle ds.\notag \\
	& \qquad+\frac{1}{2}\int_0^t\int_{\mathbb{T}^N}\int_{\mathbb{R}} H^{2}(x,\mathfrak{\zeta}) \partial_{\mathfrak{\zeta}}\varphi(x,\mathfrak{\zeta}) d\mathfrak{\mathcal{V}}_{x,s}^\tau(\mathfrak{\zeta})dx ds- \mathfrak{m}^\tau(\partial_{\mathfrak{\zeta}}\varphi)([0,t]).
	\end{align}
 \end{remark}
\subsection*{Convergence of approximate kinetic functions}\label{section 5.3}
Here, we will use the following notations: we say that a sequence $(\mathfrak{\mathcal{V}}^{\mathfrak{\tau_n}})$ of Young measures converges to $\mathfrak{\mathcal{V}}$ in $\mathfrak{\mathcal{Y}}^1$ if \eqref{2.12} is satisfied. By definition, a random Young measure is a $\mathfrak{\mathcal{Y}}^1$- valued random variable. We define Young measures on $\mathbb{T}^N\times[0,T]$ as
$\mathfrak{\mathcal{V}}^{\mathfrak{\tau_n}}=\delta_{\mathfrak{u^{\mathfrak{\tau_n}}}=\mathfrak{\zeta}},\,\, \& \,\,{\mathfrak{\mathcal{V}}}=\delta_{{u}=\mathfrak{\zeta}}$ and following uniform bound holds,
\begin{align}\label{4.10}\mathbb{E}\Bigg[\sup_{t\in[0,T]}\int_{\mathbb{T}^N}\int_{\mathbb{R}}|\mathfrak{\zeta}|^p d\mathfrak{\mathcal{V}}_{x,t}^{\mathfrak{\tau_n}}(\mathfrak{\zeta}) dx\Bigg]\le\,C_p
\end{align}
\begin{proposition}\label{D}
	It holds true $($up to subsequences$)$
	\begin{enumerate}
		\item[1.] ${\mathbb{P}}$-almost surely, ${\mathfrak{\mathcal{V}}}^{\mathfrak{\tau_n}}$ converges ${\mathfrak{\mathcal{V}}}$ in $\mathfrak{\mathcal{Y}}^1$.
		\item[2.] ${\mathfrak{\mathcal{V}}}$ satisfies the following bound:
		\begin{align}\label{4.11}
		{\mathbb{E}}\Bigg(\sup_{I\subset[0,T]}\frac{1}{|I|}\int_{I}\int_{\mathbb{T}^N}\int_{\mathbb{R}}|\mathfrak{\zeta}|^p d{\mathfrak{\mathcal{V}}}_{x,t}(\mathfrak{\zeta})dx dt\Bigg)&\le\, C_p.
		\end{align}
		Here, the supremum in \eqref{4.11} is a supremum over all open intervals $I\subset[0,T]$ with rational end points which are countable. 
		\item[3.] Suppose that ${f}^{\mathfrak{\tau_n}}, {f}:\Omega\times\mathbb{T}^N\times[0,T]\times\mathbb{R}\to [0,1]$ are defined by
		$${f}^{\mathfrak{\tau_n}}(x,t,\mathfrak{\zeta})={\mathfrak{\mathcal{V}}}_{x,t}^{\mathfrak{\tau_n}}(\mathfrak{\zeta},+\infty),\,\,\,\, {f}(x,t,\mathfrak{\zeta})={\mathfrak{\mathcal{V}}}_{x,t}(\mathfrak{\zeta},+\infty),$$
		then ${f}^{\mathfrak{\tau_n}}\to{f}$ in $L^{\infty}(\mathbb{T}^N\times[0,T]\times\mathbb{R})$-weak-* ${\mathbb{P}}$-almost surely.
		\item[4.] There exists a full measure suset $\mathcal{S}$ of $[0,T]$, containing $0$  such that for all $t\in \mathcal{S}$
		$${f}^{\mathfrak{\tau_n}}\to {f}\,\,\, in\,\,\,\,\,\,L^\infty (\Omega\times\mathbb{T}^N\times\mathbb{R})-weak^*.$$
		$$$$
	\end{enumerate}
\end{proposition}
\begin{proof}
	We refer to \cite[Propostion 5.3]{Chaudhary} for a proof.
\end{proof}
\subsection*{Compactness of the Random Measures}\label{section 5.4}
 We know the following duality holds for $q$, $q^*\in(1,\infty)$ being conjugate exponents 
$$L_w^q({\Omega};\mathcal{M}_b(\mathbb{T}^N\times[0,T]\times\mathbb{R})\simeq(L^{q^*}({\Omega};C_0(\mathbb{T}^N\times[0,T]\times\mathbb{R}))^*,$$
where $\mathcal{M}_b(\mathbb{T}^N\times[0,T]\times{R})$ denote the space of bounded Borel measures on $\mathbb{T}^N\times[0,T]\times\mathbb{R}$ whose norm is given by the total variation of measures. It is the dual space to $C_0(\mathbb{T}^N\times[0,T]\times\mathbb{R})$, the collection of all continuous functions vanishing at infinity equipped with the supremum norm. $L_w^q({\Omega};\mathcal{M}_b(\mathbb{T}^N\times[0,T]\times\mathbb{R})$ contains all $weak^*$-measurable functions $\mathfrak{\eta}:{\Omega}\to\mathcal{M}_b(\mathbb{T}^N\times[0,T]\times\mathbb{R})$ such that
$${\mathbb{E}}\left\Vert \mathfrak{\eta}\right\Vert_{\mathcal{M}_b}^q\textless\,\infty.$$ 
Let us define measures
$$d{\mathfrak{m}}_1^{\mathfrak{\tau_n}}(x,t,\mathfrak{\zeta})=\mathfrak{\tau_n}|\nabla \mathfrak{u^{{\tau_n}}}|^2d\delta_{\mathfrak{u^{\tau_n}}=\mathfrak{\zeta}}(\mathfrak{\zeta})\,dx\,dt,$$
$$d{\mathfrak{\eta}}_1^{\tau_n}(x,t,\mathfrak{\zeta})=\int_{\mathbb{R}^{N}}|A^{\tau_n}(\mathfrak{u^{\tau_n}}(x+z))-A^{\tau_n}(\mathfrak{\zeta})|\mathbbm{1}_{\mbox{Conv}\{{\mathfrak{u^{\tau_n}}}(x+z),{\mathfrak{u^{\tau_n}}}(x)\}}(\mathfrak{\zeta})\mu(z)dz\,d\mathfrak{\zeta}\,dx\,dt.$$
To obtain the convergence of these measures to a kinetic measure, we prove the following lemma.
\begin{lem}\label{m}
	There exists a kinetic measure ${\mathfrak{m}}$ such that
	\[ {\mathfrak{m}}^{\tau_n}\overset{w^*}{\to}{\mathfrak{m}}\,\,\qquad in\,\,\qquad\qquad L_w^2({\Omega}; \mathcal{M}_b(\mathbb{T}^N\times[0,T]\times\mathbb{R}))-weak^* .\]
	 Here, ${\mathfrak{m}}$ can be written as ${\mathfrak{\eta}}_1+{\mathfrak{m}}_1$, where
$$d{\mathfrak{\eta}}_1(x,t,\mathfrak{\zeta})=\int_{\mathbb{R}^N}|A({u}(x+z,t))-A(\mathfrak{\zeta})|\mathbbm{1}_{\mbox{Conv}\{{u}(x+z,t),{u}(x,t)\}}\mu(z)dz dx d\mathfrak{\zeta} dt$$
	and $\mathbb{P}$-almost surely, $\mathfrak{m}_1$ is a nonnegative measure over $\mathbb{T}^N\times[0,T]\times\mathbb{R}$.
	
\end{lem}
\begin{proof}
	 Due to the computations used in the proof of Proposition \ref{p estimate}, we can conclude that $\mathbb{P}$-almost surely, for all $t\in[0,T]$
	\begin{align*}
	\int_0^T\int_{\mathbb{T}^N}\int_{\mathbb{R}}\mathfrak{m}^{\tau_n}(x,t,\mathfrak{\zeta})\,d\mathfrak{\zeta}\,dx\,dt\le\,&\left\Vert u_0\right\Vert_{L^2(\mathbb{T}^N)}^2
	+\sum_{k\ge1}\int_0^T\int_{\mathbb{T}^N}\,\mathfrak{u^{\tau_n}}\,h_k(\mathfrak{u^{\tau_n}})\,dx\,d\beta_k(t)\\&\qquad+\frac{1}{2}\int_0^T\int_{\mathbb{T}^N}\, H^2(\mathfrak{u^{\tau_n}})\,dx ds.
	\end{align*}
	Taking square and expectation and finally by the Ito isometry, we get that
	\begin{align*}
	\mathbb{E}|\mathfrak{m^{\tau_n}}([0,T]\times\mathbb{T}^N\times\mathbb{R})|^2
	\le\,C.
	\end{align*}
Hence, sequence $m^{\tau_n}$ is bounded in $L_w^2({\Omega};\mathcal{M}_b(\mathbb{T}^N\times[0,T]\times\mathbb{R})$, and As a consequence of the Banach-Alaoglu theorem, there exists a $weak^*$ convergent subsequence, still denoted by $\{{\mathfrak{m}}^{\tau_n}\;n\in\mathbb{N}\}$, that is, $\mathfrak{m}^{\tau_n}$ converge to ${\mathfrak{m}}$ in $L_w^2({\Omega};\mathcal{M}_b(\mathbb{T}^N\times[0,T]\times\mathbb{R})$-weak-*.
Now, it is only remaining to show that $weak^*$ limit ${\mathfrak{m}}$ is a kinetic measure. The first point of Definition \ref{kinetic measure 1} of kinetic measure is direct. To prove the behavior for large $\mathfrak{\zeta}$, we need $L^p$-estimate. From \eqref{4.7} we conclude that 
	\begin{align}\label{4.13}
	\mathbb{E}\sup_{t\in[0,T]}\left\Vert \mathfrak{u^{\tau_n}}\right\Vert_{L^p(\mathbb{T}^N)^p}^p+\mathbb{E}\int_0^T\int_{\mathbb{T}^N}\int_{\mathbb{R}} |\mathfrak{\zeta}|^{p-2}&d\mathfrak{m}^{\tau_n}(x,t,\mathfrak{\zeta})\le\,C(1+\mathbb{E}\left\Vert u_0\right\Vert_{L^p(\mathbb{T}^N)}^p).
	\end{align}
	Suppose that $(\chi_{\delta})$ be a truncation on $\mathbb{R}$, then it implies, for $p\in[2,\infty),$
	\begin{align*}
	&\mathbb{E}\int_{[0,T]\times\mathbb{T}^N\times{R}}|\mathfrak{\zeta}|^{p-2}d{\mathfrak{m}}(x,t,\mathfrak{\zeta})\le\,\liminf_{\delta\to0}\mathbb{E}\int_{[0,T]\times\mathbb{T}^N\times\mathbb{R}}|\mathfrak{\zeta}|^{p-2}\chi_{\delta}(\mathfrak{\zeta}) d{\mathfrak{m}}(t,x,\mathfrak{\zeta})\\
	&=\liminf_{\delta\to 0}\liminf_{\tau_n\to\,0}\mathbb{E}\int_{[0,T]\times\mathbb{T}^N\times{R}}|\mathfrak{\zeta}|^{p-2}\chi_{\delta}(\mathfrak{\zeta})d{\mathfrak{m}}^{\tau_n}(t,x,\mathfrak{\zeta})\le\,C.
	\end{align*}
	where the last inequality holds from \eqref{4.13}. It shows that ${\mathfrak{m}}$ vanishes for large $\mathfrak{\zeta}$. Finally, we can deduce that there exist kinetic meaures $\lambda_1,$ and $\lambda_2$ such that
	$${\mathfrak{m}}_1^{\tau_n}\overset{w^*}\to \lambda_1,\,\,\,\,\,{\mathfrak{\eta}}_1^n\overset{w^*}\to\,\lambda_2\,\,\,\,\, in\,\,L_w^2({\Omega};\mathcal{M}_b(\mathbb{T}^N\times[0,T]\times\mathbb{R}) $$
	We observe that $\mathbb{P}$-almost surely, there exists a negligible $\mathcal{S}\subset\,\mathbb{T}^N\times[0,T]$, such that upto subsequence,  $\mathfrak{u^{\tau_n}}(x,t)\to{u}(x,t)$, for all $(t,x)\notin \mathcal{S}$. If we fixed  $z\in\mathbb{R}^N$, then we have also $\mathfrak{u^{\tau_n}}(x+z,t)\to{u}(x+z, t)$ for any $(x,t)$ not in some negligible $\mathcal{S}_z$.
	Hence we have
	$$|A^{\tau_n}(\mathfrak{u^{\tau_n}}(x+z,t))-A^{\tau_n}(\mathfrak{\zeta})|\mathbbm{1}_{\mbox{Conv}\{\mathfrak{u^{\tau_n}}(x,t),\mathfrak{u^{\tau_n}}(x+z,t)\}}(\mathfrak{\zeta})\to|A({u}(x+z,t))-A(\mathfrak{\zeta})|\mathbbm{1}_{\mbox{Conv}\{{u}(x,t),{u}(x+z,t)\}}(\mathfrak{\zeta})\,\,\, $$
	$as\,\,\, n\to \infty,$
	for any $(x,t,\mathfrak{\zeta})\in[0,T]\times\mathbb{T}^N\times\mathbb{R}$ such that $(t,x)\notin \mathcal{S}\cup \mathcal{S}_z$ and $\mathfrak{\zeta}\ne {u}(t,x)$. We can then use Fatou's lemma to conclude that $\mathbb{P}$-almost surely, ${\mathfrak{\eta}}_1\,\le\,\lambda_2$.
	Hence, $\mathbb{P}$-almost surely, ${\mathfrak{m}}\ge\lambda_1+\lambda_2\ge \,{\mathfrak{\eta}}_1.$ 
	\end{proof}
	\subsection*{Kinetic Solution}\label{section 5.5}
As a direct consequence of convergence results, stated in previous subsections and \cite[Lemma 2.1, Proposition 4.9]{sylvain}, we can pass to the limit in all terms of \eqref{approximate formulation}. The convergence of the It\^o integral can easily verify using uniform $L^p$-bound  and strong convergence of $\mathfrak{u^{\tau}}$.\\
For all $\varphi\in C_c^2(\mathbb{T}^N\times\mathbb{R}),$ for ${\mathbb{P}}$-almost surely, there exists a subset $\mathcal{N}_0\subset[0,T]$ of measure zero such that for all $ t\in[0,T]\backslash\mathcal{N}_0$,
	\begin{align}\label{4.17}
	\langle {f}(t),\varphi \rangle &= \langle {f}_0, \varphi \rangle + \int _0^t \langle {f}(s), F'(\mathfrak{\zeta})\cdot\nabla\varphi \rangle ds - \int_0^t \langle {f}(s), A'(\mathfrak{\zeta}) (-\Delta)^\alpha[\varphi] \rangle ds\notag\\
	&\qquad+\sum_{k=1}^\infty \int_0^t\int_{\mathbb{T}^N}\int_{\mathbb{R}} h_k(x,\mathfrak{\zeta}) \varphi (x,\mathfrak{\zeta}) d{\mathfrak{\mathcal{V}}}_{x,s}(\mathfrak{\zeta}) dx ds\notag\\
	&\qquad +\frac{1}{2} \int_0^t \int_{\mathbb{T}^N}\int_{\mathbb{R}} \partial_{\mathfrak{\zeta}} \varphi( x,\mathfrak{\zeta}) H^2(x,\mathfrak{\zeta}) d{\mathfrak{\mathcal{V}}}_{x,s}(\mathfrak{\zeta}) dx ds  - {m}(\partial_{\mathfrak{\zeta}}\varphi)([0,t]),
	\end{align}
	We state the following proposition which help to extend kinetic formulation for all $t\in[0,T]$. 
\begin{proposition}\label{th4.15}
	There exists a measurable subset $\Omega_2$ of ${\Omega}$ of probability 1, and a random Young measure ${\mathfrak{\mathcal{V}}}^+$ on $\mathbb{T}^N\times(0, T)$ such that the following properties holds:
	\begin{enumerate}
		\item[1.] for all ${\omega}\in\Omega_2$, for almost every $(x,t)\in\mathbb{T}^N\times(0,T)$, the probability measures ${\mathfrak{\mathcal{V}}}_{x,t}^+$  and ${\mathfrak{\mathcal{V}}}_{x,t}$ are equal.
		\item[2.] the kinetic function ${f}^+(x,t,\mathfrak{\zeta}):= {\mathfrak{\mathcal{V}}}_{x,t}^+(\mathfrak{\zeta},+\infty)$ satisfies: for all ${\omega}\in\Omega_2$, for all $\varphi\in C_c(\mathbb{T}^N\times(0,T))$, $t\mapsto \langle {f}^+(t), \varphi \rangle$ is c\'adl\'ag.
		\item[3.] The random Young measure ${\mathfrak{\mathcal{V}}}^+ $ satisfies
		$${\mathbb{E}}\sup_{t\in[0,T]}\int_{\mathbb{T}^N}\int_{\mathbb{R}}|\mathfrak{\zeta}|^p d{\mathfrak{\mathcal{V}}}_{x,t}^+(\mathfrak{\zeta}) dx\,\,\le\,\, C_p.$$
	\end{enumerate}
\begin{proof}
	For a proof, we refer to \cite{sylvain} in which \eqref{4.11} is helpful to get this result.
\end{proof}
\end{proposition}
\noindent
Here, we consider only the c\'adl\'ag version. We have to replace ${\mathfrak{\mathcal{V}}}$ by ${\mathfrak{\mathcal{V}}}^+$ and ${f}$ by ${f}^+$ (by help of It\^o isometry for It\^o integral) in \eqref{4.17} to conclude that the kinetic formulation is true for all t. That is, $\mathbb{P}$-almost surely, 
for all $t\in[0,T]$, for all $\varphi\in C_c^2(\mathbb{T}^N\times\mathbb{R})$ 
\begin{align}\label{4.26}
\langle {f}^+(t),\varphi \rangle &= \langle {f}_0, \varphi \rangle + \int _0^t \langle {f}^+(s), F'(\mathfrak{\zeta})\cdot\nabla\varphi \rangle - \int_0^t \langle {f}^+(s), A'(\mathfrak{\zeta}) (-\Delta)^\alpha[\varphi] \rangle ds\notag\\
&\qquad +\sum_{k=1}^\infty \int_0^t\int_{\mathbb{R}}\int_{\mathbb{T}^N} h_k(x,\mathfrak{\zeta}) \varphi (x,\mathfrak{\zeta}) d{\mathfrak{\mathcal{V}}}^+(\mathfrak{\zeta}) dx ds\notag\\
&\qquad +\frac{1}{2} \int_0^t\int_{\mathbb{R}} \int_{\mathbb{T}^N} \partial_{\mathfrak{\zeta}} \varphi( x,\mathfrak{\zeta} ) H^2(x,\mathfrak{\zeta}) d{\mathfrak{\mathcal{V}}}^+(\mathfrak{\zeta}) dx ds  - {m}(\partial_{\mathfrak{\zeta}}\varphi)([0,t])
\end{align}
$\mathbb{P}$-almost surely.

\noindent
Here, we do not want formulation of kinetic solution in the form in which ${f}^+$ is only kinetic function. But $f^+$ should be equilibrium for all $t\in[0,T]$. We have only  $\mathbb{P}$-almost surely ${f}^+(x,t,\mathfrak{\zeta})=\mathbbm{1}_{{u}(x,t)\textgreater\mathfrak{\zeta}}$ for all $t\in[0,T]\backslash \mathcal{N}_0$, almost $(x,\mathfrak{\zeta})\in\mathbb{T}^N\times\mathbb{R}$. To show that ${f}^+$ has equilibrium form for all $t\in[0,T]$, we can repeat the proof of Theorem \ref{th3.6} with ${f}^+$ and kinetic measure ${\mathfrak{m}}$. Then we get, ${\mathbb{P}}$-almost surely, for all $t\in[0,T]$,
$${f}^+(x,t,\mathfrak{\zeta})=\mathbbm{1}_{\tilde{u}(x,t)\textgreater\mathfrak{\zeta}}$$
where 
$$\tilde{u}(x,t)=\int_{\mathbb{R}} ({f}^+(x,t,\mathfrak{\zeta})-\mathbbm{1}_{0\textgreater\mathfrak{\zeta}})\,\,d\mathfrak{\zeta}.$$
Since ${\mathbb{P}}$-almost surely, $u(t)=\tilde{u}(t)$ almost $t\in [0,T]$. We can modify $u$ as $\mathbb{P}$-almost surely, ${u}(t)=\tilde{u}(t)$ for all $t\in[0,T]$.
Finally, ${(u(t))_{t\in[0,T]}}$ and $f(t)=\mathbbm{1}_{u(t)\textgreater\mathfrak{\zeta}}$ (after redefine) satisfies the all points of definition of kinetic solution. It shows that ${(u(t))_{t\in[0,T]}}$ is kinetic solution.
\subsection*{Existence for general initial data}\label{section 6}
Here, we provide an existence proof for the general case, for all $p\in[1,+\infty)$, $u_0\in L^p(\Omega; L^p(\mathbb{T}^N))$. Let $u^n_0$ be a sequence in $L^p(\Omega; C^\infty(\mathbb{T}^N))$ such that sequence $u^n_0$ converges to $u_0$ in $L^1(\Omega\times\mathbb{T}^N)$. That is, the initial conditions $u_0^n$ can be defined as a pathwise molification of $u_0$ then the following bound holds:
\begin{align}\label{4.27}
\|u_0^n\|_{L^p(\Omega; L^p(\mathbb{T}^N))}\,\le\,\|u_0\|_{L^p(\Omega; L^p(\mathbb{T}^N))}
\end{align}
It is clear from the previous case that for each $n\in\mathbb{N}$, there exists a kinetic solution $u^n$ to \eqref{1.1} with initial data $u_0^n$. From the contraction principle, we have for all $t\in[0,T],$
\[\mathbb{E}\|u^{n_1}(t)-u^{n_2}(t)\|_{L^1(\mathbb{T}^N)}\,\le\,\mathbb{E}\|u_0^{n_1}-u_0^{n_2}\|_{L^1(\mathbb{T}^N)},\,\,\text{for all}\,n_1, n_2 \in\mathbb{N}.\]
By \eqref{4.27} and \eqref{lp estimate}, we have the following uniform estimates, $p\in[1,+\infty)$, 
\begin{align}\label{4.28}
\mathbb{E}\,\big[\sup_{0\le t\le T}\|u^n(t)\|_{L^p(\mathbb{T}^N)}^p\big]\,\le\,C_{T}\mathbb{E}\big[\|u_0\|_{L^p(\mathbb{T}^N)}^p\big]
\end{align}
and
\begin{align*}
\mathbb{E}|m^n(\mathbb{T}^N\times[0,T]\times\mathbb{R})|^2\,\le\,C_{T,u_0}.
\end{align*}
Thus, from the observations of previous subsections, we can show that there exists a subsequence $u^{n_k}$ such that
\begin{enumerate}
	\item[A.] there exists $u\in L^p_{\mathcal{P}}(\Omega\times[0,T];L^1(\mathbb{T}^N))$ such that $u^{n_k}\,\to u$ in $ L^p_{\mathcal{P}}(\Omega\times[0,T];L^1(\mathbb{T}^N))$ as $n_k\,\to\,\infty$.
	\item[B.] $f^{n_k}\,\overset{w^*}{\to}f=\mathbbm{1}_{u \textgreater\mathfrak{\zeta}}$\,\,in\,\, $L^\infty(\Omega\times\mathbb{T}^N\times[0,T]\times\mathbb{R})$-weak-star,
	\item[C.] there exists a kinetic measure $\mathfrak{m}$ such that
	\[m^{n_k}\overset{w^*}{\to}\,\mathfrak{m}\qquad\,in\qquad L_w^2(\Omega; \mathcal{M}_b(\mathbb{T}^N\times[0,T]\times\mathbb{R}))-weak^*\]
	and $\mathfrak{m}\ge\mathfrak{\eta}_1$, $\mathbb{P}$-almost surely.
\end{enumerate}
With these information, we can pass the limit in \eqref{2.6} and conclude that $u$ is the kinetic solution to \eqref{1.1}. This finishes the proof of Theorem \ref{th2.10}.

\section{invariant measure: Proof of Theorem \ref{Invariant measure}}\label{section 4}
In this section, we prove existence and uniquenes of invariant measure. In case of existence and uniqueness of an invariant measure, we focus only for additive noise and $\alpha\in (0, \frac{1}{2})$.
\subsection{Existence of invariant measure} \label{subsection 4.1}
In this subsection, we prove the existence of invariant measures as an application of the Krylov-Bogoliubov theorem. There are many ways to prove the existence of invariant measures. One of them is the Krylov-Bogoliubov approach, as we discussed in this article. 

\noindent
\textbf{Existence:}
Let $u_0\in L^3(\mathbb{T}^N)$ and $u$ be solution to \eqref{1.1} starting from $u_0$.
Let $\beta,\,\,\delta,\,\,\theta\,\,\textgreater\,0$ be some positive parameters.  On $L^2(\mathbb{T}^N)$, we consider the following regularization of the operators: Consider the kinetic formulation of equation \eqref{1.1}, $\mathbb{P}$-almost surely,
\begin{align}\label{fomulation in distribution sense}
\partial_t f +\,\big(F'(\mathfrak{\zeta})\cdot\nabla + A'(\mathfrak{\zeta})\big(-\Delta\big)^\alpha\big)f =\partial_{\mathfrak{\zeta}}\big(\mathfrak{m}-\frac{1}{2}H^2 \delta_{u=\mathfrak{\zeta}}\big)+\varphi(x)\delta_{u=\mathfrak{\zeta}} \partial_t W
\end{align}
where $f=\mathbbm{1}_{u\textgreater\mathfrak{\zeta}}-\mathbbm{1}_{0\textgreater\,\mathfrak{\zeta}}$.
In order to handle the two measures: $\mu_1:=\mathfrak{m}-\frac{1}{2}H^2\delta_{u=\mathfrak{\zeta}}$ and $\mu_2:= \varphi(x)\delta_{u=\mathfrak{\zeta}}$$\partial_t{W}$, we need to regularize the operators. We define
$$A_\theta:= -F'(\mathfrak{\zeta})\cdot\nabla-B_\theta,\,\,\,A_{\theta,\delta}=A_\theta-\delta Id,\,\,\,\,B_\theta=C_\theta+ A'(\mathfrak{\zeta})\big(-\Delta\big)^\alpha,\,\,\, C_\theta=\theta\big(-\Delta\big)^\beta.$$
Let $S_{A_\theta}(t)$ and $S_{A_{\theta,\delta}}(t)$ be also the associated semi-groups on $L^2(\mathbb{T}^N)$:
$$S_{A_\theta}(t)\kappa(x)=\big(\e ^{-tB_{\theta}}\kappa(x)\big)(x-F'(\mathfrak{\zeta})t),\,\,\,\,\,\,\,S_{A_{\theta,\delta}}(t)\kappa(x)= \e^{-\delta t}\big(\e^{t\,B_\theta}\kappa\big)(x-F'(\mathfrak{\zeta})t).$$
Adding $\theta(-\Delta)^\beta+\delta Id$ to each side to \eqref{fomulation in distribution sense}:
\begin{align}
(\partial_t - A_{\theta,\delta})f =\big(C_\theta+\delta Id\big)f+\partial_{\mathfrak{\zeta}}\big(\mu_1\big)+\mu_2.
\end{align}
Here, we adapt the semi group approach(cf.\cite[Proposition 6.3 ]{prato}). We can express to the kinetic formulation in the mild formulation as $\mathbb{P}$-almost surely,
\begin{align}
f(x,t,\mathfrak{\zeta})=&S_{A_{\theta,\delta}}(t) f_0(x,\mathfrak{\zeta})+\int_0^t S_{A_{\theta,\delta}}(s)\big(C_\theta f(x,s,\mathfrak{\zeta})+\delta f(x,s,\mathfrak{\zeta}) \big)f(x,\mathfrak{\zeta},t-s )ds\notag\\&\qquad+\int_0^t S_{A_{\theta,\delta}}(t-s)\partial_{\mathfrak{\zeta}}\mu_1(x,s,\mathfrak{\zeta}) ds
+\int_0^t S_{A_{\theta,\delta}}(t-s)\varphi(x) \delta_{u=\mathfrak{\zeta}} \,dW(s). 
\end{align}
We deduce the  following decompsition of the solution:
\begin{align}\label{u}
u=u^0+u^b+P+Q
\end{align}
where
\begin{align}
u^0(x,t)&=\int_{\mathbb{R}} S_{A_{\theta,\delta}}(t)f(\mathfrak{\zeta},0)d\mathfrak{\zeta},\\
u^b(x,t)&=\int_{\mathbb{R}}\int_0^t S_{A_{\theta,\delta}}(s)(C_\theta f+\delta f)(\mathfrak{\zeta},t-s) ds d\mathfrak{\zeta},\\
\langle P(t),\varphi\rangle&=\sum_{k\ge\,1}\int_{\mathbb{T}^N}\int_0^t \bigg(S^*_{A_{\theta,\delta}}(t-s)\varphi\bigg)(x,u(x,s))h_k(x)d\beta_k(s)dx.\\
&\text{and}\\
\langle Q(t), \varphi \rangle&=\int_0^t\int_{\mathbb{T}^N}\langle \partial_\mathfrak{\zeta} \mu_1(x,\cdot,t-s),S^*_{A_{\theta,\delta}}(s)\varphi\rangle_{\mathfrak{\zeta}}dx ds.
\end{align}
where $\varphi\in C(\mathbb{T}^N)$ and $S^*_{A_{\theta,\delta}}(t)$ is the dual operator of the semigroup operator $S_{A_{\theta\,\delta}}(t).$
We will estimate each term separately.  
\subsection*{ Estimate on $u^0$:}
Here we use the Fourier transform with respect for $x\in\mathbb{T}^N$, that is,
$$\hat{v}(j)=\int_{\mathbb{T}^N}v(x)\,e^{-2\pi i j\cdot x}dx,\,\,\,\,\,\,j\in\mathbb{Z}^N,\,\,\,\,v\in L^1(\mathbb{T}^N).$$
After Fourier transform, for all $j\in\mathbb{Z}^N,\,\,j\neq0,$
\begin{align*}\hat{v}(j,t)&=\int_{\mathbb{R}}\e^{-\big(iF'(\mathfrak{\zeta})\cdot j+A'(\mathfrak{\zeta})|j|^{2\alpha}+\theta|j|^{2\beta}+\delta\big)t} \hat{f}_0(j,\mathfrak{\zeta})\, d\mathfrak{\zeta}\\
&=\int_{\mathbb{R}}\e^{-\big(iF'(\mathfrak{\zeta})\cdot \hat{j}\,|j|+A'(\mathfrak{\zeta})|j|^{2\alpha-1}|j|+\omega_j|j|\big)t}\hat{f}_0(j,\mathfrak{\zeta})\,d\mathfrak{\zeta},
\end{align*}
where $\omega_j=\theta|j|^{2\beta-1}+\delta|j|^{-1}$ and $\hat{j}=\frac{j}{|j|}$.\\
Then, for $T\,\ge\,0,\,\,\,j\neq0,$ we get 
\begin{align*}
\int_0^T|\hat{u}^0(j,t)|^2 dt&=\int_0^T \bigg|\int_{\mathbb{R}}\e^{-\big(iF'(\mathfrak{\zeta})\cdot \hat{j}\,|j|+A'(\mathfrak{\zeta})|j|^{2\alpha-1}|j|+\omega_j|j|\big)t}\hat{f}_0(j,\mathfrak{\zeta})\,d\mathfrak{\zeta}\bigg|^2 dt\\
&\le\int_{\mathbb{R}^+} \bigg|\int_{\mathbb{R}}\e^{-\big(iF'(\mathfrak{\zeta})\cdot \hat{j}\,|j|+A'(\mathfrak{\zeta})|j|^{2\alpha-1}|j|+\omega_j|j|\big)t}\hat{f}_0(j,\mathfrak{\zeta})\,d\mathfrak{\zeta}\bigg|^2 dt\\
&\le\frac{1}{|j|}\int_{\mathbb{R}^+} \bigg|\int_{\mathbb{R}}\e^{-\big(iF'(\mathfrak{\zeta})\cdot \hat{j}\,+A'(\mathfrak{\zeta})|j|^{2\alpha-1}+\omega_j\big)s}\hat{f}_0(j,\mathfrak{\zeta})\,d\mathfrak{\zeta}\bigg|^2 ds.
\end{align*}
Here we thank to the change of variable $s=|j|t$ for last inequality.
To estimate right hand side, we want to use Fourier transformation in time variable as follows: 
\begin{align}
\int_0^T|\hat{u}^0(j,t)|^2 dt&\le\frac{1}{|j|}\int_{\mathbb{R}^{+}} \bigg|\int_{\mathbb{R}}\e^{-iF'(\mathfrak{\zeta})\cdot \hat{j}s}\,\e^{-\big(A'(\mathfrak{\zeta})|j|^{2\alpha-1}+\omega_j\big)|s|}\hat{f}_0(j,\mathfrak{\zeta})\,d\mathfrak{\zeta}\bigg|^2 ds\notag\\
&\le\frac{1}{|j|}\int_{\mathbb{R}}\bigg|\int_{\mathbb{R}}\e^{-iF'(\mathfrak{\zeta})\cdot \hat{j}s}\,\e^{-\big(A'(\mathfrak{\zeta})|j|^{2\alpha-1}+\omega_j\big)|s|}\hat{f}_0(j,\mathfrak{\zeta})\,d\mathfrak{\zeta}\bigg|^2 ds\notag\\
&=\frac{1}{|j|}\int_{\mathbb{R}}\bigg|\mathcal{F}\bigg\{\int_{\mathbb{R}}\e^{-iF'(\mathfrak{\zeta})\cdot \hat{j}s}\,\e^{-\big(A'(\mathfrak{\zeta})|j|^{2\alpha-1}+\omega_j\big)|s|}\hat{f}_0(j,\mathfrak{\zeta})\,d\mathfrak{\zeta}\bigg\}(r)\bigg|^2dr
\end{align}
We also know that
\begin{align*}\mathcal{F}\bigg\{\e^{-iF'(\mathfrak{\zeta})\cdot \hat{j}s}\,\e^{-\big(A'(\mathfrak{\zeta})|j|^{2\alpha-1}+\omega_j\big)|s|}\bigg\}(r)&=\int_{\mathbb{R}}\e^{-iF'(\mathfrak{\zeta})\cdot \hat{j}s}\,\e^{-\big(A'(\mathfrak{\zeta})|j|^{2\alpha-1}+\omega_j\big)|s|}\e^{-irs}ds\\
&=\int_{\mathbb{R}}\e^{-\big(A'(\mathfrak{\zeta})|j|^{2\alpha-1}+\omega_j\big)|s|}\e^{-i\big(F'(\mathfrak{\zeta})\cdot\hat{j}+r\big)s}ds\\
&=\mathcal{F}\bigg\{e^{-\big(A'(\mathfrak{\zeta})|j|^{2\alpha-1}+\omega_j\big)|s|}\bigg\}(F'(\mathfrak{\zeta})\cdot\hat{j}+r)\\
&=\frac{2\big(A'(\mathfrak{\zeta})|j|^{2\alpha-1}+\omega_j\big)}{\big(A'(\mathfrak{\zeta})|j|^{2\alpha-1}+\omega_j\big)^2+|F'(\mathfrak{\zeta})\cdot\hat{j}+r|^2}.
\end{align*}
We use the above identity to conclude that
\begin{align*}
\int_0^T|\hat{u}^0(j,t)|^2 dt&\le\frac{1}{|j|}\int_{\mathbb{R}}\bigg|\int_{\mathbb{R}}\frac{2\big(A'(\mathfrak{\zeta})|j|^{2\alpha-1}+\omega_j\big)}{\big(A'(\mathfrak{\zeta})|j|^{2\alpha-1}+\omega_j\big)^2+|F'(\mathfrak{\zeta})\cdot\hat{j}+r|^2} \hat{f}_0(\mathfrak{\zeta},j)d\mathfrak{\zeta}\bigg|^2 dr\\
&\le\,\frac{4}{|j|}\int_{\mathbb{R}}\bigg(\int_{\mathbb{R}_\mathfrak{\zeta}}|\hat{f}_0(\mathfrak{\zeta},j)|^2\frac{\big(A'(\mathfrak{\zeta})|j|^{2\alpha-1}+\omega_j\big)}{\big(A'(\mathfrak{\zeta})|j|^{2\alpha-1}+\omega_j\big)^2+|F'(\mathfrak{\zeta})\cdot\hat{j}+r|^2}d\mathfrak{\zeta}\bigg)\\
&\qquad\qquad\qquad\qquad\qquad\times\bigg(\int_{\mathbb{R}_\mathfrak{\zeta}}\frac{\big(A'(\mathfrak{\zeta})|j|^{2\alpha-1}+\omega_j\big)}{\big(A'(\mathfrak{\zeta})|j|^{2\alpha-1}+\omega_j\big)^2+|F'(\mathfrak{\zeta})\cdot\hat{j}+r|^2}d\mathfrak{\zeta}\bigg)dr\\
&\le\frac{4}{|j|\omega_j}\sup_r\int_\mathbb{R}\frac{\omega_j (A'(\mathfrak{\zeta})|j|^{2\alpha-1}+\omega_j)}{(A'(\mathfrak{\zeta})|j|^{2\alpha-1}+\omega_j)^2+|F'(\mathfrak{\zeta})\cdot\hat{j}+r|^2}\,d\mathfrak{\zeta} \\
&\qquad\qquad\times\int_{\mathbb{R}}|\hat{f}_0(\mathfrak{\zeta},j)|^2\bigg(\int_\mathbb{R}\frac{(A'(\mathfrak{\zeta})|j|^{2\alpha-1}+\omega_j)}{(A'(\mathfrak{\zeta})|j|^{2\alpha-1}+\omega_j)^2+|F'(\mathfrak{\zeta})\cdot\hat{j}+r|^2}dr\bigg)d\mathfrak{\zeta}\\
&\le\,\frac{4\mathfrak{\eta}(\omega_j)}{|j|\omega_j}\times\pi\times\int_{\mathbb{R}}|\hat{f}_0(\mathfrak{\zeta},j)|^2 d\mathfrak{\zeta}\\
&\le\,\frac{4\pi}{|j|\omega_j^{1-s}}\int_{\mathbb{R}}|\hat{f}_0(\mathfrak{\zeta},j)|^2 d\mathfrak{\zeta},
\end{align*}
where second last inequality due to \eqref{non}. It implies that
\begin{align*}
|j|\omega_j^{1-s}\int_0^T|\hat{u}^0(j,t)|^2 dt&\le 4\pi \int_{\mathbb{R}}|\hat{f}_0(\mathfrak{\zeta},j)|^2 d\mathfrak{\zeta}.
\end{align*}
It is easy to see that
$$|j|\omega_j^{1-s}\,\ge\,\theta^{1-s}|j|^{(2\beta-1)(1-s)+1}$$
It implies that
\begin{align}\label{summing}
\int_0^T|j|^{2\{(\beta-\frac{1}{2})(1-s)+\frac{1}{2}\}}|\hat{u}^0(j,t)|^2 dt \le 4\pi \theta^{s-1}\int_\mathbb{R}|\hat{f}_0(\mathfrak{\zeta},k)|^2 d\mathfrak{\zeta}
\end{align}
Observe that
\begin{align*}
\hat{u}^0(0,t)&=\int_{\mathbb{T}^N} u^0(t,x)dx=\int_{\mathbb{R}}\int_{\mathbb{T}^N}\e^{-\delta t}f(0,x,\mathfrak{\zeta})d\mathfrak{\zeta} dx=\int_{\mathbb{T}^N}\e^{-\delta t}u_0(x) dx=0,
\end{align*}
Summing over $j\in\mathbb{Z}^N$ in \eqref{summing}, we have
\begin{align}
\int_0^T\|u^0(t)\|_{H^{\beta+\big(\frac{1}{2}-\beta\big)s}}dt\,\le\,4\pi\,\theta^{s-1}\|f_0\|_{L_{x,\mathfrak{\zeta}}^2}^2=\,4\pi\theta^{s-1}\|u_0\|_{L_x^1}.
\end{align}
\subsection*{Estimate on $u^b$:}
Recall that
$$u^b(x,t)=\int_{\mathbb{R}}\int_0^t S_{A_{\theta,\delta}}(s)(C_\theta f+\delta f)(\mathfrak{\zeta},t-s) ds d\mathfrak{\zeta}$$
Here, we use the same argument as above to get estimate on the Fourier transform of $u^b$ for $j\neq0$ as follows:
\begin{align*}
\hat{u}^b(j,t)&=\int_0^t\int_{\mathbb{R}}\big(\theta|j|^{2\alpha}+\delta\big)\e^{-s\big(iF'(\mathfrak{\zeta})\cdot \hat{j}\,|j|+A'(\mathfrak{\zeta})|j|^{2\alpha-1}|j|+\omega_j|j|\big)}\hat{f}(j,\mathfrak{\zeta},t-s)d\mathfrak{\zeta} ds\\
&=\int_0^t\int_{\mathbb{R}}|j|\omega_j \e^{-s|j|\big(iF'(\mathfrak{\zeta})\cdot\hat{j}+A'(\mathfrak{\zeta})|j|^{2\alpha-1} +\,\omega_j\big)}\hat{f}(x,\mathfrak{\zeta},t-s) d\mathfrak{\zeta} ds.
\end{align*}
It implies that
\begin{align*}
\int_0^T|\hat{u}^b(x&,j)^2 dt=\int\limits_0^T\bigg|\int\limits_0^t\int\limits_{\mathbb{R}}|j|\omega_j \e^{-s|j|\big(iF'(\mathfrak{\zeta})\cdot\hat{j}+A'(\mathfrak{\zeta})|j|^{2\alpha-1} +\,\omega_j\big)}\hat{f}(x,\mathfrak{\zeta},t-s) d\mathfrak{\zeta} ds\bigg|^2dt\\
&=\,\int\limits_0^T\bigg|\int\limits_0^t\int\limits_{R}\big(\sqrt{|j|\omega_j}\e^{\frac{-s|j|\omega_j}{2}}\big)\big(\sqrt{|j|\omega_j}\e^{-s|j|\big(iF'(\mathfrak{\zeta})\cdot\hat{j}+A'(\mathfrak{\zeta})|j|^{2\alpha-1} +\,\frac{\omega_j}{2}\big)}\big)\hat{f}(x,\mathfrak{\zeta},t-s)d\mathfrak{\zeta} ds\bigg|^2 dt\\
&=\,\int\limits_0^T\bigg|\int\limits_0^t\big(\sqrt{|j|\omega_j}\e^{\frac{-s|j|\omega_j}{2}}\big)\bigg(\int_{R}\sqrt{|j|\omega_j}\e^{-s|j|\big(iF'(\mathfrak{\zeta})\cdot\hat{j}+A'(\mathfrak{\zeta})|j|^{2\alpha-1} +\,\frac{\omega_j}{2}\big)}\hat{f}(x,\mathfrak{\zeta},t-s)d\mathfrak{\zeta}\bigg) ds\bigg|^2 dt\\
&\le\,\int\limits_0^t|j|\omega_j e^{-s|j|\omega_js} ds\int\limits_0^T\int_0^t\bigg|\int\limits_{R}\sqrt{|j|\omega_j}\e^{-s|j|\big(iF'(\mathfrak{\zeta})\cdot\hat{j}+A'(\mathfrak{\zeta})|j|^{2\alpha-1} +\,\frac{\omega_j}{2}\big)}d\mathfrak{\zeta}\bigg)\hat{f}(x,\mathfrak{\zeta}, t-s) \bigg|^2 dsdt\\
&=\,\int\limits_0^\infty|j|\omega_j e^{-s|j|\omega_js} ds\int\limits_0^T\int\limits_0^T\bigg|\int_{R}\sqrt{|j|\omega_j}\e^{-s|j|\big(iF'(\mathfrak{\zeta})\cdot\hat{j}+A'(\mathfrak{\zeta})|j|^{2\alpha-1} +\,\frac{\omega_j}{2}\big)}d\mathfrak{\zeta}\bigg)\hat{f}(x,\mathfrak{\zeta},t) \bigg|^2 dsdt
\end{align*}
Now,
$$\int_0^\infty|j|\omega_j e^{-s|j|\omega_j} ds=\int_0^\infty\e^{-s}ds=1,$$
and by similar computations as the previous subsection, we obtain
\begin{align*}
\int_0^T\bigg|\int_{R}\sqrt{|j|\omega_j}\e^{-s|j|\big(iF'(\mathfrak{\zeta})\cdot\hat{j}+A'(\mathfrak{\zeta})|j|^{2\alpha-1} +\,\frac{\omega_j}{2}\big)}d\mathfrak{\zeta}\bigg)\hat{f}(x,\mathfrak{\zeta},t) \bigg|^2 ds\le\frac{4\pi \omega_j^s}{2^s}\|u(t)\|_{L_x^1}
\end{align*}
\begin{align}
\int_0^T\omega_j^s|\hat{u}^b(x,j)|^2 dt\le\frac{4\pi }{2^s}\int_0^T\|u(t)\|_{L_x^1}\notag
\end{align}
\begin{align}
\int_0^T |j|^{(1-2\beta)s}|\hat{u}^b(x,j)|^2 dt\le\frac{4\pi \theta^s}{2^s}\int_0^T\|u(t)\|_{L_x^1}\notag
\end{align}
\begin{align}
\int_0^T\|u^b(t)\|^2_{H^{\big(\frac{1}{2}-\beta\big)s}}dt\le\frac{4\pi\theta^s}{2^s}\int_0^T\|u(t)\|_{L_x^1}dt.
\end{align}
\subsection*{Estimate on P:}
We recall that $P$ is random measure on $\mathbb{T}^N$ given by:
$$\langle P(t),\varphi\rangle=\sum_{k\ge\,1}\int_{\mathbb{T}^N}\int_0^t \bigg(S^*_{A_{\theta,\delta}}(t-s)\varphi\bigg)(x,u(x,s))h_k(x)d\beta_k(s)dx,\,\,\,\,\,\varphi\in C(\mathbb{T}^N).$$
Note that, for each $k\in\mathbb{N}$,
\begin{align*}
\int_{\mathbb{T}^N}\int_0^t \bigg(S^*_{A_{\theta,\delta}}(t-s)&\varphi\bigg)(x,u(x,s))h_k(x)d\beta_k(s)dx\\
&=\int_{\mathbb{T}^N}\int_0^t\, \e^{-\delta(t-s)}\big(\e^{-B_\theta(t-s)}\varphi(x)\big)(x+F'(u(x,s)))h_k(x) d\beta_k(s) dx.\\
&=\int_{\mathbb{T}^N} \int_0^t\e^{-\delta(t-s)}\varphi(x)\e^{-B_\theta(t-s)}h_k(x-F'(u(x,s)))d\beta_k(s)dx
\end{align*}
 Let set $p(x,\mathfrak{\zeta})=\delta_{u=\mathfrak{\zeta}} dx,$
then we define the measure
\begin{align*}
\langle P'(s), \varphi\rangle&=\int_{\mathbb{T}^N\times{\mathbb{R}}}\e^{-\delta(t-s)}\varphi(x)\e^{-tB_\theta(t-s)}h_k(x-F'(\mathfrak{\zeta})) dp(x,\mathfrak{\zeta})
\end{align*}
where $\varphi\in L^2(\mathbb{T}^N)$. 
By \eqref{noise1}, the mapping $x\mapsto h_k(x)$ is a bounded function. therefore by \cite[Lemma 9]{vovelle 2}, we can write for $\lambda\,\ge\,0$:
\begin{align*}
\|\big(-\Delta\big)^\frac{\lambda}{2} e^{-tB_\theta(t-s)}h_k(\cdot-F'(\mathfrak{\zeta}))\|_{L^\infty(\mathbb{T}^N)}&=\|\big(-\Delta\big)^{\frac{\lambda}{2}}\e^{-tC_\theta-tA'(\mathfrak{\zeta})\big(-\Delta\big)^\alpha }\big(h_k(\cdot-F'(\mathfrak{\zeta})(t-s))\big)\|_{L^\infty(\mathbb{T}^N)}\\
&=\|\big(-\Delta\big)^{\frac{\lambda}{2}}\e^{-tC_\theta}\bigg( \e^{-tA'(\mathfrak{\zeta})\big(-\Delta\big)^\alpha }\big(h_k(\cdot-F'(\mathfrak{\zeta})(t-s))\big)\bigg)\|\\
&\le\,c\big(\theta(t-s)\big)^{\frac{-\lambda}{2\beta}}\|\e^{-tA'(\mathfrak{\zeta})\big(-\Delta\big)^\alpha }\big(h_k(\cdot-F'(\mathfrak{\zeta})(t-s))\big)\|_{L^\infty(\mathbb{T}^N)}\\
&\le\,c\big(\theta(t-s)\big)^{\frac{-\lambda}{2\beta}}\times C_1 \|h_k\|_{C(\mathbb{T}^N)}\\
&\le\,C\big(\theta(t-s)\big)^{\frac{-\lambda}{2\beta}}\e^{-\delta(t-s)}\|h_k\|_{C(\mathbb{T}^N)}.
\end{align*}
\begin{align*}
\langle \big(-\Delta\big)^{\lambda}P'(s),\varphi\rangle&=\int_{\mathbb{T}^N\times{\mathbb{R}}}\e^{-\delta(t-s)}\varphi(x)\big(-\Delta\big)^\lambda e^{-tB_\theta(t-s)}h_k(x-F'(\mathfrak{\zeta})) dp(x,\mathfrak{\zeta})\\
&\le\,\e^{-\delta(t-s)}\bigg(\int_{\mathbb{T}^N\times\mathbb{R}}|\big(-\Delta\big)^\lambda e^{-tB_\theta(t-s)}h_k(x-F'(\mathfrak{\zeta}))|^2 dp(x,\mathfrak{\zeta})\bigg)^\frac{1}{2}\,\|\varphi\|_{L^2(\mathbb{T}^N)},\\
&\le\,C\big(\theta(t-s)\big)^{\frac{-\lambda}{2\beta}}\e^{-\delta(t-s)}\|h_k\|_{C(\mathbb{T}^N)}\,\|\varphi\|_{L^2(\mathbb{T}^N)}.
\end{align*}
It implies that
\begin{align*}
\|P'(s)\|_{H^\lambda(\mathbb{T}^N)}\le\,C\,\big(\theta(t-s)\big)^{\frac{-\lambda}{2\beta}}\e^{-\delta(t-s)}\|h_k\|_{C(\mathbb{T}^N)}.
\end{align*}
We deduce that, for $\lambda\in[0,\beta)$, this defines a function in $H^\lambda(\mathbb{T}^N)$ and
\begin{align*}
\mathbb{E}\bigg(\bigg\|\int_0^t \e^{-\delta(t-s)}&\e^{-tB_\theta(t-s)}\big(h_k(\cdot,-F'(u(\cdot,s)))d\beta_k(s)\big)\bigg\|_{H^\lambda(\mathbb{T}^N)}^2\bigg)\\
&=\mathbb{E}\int_0^t \|P'(s)\|_{H^\lambda(\mathbb{T}^N)}^2 ds\\
&\le\,C\int_0^t\e^{-2\delta(t-s)}(\theta(t-s))^{\frac{-\lambda}{\beta}}\|h_k\|_{C(\mathbb{T}^N)}^2ds\\
&\le\,C\,\theta^{-\frac{\lambda}{\beta}}\delta^{\frac{\lambda}{\beta}-1}\bigg(\int_{\mathbb{R}^+}\e^{2s}\,s^{-\frac{\lambda}{\beta}}ds\bigg)\,\|h_k\|_{C(\mathbb{T}^N)}^2
\end{align*}
Since the sum over $k$ of the right hand side is finite. We conclude
\begin{align*}
\mathbb{E}\bigg(\bigg\|&\sum_{k\ge1}\int_0^t\e^{B_\theta (t-s)}\big(h_k(\cdot-F'(u(\cdot,s))(t-s))\big)d\beta (s)\bigg\|_{H^\lambda(\mathbb{T}^N)}^2\bigg)\\
&\le C\theta^{-\frac{\lambda}{\alpha}-1}(\int_0^{+\infty}\e^{-2s} s^{-\frac{\lambda}{\alpha}} ds) C_0.
\end{align*}
In fact, this shows that $P(t)$ is more regular than a measure. It is in the space $L^2(\Omega;H^\lambda(\mathbb{T}^N))$ for $\lambda\,\textless\,\alpha$ and satisfying the following estimate
\begin{align}
\mathbb{E}\big(\|P(t)\|_{H^\lambda(\mathbb{T}^N)}^2\big)\,\le\,C\,\theta^{\frac{-\lambda}{\alpha}}\delta^{\frac{\lambda}{\alpha}-1}\big(\int_0^{+\infty}\e^{-2s}\,s^{\frac{-\lambda}{\alpha}}ds\big)C_0.
\end{align}

\subsection*{Bound on the Kinetic measure:}
\begin{lem}
	Let $u:\mathbb{T}^N\times[0,T]\times\Omega\to\mathbb{R}$ be the solution to \eqref{1.1} with initial data $u_0$. Then the measure $\mu_1$ satisfies
	\begin{align}
	\mathbb{E} \int_{\mathbb{T}^N\times[0,T]\times\mathbb{R}} \theta_1(\mathfrak{\zeta}) d|\mu_1|(x,t,\mathfrak{\zeta})\,\le, C_0\, \mathbb{E}\|\theta_1(u)\|_{L^1(\mathbb{T}^N\times[0,T])}+\mathbb{E}\int_{\mathbb{T}^N}\Theta(u_0(x)) dx
	\end{align}
	for all non-negative $\theta_1\in C_c(\mathbb{R})$, where $\Theta(s)=\int_0^s\int_0^\sigma\theta_1(r) dr d\sigma.$
\end{lem}
\begin{proof}
	One can follow similar lines as proposed in \cite[Lemma 11]{vovelle 2} for a proof.
\end{proof}
\subsection*{Estimate on $Q$:}
Recall that
\begin{align*}
\langle Q(t), \varphi \rangle&=\int_0^t\int_{\mathbb{T}^N}\langle \partial_\mathfrak{\zeta} \mu_1(x,\cdot,t-s),S^*_{A_{\theta,\delta}}(s)\varphi\rangle_{\mathfrak{\zeta}}dx ds.
\end{align*}
By using the following identity
\begin{align*}
\partial_{\mathfrak{\zeta}} \big(S_{A_{\theta,\delta}}^*(t-s)\varphi(x)\big)=-(t-s)F''(\mathfrak{\zeta})\cdot\nabla\big(S_{A_{\theta,\delta}}^*(t-s)\varphi\big)-(t-s)A''(\mathfrak{\zeta})\big(-\Delta\big)^\alpha\big(S_{A_{\theta,\delta}}(t-s)\varphi\big),
\end{align*}
then we have
\begin{align*}
\langle Q(t), \varphi \rangle&=-\int_0^t\int_{\mathbb{T}^N}\int_{\mathbb{R}}\partial_\mathfrak{\zeta}\big( S^*_{A_{\theta,\delta}}(t-s)\varphi\big) d\mu_1(x,\mathfrak{\zeta},s)\\
&=\int_0^t\int_{\mathbb{T}^N}\int_{\mathbb{R}}(t-s)F''(\mathfrak{\zeta})\cdot\nabla\big(S_{A_{\theta,\delta}}^*(t-s)\varphi\big)d\mu_1(x,s,\mathfrak{\zeta})\\&\qquad\qquad\qquad\qquad\qquad+\int_0^t\int_{\mathbb{T}^N}\int_{\mathbb{R}}A''(\mathfrak{\zeta})\big(-\Delta\big)^\alpha\big(S_{A_{\theta,\delta}}(t-s)\varphi\big) d\mu_1(x,t,\mathfrak{\zeta})\\
&:=\langle Q_1, \varphi\rangle+\langle Q_2, \varphi \rangle,
\end{align*}
where
\begin{align*}
\langle Q_1, \varphi\rangle&=\int_0^t\int_{\mathbb{T}^N}\int_{\mathbb{R}}(t-s)F''(\mathfrak{\zeta})\cdot\nabla\big(S_{A_{\theta,\delta}}^*(t-s)\varphi\big)d\mu_1(x,s,\mathfrak{\zeta}),\\
\langle Q_2, \varphi \rangle&=\int_0^t\int_{\mathbb{T}^N}\int_{\mathbb{R}}A''(\mathfrak{\zeta})\big(-\Delta\big)^\alpha\big(S_{A_{\theta,\delta}}(t-s)\varphi\big) d\mu_1(x,t,\mathfrak{\zeta}).
\end{align*}
From \cite[section 3.6]{vovelle 2}, we can deduce the following estimate on $Q_1:$
\begin{align}\label{Q-1}
\mathbb{E}\|Q_1\|_{L^1((0,T);W^{\lambda,p}(\mathbb{T}^N))}\,\le\,C_{N,\beta,\lambda,p}\theta^{-2}\big(\frac{\theta}{\delta}\big)^{\kappa}\bigg(D\,\mathbb{E}\|F''(u)\|_{L^1(\mathbb{T}^N\times(0,T))}+\mathbb{E}\int_{\mathbb{T}^N}\Theta(u_0)dx\bigg),
\end{align}
where $\Theta(s)=\int_0^s\int_0^\sigma|F'(r)|drd\sigma,$ provided $\kappa=2-\frac{N+\lambda+1}{2\beta}+\frac{N}{2\beta}\,\textgreater\,0.$\\
As similar calculation in \cite[Section 3.6]{vovelle 2}, we can get the following estimate on $Q_2:$
\begin{align}\label{Q-2}
\mathbb{E}\|Q_2\|_{L^1((0,T);W^\lambda,p(\mathbb{T}^N))}\,\le\,C_{N,\beta,\lambda, p}\,\theta^{-2}\,\big(\frac{\theta}{\delta}\big)^{\kappa_1}\bigg(D\,\mathbb{E}\|A''(u)\|_{L^1(\mathbb{T}^N\times(0,T))}+\mathbb{E}\int_{\mathbb{T}^N}\Theta(u_0)dx\bigg),
\end{align}
where $\kappa_1=2-\frac{N+\lambda+2\alpha}{2\beta}+\frac{N}{2\beta}$.
\subsection*{$W^{s,q}$- Regularity bound on $u$: }
Here, We first collect all the estimates together. For $\lambda\,\textless\,\beta$ and $\beta\,\textless\,\frac{1}{2}$, we have the following estimates
\begin{align}
\int_0^T\|u^0(t)\|_{H^{\beta+\big(\frac{1}{2}-\beta\big)s}}dt\,&\le\,\,4\pi\theta^{s-1}\|u_0\|_{L_x^1},\label{u-1}\\
\int_0^T\|u^b(t)\|^2_{H^{\big(\frac{1}{2}-\beta\big)s}}dt&\le\frac{4\pi\theta^s}{2^s}\int_0^T\|u(t)\|_{L_x^1}dt,\label{u-2}\\
\mathbb{E}\big(\|P(t)\|_{H^\lambda(\mathbb{T}^N)}^2\big)\,&\le\,C(\theta,\delta, D),
\end{align}
and
\begin{align}\label{u-4}
\mathbb{E}\|Q\|_{L^1((0,T);W^{\lambda,p}(\mathbb{T}^N))}\,\le\,&C_{N,\beta,\lambda,p}\theta^{-2}\bigg(\big(\frac{\theta}{\delta}\big)^{\kappa}+\big(\frac{\theta}{\delta}\big)^{\kappa_1}\bigg)\notag\\&\qquad\bigg(\,\mathbb{E}\|F''(u)\|_{L^1(\mathbb{T}^N\times(0,T))}+\mathbb{E}\|A''(u)\|_{L^1(\mathbb{T}^N)}+\mathbb{E}\int_{\mathbb{T}^N}\Theta(u_0)dx\bigg),
\end{align}
Since $H^{\beta+\big(\frac{1}{2}-\beta\big)s}(\mathbb{T}^N)\hookrightarrow H^{\big(\frac{1}{2}-\beta\big)s}(\mathbb{T}^N),$
then  from \eqref{u-1} we have
\begin{align}\label{u-5}
\int_0^T\|u^0(t)\|_{H^{\big(\frac{1}{2}-\beta\big)s}}dt\,\le\,\,C\,\|u_0\|_{L_x^1}
\end{align}
Let $r\,\textgreater\,0,\,\,q\,\ge\,1$ such that $r\,\le\,(\frac{1}{2}-\beta)s$ and $\frac{1}{q}\,\ge\,\frac{r}{N}+\frac{1}{2}-\frac{(1-2\beta)s}{2N}$, then
$$H^{\big(\frac{1}{2}-\beta\big)s}(\mathbb{T}^N)\hookrightarrow W^{r,q}(\mathbb{T}^N).$$
It implies that 
\begin{align}
\mathbb{E}\big(\|u^0+u^b+P\|_{L^2((0,T);W^{r,q})}^2\big)\le\,C(\theta,\delta, C_0)\big(1+\mathbb{E}\|u\|_{L^1(\mathbb{T}^N\times(0,T))}+\mathbb{E}\|u_0\|_{L^1(\mathbb{T}^N)}+T\big).
\end{align}
By the H\"older ineqality, we have
\begin{align}\label{w1}
\mathbb{E}\big(\|u^0+u^b+P\|_{L^1((0,T);W^{r,q}(\mathbb{T}^N))}^2\big)\le\,C(\theta,\delta, C_0)\,T\,\big(1+\mathbb{E}\|u\|_{L^1(\mathbb{T}^N\times(0,T))}+\mathbb{E}\|u_0\|_{L^1(\mathbb{T}^N)}+T\big),
\end{align}
By using Jensen inequality,
\begin{align}\label{w2}
\mathbb{E}\big(\|u^0+u^b+P\|_{L^1((0,T);W^{r,q}(\mathbb{T}^N))}\big)\le\,C(\theta,\delta, C_0)\,T^{1/2}\,\big(1+\mathbb{E}\|u\|_{L^1(\mathbb{T}^N\times(0,T))}+\mathbb{E}\|u_0\|_{L^1(\mathbb{T}^N)}+T\big)^{1/2}.
\end{align}
Now by using Young inequality, we obtain
\begin{align}\label{w3}
T\big(1+\mathbb{E}\|u_0\|_{L^1(\mathbb{T}^N)}\big)\,\le\,T^2+\big(1+\mathbb{E}\|u_0\|_{L^1(\mathbb{T}^N)}\big)^2/2,
\end{align}
and
\begin{align}\label{w4}
T\mathbb{E}\|u\|_{L^1((0,T)\times\mathbb{T}^N)}\le\,4\tau T^2+\frac{\mathbb{E}^2\|u\|_{L^1((0,T)\times\mathbb{T}^N)}}{16\tau}\,\,\text{for some}\,\,\tau\,\textgreater\,0.
\end{align}
From \eqref{w1}-\eqref{w4} and choose $\tau$ such that, we obtain
\begin{align}\label{uu}
\mathbb{E}\big(\|u^0+u^b+P\|_{L^1((0,T);W^{r,q}(\mathbb{T}^N))}\big)\le\,C(\theta,\delta, C_0)\,\big(1+\mathbb{E}\|u_0\|_{L^1(\mathbb{T}^N)}+T\big) +\frac{1}{4}\mathbb{E}\|u\|_{L^1((0,T)\times\mathbb{T}^N)}.
\end{align}
Under the growth hypothesis, sub-linearity of $A''$ and $F''$, the bound \eqref{u-4} gives
\begin{align}
\mathbb{E}\|Q\|_{L^1((0,T);W^{\lambda,p}(\mathbb{T}^N))}\,\le\,&C_{N,\beta,\lambda,p}\theta^{-2}\bigg(\big(\frac{\theta}{\delta}\big)^{\kappa}+\big(\frac{\theta}{\delta}\big)^{\kappa_1}\bigg)\notag\bigg(\,1+\mathbb{E}\|u\|_{L^1(\mathbb{T}^N\times(0,T))}+\mathbb{E}\|u_0\|_{L^3(\mathbb{T}^N)}^3\bigg).
\end{align}
If we choose $\theta,\delta\,\textgreater\,0,$ such that
$$C_{N,\beta,\lambda,p}\theta^{-2}\bigg(\big(\frac{\theta}{\delta}\big)^{\kappa}+\big(\frac{\theta}{\delta}\big)^{\kappa_1}\bigg)\,\le\,\frac{1}{4}$$
then we obtain
\begin{align}\label{uuu}
\mathbb{E}\|Q\|_{L^1((0,T);W^{\lambda,p}(\mathbb{T}^N))}\,\le\,&\frac{1}{4}\mathbb{E}\|u\|_{L^1(\mathbb{T}^N\times(0,T))}+C(\theta,\delta, D)\big(1+\mathbb{E}\|u_0\|_{L^3(\mathbb{T}^N)}^3\big).
\end{align}
Choose $r\,\textgreater\,0\,,\, q\,\textgreater\,1$ such that $\min\{\lambda,\frac{1}{2}-\beta\}\,\ge\,r\,$ and $\frac{1}{q}\,\ge\,\max\{\frac{r}{N}+\frac{1}{p}-\frac{\lambda}{N}\,,\,\frac{r}{N}+\frac{1}{2}-\frac{(1-2\beta)s}{2N}\}$, from \eqref{u}, \eqref{uu}, \eqref{uuu} we deduce the following estimate
\begin{align*}
\mathbb{E}\|u\|_{L^1(0,T; W^{r,q}(\mathbb{T}^N))}\,\le\,C\big(1+T+\mathbb{E}\|u_0\|_{L^3(\mathbb{T}^N)}^3\big)+\frac{1}{2}\mathbb{E}\|u\|_{L^1((0,T)\times\mathbb{T}^N)}
\end{align*}
and thus we conclude
\begin{align}\label{final}
\mathbb{E}\|u\|_{L^1(0,T; W^{r,q}(\mathbb{T}^N))}\,\le\,C(\delta,\theta,D)\big(1+T+\mathbb{E}\|u_0\|_{L^3(\mathbb{T}^N)}^3\big).
\end{align}
Since $W^{r,q}(\mathbb{T}^N)$ is compactly embedded in $L^1(\mathbb{T}^N)$, then the Krylov-Bogoliubov theorem gives the existence of an invariant measure. Indeed, let $\lambda_{t,u_0}$ be the law of solution $u(t)$ and we define $$\mathfrak{\mathcal{V}}_T:=\frac{1}{T}\int_0^T\lambda_{s,u_0}ds.$$
By Markov's inequality,
$$\mathbb{P}\bigg(\frac{1}{T}\int_0^T\|u(s)\|_{W^{r,q}(\mathbb{T}^N)}ds\,\ge\,R\bigg)\,\le\,\frac{\mathbb{E}\|u\|_{L^1([0,T];W^{r,q}(\mathbb{T}^N))}}{R\,T}\,\le\,C\,\frac{(\mathbb{E}\|u_0\|_{L^3(\mathbb{T}^N)}^3+1)}{R}.$$
It shows that $\mathfrak{\mathcal{V}}_T$ is tight. By Prokhorov theorem, there exists a sequence $(t_n)$ increase to $\infty$ such that $\mathfrak{\mathcal{V}}_{t_n}$ converges weak-* to measure $\lambda$. Krylov-Bogoliubov theorem ensures that $\lambda$ is an invariant measure.
\subsection{Uniqueness of invariant measure}\label{section 5}
In this subsection, we prove the uniqueness of invariant measures.

\noindent
\textbf{Strategy:} In the first step, we will prove that, if the initial data is in a fixed ball and the noise is small in $W^{2,\infty}(\mathbb{T}^N)$ for some time interval. Then the average of a solution is small. In the second step, we will show that the solution enters a ball of some fixed radius in a finite time almost surely. With the help of this property, we will construct an increasing sequence of stopping times. In the third step, finally, we will analyze the behavior of the solution at a large time with the help of a constructed sequence of stopping times and strong Markov property. In this subsection, for technical purposes, we refer to \cite[Section 4]{vovelle 2}.
\subsection*{Smallness of solution:} 
\begin{proposition}\label{smallness} Let initial data $u_0\in L^1(\mathbb{T}^N)$ and $u$ be solution to \eqref{1.1} corresponding initial data $u_0$ in the sense of Definition \ref{definition kinetic solution in l1 setting}. Suppose that $F$ and $A$ satisfy \eqref{flux uniqueness}. Then, for any $\varepsilon\,\textgreater\,0$, there exist $T\,\textgreater\,0$ and $\mathfrak{\eta}\,\textgreater\,0$ such that:
	$$\frac{1}{T}\int_0^T\|u(s)\|_{L^1(\mathbb{T}^N)}ds\,\le\,\frac{\varepsilon}{2},$$
	whenever  $\|u(0)\|_{L^1(\mathbb{T}^N)}\,\le\,2\kappa_1\,\text{and}\,\sup_{t\in[0,T]}\|\Psi W(t)\|_{W^{2,\infty}(\mathbb{T}^N)}\,\le\,\mathfrak{\eta}. $
\end{proposition}
\begin{proof}
	We first take $\tilde{u}_0\in L^2(\mathbb{T}^N)$ such that
	$$\|u_0-\tilde{u}_0\|_{L^1(\mathbb{T}^N)}\,\le\,\frac{\varepsilon}{8}\,\,\,\|\tilde{u}_0\|_{L^2(\mathbb{T}^N)}\,\le\,C\,\kappa_1\varepsilon^{-N/2}.$$
	Let $\tilde{u}$ be the kineic solution starting from $\tilde{u}_0$. By \eqref{contraction principal}, we know that 
	$$\|u(t)-\tilde{u}(t)\|_{L^1(\mathbb{T}^N)}\,\le\,\frac{\varepsilon}{8},\,\,t\,\ge\,0,$$
	If we define $v=\tilde{u}-\Psi W$, then $v$ is a kinetic solution to  the following equation
	\begin{align*}
	\partial_t v + \mbox{div}F(v+\Psi W)+(-\Delta)^\alpha(A(v+\Psi W))=0
	\end{align*}
	and $g(x,\mathfrak{\zeta})=\mathbbm{1}_{v(x,t)\,\textgreater\,\mathfrak{\zeta}}-\mathbbm{1}_{0\,\textgreater\,\mathfrak{\zeta}}$,\,$f=\mathbbm{1}_{\tilde{u}\,\textgreater\,\mathfrak{\zeta}}-\mathbbm{1}_{0\,\textgreater\,\mathfrak{\zeta}}$ satisfy the following kinetic formulation in $\mathcal{D}'(\mathbb{T}^N\times[0,T]\times\mathbb{R})$(see Appendix \ref{A}) as follows:
	\begin{align}\label{kinetic formulation for v}
	\partial_t g(x,\mathfrak{\zeta},t)&+ F'(\mathfrak{\zeta})\cdot \nabla g(x,\mathfrak{\zeta},t) +A'(\mathfrak{\zeta})(-\Delta)^\alpha g(x,\mathfrak{\zeta},t)\notag\\&=(F'(\mathfrak{\zeta})-F'(\mathfrak{\zeta}+\Psi W)).\nabla g(x,\mathfrak{\zeta},t)-F'(\mathfrak{\zeta}+\Psi W) \nabla \big(\Psi W\big)\delta_{v=\mathfrak{\zeta}}\notag\\
	&\qquad\qquad+\big(A'(\mathfrak{\zeta})(-\Delta)^\alpha g(x,\mathfrak{\zeta},t)-A'(\mathfrak{\zeta}+\Psi W)\big((-\Delta)^\alpha f\big)(x,\mathfrak{\zeta}+\Psi(x)W,t)\big)+\partial_\mathfrak{\zeta} p(x,\mathfrak{\zeta},t),
	\end{align}
	where kinetic measure satisfies: $$p(x,\mathfrak{\zeta},t)\,\ge\,\int_{\mathbb{R}^N}{|A(\tilde{u}(x+z))-A(\Psi(x)W(x)+\mathfrak{\zeta})|}\mathbbm{1}_{\mbox{Conv}\{\tilde{u}(x+z)-\Psi(x)W(x),v(x)\}}(\mathfrak{\zeta})d\mu(z).$$
	Also  since both $u$ and $h_k$ have a zero spatial average, so does $v$. We use again $B_\theta$ and $A_{\theta,\delta}$ defined as before. Then  we can rewrite \eqref{kinetic formulation for v} in $\mathcal{D}'(\mathbb{T}^N\times[0,T]\times\mathbb{R})$ as follows: 
	
	\begin{align}\label{2}
	\partial_t g + A_{\theta, \delta} g& =(B_\theta + \delta Id) g +(F'(\mathfrak{\zeta})-F'(\mathfrak{\zeta}+\Psi W))\cdot\nabla g(x,\mathfrak{\zeta},t)-F'(\mathfrak{\zeta}+\Psi W)\cdot \nabla\big(\Psi W\big)\delta_{v=\mathfrak{\zeta}}\notag\\
	&\qquad\qquad+\big(A'(\mathfrak{\zeta})(-\Delta)^\alpha g(x,\mathfrak{\zeta},t)-A'(\mathfrak{\zeta}+\Psi W)\big((-\Delta)^\alpha f\big)(x,\mathfrak{\zeta}+\Psi(x)W,t)\big)+\partial_\mathfrak{\zeta} p(x,\mathfrak{\zeta},t).
	\end{align}
	 By using semi group approach (cf. \cite[Proposition 6.3]{prato}), we can get the following decomposition of $v$:
	$$v=v^0+v^b+v^\#+P^W+P^A+P^\alpha+P^\mathfrak{m},$$
	where
	\begin{align*}
	v^0(t):&=\int_{\mathbb{R}} S_{A_{\theta,\delta}}(t) g(0,\mathfrak{\zeta}) d\mathfrak{\zeta},\\
	v^b(t):&=\int_{\mathbb{R}}\int_0^t S_{A_{\theta,\delta}}(s)(B_{\theta}+\delta Id) g(\mathfrak{\zeta}, t-s) ds d\mathfrak{\zeta},\\
	v^\#(t):&=\int_{\mathbb{R}}\int_0^t S_{A_{\theta,\delta}}(s)\big(F'(\mathfrak{\zeta})-F'(\mathfrak{\zeta}+\Psi W)\big)\cdot \nabla g(\mathfrak{\zeta}, t-s) ds d\mathfrak{\zeta},\\
	\langle P^W(t),\varphi\rangle:&=-\int_{\mathbb{T}^N\times[0,t]}\big(F'(v+\Psi W)\cdot \nabla\big(\Psi W\big)\big)(x,s) \big(S^*(t-s)\varphi\big)(x,v(x,s)) ds dx,\\
	\langle P^A(t),\varphi\rangle:&=\int_{\mathbb{T}^N}\int_{\mathbb{R}}\int_0^t A'(\mathfrak{\zeta})(-\Delta)^\alpha S_{A_{\theta,\delta}}^*(t-s)\varphi(x) g(x,\mathfrak{\zeta},s) ds d\mathfrak{\zeta} dx,\\
	\langle P^\alpha(t),\varphi \rangle:&=\int_{\mathbb{T}^N\times\mathbb{R}\times[0,t]} A'(\mathfrak{\zeta})\big((-\Delta)^\alpha S_{A_{\theta,\delta}}^*(t-s)\varphi\big)(x,\mathfrak{\zeta}-\Psi W) f(x,\mathfrak{\zeta},s) ds d\mathfrak{\zeta} dx,\\
	\langle P^\mathfrak{m}(t),\varphi\rangle:&=-\int_{\mathbb{T}^N\times\mathbb{R}\times[0,t]}\partial_{\mathfrak{\zeta}}\big(S_{A_{\theta,\delta}}^*(t-s)\varphi\big) dp(x,\mathfrak{\zeta},s).
	\end{align*}
	We estimate above terms in the following steps.
	
	\noindent
	$\textbf{Step 1:}$ \textbf{Estimates on $v_0$ and $v^b$:}
	Reproducing similar argument of previous subsections, we can deduce the following estimates:
	$$\int_0^T\|v^0(t)\|_{H^\alpha(\mathbb{T}^N})^2\,\le\,C\,\theta^{s-1}\|u_0\|_{L^1(\mathbb{T}^N)} $$
	and 
	$$\int_0^T\|v^b(t)\|_{L^2(\mathbb{T}^N)}^2 dt\,\le\,C\,\theta^s\int_0^T\|v(t)\|_{L^1(\mathbb{T}^N)} dt.$$
	These imply
	\begin{align}
	\int_0^T\|v^0(t)\|_{L^1(\mathbb{T}^N)}dt\,\le\,T^{1/2}\bigg(\int_0^T\|v_0\|_{L^2(\mathbb{T}^N)})ds\bigg)^{1/2}\,\le\,C\,T^{1/2}\big( \theta^{s-1}\kappa_1\big)^{1/2},
	\end{align}
	and
	\begin{align}
	\int_0^T\|v^b(t)\|_{L^1(\mathbb{T}^N)}dt\,\le\,C\,T^{1/2}\bigg(\theta^{s}\,\int_0^T\|v(t)\|_{L^1(\mathbb{T}^N)}\bigg)^{1/2}\le\,C\,T\,\theta^{s} + \frac{1}{12}\int_0^T\|v(t)\|_{L^1(\mathbb{T}^N)}dt.
	\end{align}
	
	\noindent
	$\textbf{Step 2:}$ \textbf{Estimate on $v^\#$:}
	\begin{align*} v^\#(t):&=\int_{\mathbb{R}}\int_0^t S_{A_{\theta,\delta}}(s)\big(F'(\mathfrak{\zeta})-F'(\mathfrak{\zeta}+\Psi W)\big)\cdot \nabla g(\mathfrak{\zeta},t-s) ds d\mathfrak{\zeta},
	\end{align*}
	By using chain rule, we have
	\begin{align*}
	(F'(\mathfrak{\zeta})-F'(\mathfrak{\zeta}+\Psi(x)W))\cdot\nabla g(x,\mathfrak{\zeta},s)=\nabla\cdot ((F'(\mathfrak{\zeta})-F'(\mathfrak{\zeta}+\Psi(x)W))g(x,\mathfrak{\zeta},s))\\\qquad-(F''(\mathfrak{\zeta}+\Psi(x)W)\cdot\Psi(x)W)g(x,\mathfrak{\zeta},s),
	\end{align*}
	then it implies that
	\begin{align*}
	\langle v^\#(t),\varphi\rangle&=\int_{\mathbb{T}^N}\int_{\mathbb{R}}\int_0^t \varphi(x) S_{A_{\theta,\delta}}(t-s) \nabla\cdot ((F'(\mathfrak{\zeta})-F'(\mathfrak{\zeta}+\Psi(x)W))g(x,\mathfrak{\zeta},s))\,ds d\mathfrak{\zeta} dx\\
	&\qquad-\int_{\mathbb{T}^N}\int_{\mathbb{R}}\int_0^t\varphi(x) S_{A_{\theta,\delta}}(t-s)((F''(\mathfrak{\zeta}+\Psi(x)W)\cdot\Psi(x)W)g(x,\mathfrak{\zeta},s))ds d\mathfrak{\zeta} dx,\\
	&=\int_{\mathbb{T}^N}\int_{\mathbb{R}}\int_0^t\nabla (S_{A_{\theta,\delta}}^*(t-s)\varphi(x)) \cdot ((F'(\mathfrak{\zeta})-F'(\mathfrak{\zeta}+\Psi(x)W))g(x,\mathfrak{\zeta},s))\,ds d\mathfrak{\zeta} dx\\
	&\qquad-\int_{\mathbb{T}^N}\int_{\mathbb{R}}\int_0^t S_{A_{\theta,\delta}}^*(t-s)\varphi(x)((F''(\mathfrak{\zeta}+\Psi(x)W)\cdot\nabla(\Psi(x)W))g(x,\mathfrak{\zeta},s))ds d\mathfrak{\zeta} dx\\
	&=:\langle w_1(t),\varphi\rangle +\langle w_2(t), \varphi \rangle.
	\end{align*}
	We have, by using \cite[Lemma 9]{vovelle 2}
	\begin{align*}
	|\langle w_1(t) , \varphi \rangle|&\le\,\int_{\mathbb{T}^N}\int_{\mathbb{R}}\int_0^t|\nabla (S_{A_{\theta,\delta}}^*(t-s)\varphi(x)) \cdot ((F'(\mathfrak{\zeta})-F'(\mathfrak{\zeta}+\Psi(x)W))g(x,\mathfrak{\zeta},s))|dsd\mathfrak{\zeta} dx\\
	&\le\,C\,\mathfrak{\eta} \|\varphi\|_{L^\infty(\mathbb{T}^N)}\,\int_{\mathbb{T}^N}\int_{\mathbb{R}}\int_0^t (\theta(t-s))^{-\frac{1}{2\beta}}\e^{-\delta (t-s)}|g(x,\mathfrak{\zeta},s)|ds d\mathfrak{\zeta} dx.
	\end{align*}
It shows that
	\begin{align}
	\int_0^T\|w_1(t)\|_{L^1(\mathbb{T}^N)}dt\,&\le\,C\,\mathfrak{\eta}\,\sup_{s\in[0,T]}\big(\int_0^T\e^{-\delta(t-s)}(\theta(t-s))^{-\frac{1}{2\beta}}dt\big)\int_0^T\|v(s)\|_{L^1(\mathbb{T}^N)}\,ds\notag\\
	&\le\,C\,\mathfrak{\eta}\,\theta^{-\frac{1}{2\beta}}\delta^{\frac{1}{2\beta}-1}\,\int_0^\infty \e^{-t} t^{-\frac{1}{2\beta}}dt\,\int_0^T\|v(s)\|_{L^1(\mathbb{T}^N)}ds.
	\end{align}
Similarly we can get
	\begin{align}
	\int_0^T\|w_2(t)\|_{L^1(\mathbb{T}^N)}dt\,\le\,C\,\mathfrak{\eta}\,\int_0^T\|v(t)\|_{L^1(\mathbb{T}^N)}ds.
	\end{align}
	
	\noindent
	$\textbf{Step 3:}$ \textbf{Estimate on $P^W$:}
	\begin{align*}
	\langle P^W(t), \varphi \rangle\,&\le\,C\,\int_0^t\big(1+\|v(s)\|_{L^1(\mathbb{T}^N)}+\|\Psi W(s)\|_{L^1(\mathbb{T}^N)}\big)\|\nabla \Psi W(s)\|_{L^\infty(\mathbb{T}^N)}\|S^*_{A_{\theta,\delta}}(t-s)\varphi\|_{L^\infty(\mathbb{T}^N)}ds\\
	&\le\,C\,\|\varphi\|_{L^\infty(\mathbb{T}^N)}\mathfrak{\eta}\,\int_0^t(1+\|v(s)\|_{L^1(\mathbb{T}^N)})\,\e^{-\delta(t-s)} ds.
	\end{align*}
	Therefore $P^W $is more regular than a measure. In particular, we have
	\begin{align}
	\int_0^T\|P^W(t)\|_{L^1(\mathbb{T}^N)}dt\,\le\,C\,\frac{\mathfrak{\eta}}{\delta}\int_0^T\big(1+\|v(s)\|_{L^1(\mathbb{T}^N)}\big) ds.
	\end{align}

\noindent
$\textbf{Step 4:}$ \textbf{Estimate on $P^A$ and $P^\alpha$}:
	\begin{align*}
	\langle P^A(t),\varphi\rangle&\le\,C\,\int_{\mathbb{T}^N}\int_{\mathbb{R}}\int_0^t \|(-\Delta)^\alpha S_{A_{\theta,\delta}}^*(t-s)\varphi(x)\|_{L^\infty(\mathbb{T}^N\times\mathbb{R)}} |g(x,\mathfrak{\zeta},s)| ds d\mathfrak{\zeta} dx\\
	&\le\,C\int_{\mathbb{T}^N}\int_{\mathbb{R}}\int_0^t \big(\theta(t-s)\big)^{-\frac{\alpha}{\beta}}\e^{-\delta(t-s)}\|\varphi\|_{L^\infty(\mathbb{T}^N)}|g(x,\mathfrak{\zeta},s)|dsd\mathfrak{\zeta} dx,\\
	\end{align*}
It implies that
	\begin{align*}
	\int_0^T\langle P^A(t),\varphi\rangle dt&\le\,C\,\|\varphi\|_{L^\infty(\mathbb{T}^N)}\int_0^T\|v(t)\|_{L^1(\mathbb{T}^N)}dt\, \sup_{s\in[0,T]}\int_0^T\big(\theta(t-s)\big)^{-\frac{\alpha}{\beta}}\e^{-\delta(t-s)}dt\\
	&\le\,C\theta^{-\frac{\alpha}{\beta}}\delta^{\frac{\alpha}{\beta}}\|\varphi\|_{L^\infty(\mathbb{T}^N)}\,\bigg(\int_0^\infty t^{-\frac{\alpha}{\beta}}\e^{-t}dt\bigg)\int_0^T\|v(t)\|_{L^1(\mathbb{T}^N)}dt,
	\end{align*}
	It shows that
	\begin{align}
	\int_0^T\|P^A(t)\|_{L^1(\mathbb{T}^N)}dt&\le\,C\theta^{-\frac{\alpha}{\beta}}\delta^{\frac{\alpha}{\beta}}\,\bigg(\int_0^\infty t^{-\frac{\alpha}{\beta}}\e^{-t}dt\bigg)\,\int_0^T\|v(t)\|_{L^1(\mathbb{T}^N)}dt.
	\end{align}
	Similarly we can conclude that
	\begin{align}
	\int_0^T\|P^\alpha(t)\|_{L^1(\mathbb{T}^N)}dt&\le\,C\theta^{-\frac{\alpha}{\beta}}\delta^{\frac{\alpha}{\beta}}\,\bigg(\int_0^\infty t^{-\frac{\alpha}{\beta}}\e^{-t}dt\bigg)\,\int_0^T\|\tilde{u}(t)\|_{L^1(\mathbb{T}^N)}dt.
	\end{align}

\noindent
	$\textbf{Step 5:}$ \textbf{Bound on kinetic measure in $L^2(\mathbb{T}^N)$ setting}:
	We test equation \eqref{kinetic formulation for v} against $\mathfrak{\zeta}$ (justified by standard argument of smooth approximation), we get
	\begin{align*}
	\langle g, \mathfrak{\zeta} \rangle&=\langle g, \mathfrak{\zeta} \rangle-\langle F'(\mathfrak{\zeta}+\Psi W)\cdot\nabla g(x,\mathfrak{\zeta},t),\mathfrak{\zeta}\rangle-\langle F'(\mathfrak{\zeta}+\Psi W)\nabla (\Psi(x)W(x)) \delta_{v=\mathfrak{\zeta}},\mathfrak{\zeta}\rangle\\&\qquad-\langle A'(\mathfrak{\zeta}+\Psi W)\big((-\Delta)^\alpha f\big)(x,\mathfrak{\zeta}+\Psi W,t),\mathfrak{\zeta}\rangle+\langle \partial_{\mathfrak{\zeta}}p(x,\mathfrak{\zeta},s),\mathfrak{\zeta}\rangle,\\
	\\
	\frac{1}{2}\|v(t)\|^2_{L^2(\mathbb{T}^N)}&=\frac{1}{2}\|\tilde{u}_0\|_{L^2(\mathbb{T}^N)}^2+\int_0^t\int_{\mathbb{T}^N}\int_{\mathbb{R}}\mathfrak{\zeta}\,F''(\mathfrak{\zeta}+\Psi W)\cdot\nabla(\Psi(x)W)g(x,\mathfrak{\zeta},s) d\mathfrak{\zeta} dx ds\\&\qquad-\int_0^t\int_{\mathbb{T}^N}\int_\mathbb{R}(\mathfrak{\zeta}-\Psi(x)W)A'(\mathfrak{\zeta})(-\Delta)^\alpha f(x,\mathfrak{\zeta},s)d\mathfrak{\zeta} dx ds\\&\qquad+\int_0^t\int_{\mathbb{T}^N}vF'(v+\Psi(x)W)\cdot\nabla (\Psi(x)W)dxds+\int_0^t\int_{\mathbb{T}^N}\int_{\mathbb{R}}dp(x,\mathfrak{\zeta},s),
	\end{align*}
It implies that
	\begin{align*}
	\frac{1}{2}\|v(t)\|^2_{L^2(\mathbb{T}^N)}+\int_0^t\int_{\mathbb{T}^N}\int_{\mathbb{R}}dp(x,\mathfrak{\zeta},s)&=\frac{1}{2}\|\tilde{u}_0\|_{L^2(\mathbb{T}^N)}^2\\&\qquad+\int_0^t\int_{\mathbb{T}^N}\int_{\mathbb{R}}\mathfrak{\zeta}\,F''(\mathfrak{\zeta}+\Psi W)\cdot\nabla(\Psi(x)W)g(x,\mathfrak{\zeta},s) d\mathfrak{\zeta} dx ds\\&\qquad+\int_0^t\int_{\mathbb{T}^N}vF'(v+\Psi(x)W)\cdot\nabla (\Psi(x)W)dxds\\&\qquad+\int_0^t\int_{\mathbb{T}^N}\int_\mathbb{R}A'(\mathfrak{\zeta})(-\Delta)^\alpha(\Psi(x)W)f(x,\mathfrak{\zeta},s)d\mathfrak{\zeta} dx ds.
	\end{align*}
We conclude that
	\begin{align*}
	\frac{1}{2}\|v(t)\|^2_{L^2(\mathbb{T}^N)}&+|p|([0,t]\times\mathbb{T}^N\times\mathbb{R})\le\,\frac{1}{2}\|\tilde{u}_0\|_{L^2(\mathbb{T}^N)}^2+C\mathfrak{\eta}\int_0^t\int_{\mathbb{T}^N}\int_{\mathbb{R}}|\mathfrak{\zeta}||g(x,\mathfrak{\zeta},s)|d\mathfrak{\zeta} dx ds\\
	&\qquad+\mathfrak{\eta}\,\int_0^t\int_{\mathbb{T}^N} |v(x)|\big(1+|v(x)|+|\Psi(x)W|\big)dxds+C\,\mathfrak{\eta}\int_0^t\int_{\mathbb{T}^N}\int_{\mathbb{R}}|f(x,\mathfrak{\zeta},s)|dx d\mathfrak{\zeta} ds,
	\end{align*}
	In last term, we used that $\|(-\Delta)^\alpha \Psi W(t)\|_{L^{\infty}(\mathbb{T}^N)}\,\,\le\,C\,\sup_{t\in[0,T]}\|\Psi W\|_{W^{2,\infty}(\mathbb{T}^N)}$. It proves that
	\begin{align*}
	\frac{1}{2}\|v(t)\|^2_{L^2(\mathbb{T}^N)}&+|p|([0,t]\times\mathbb{T}^N\times\mathbb{R})\le\frac{1}{2}\|\tilde{u}_0\|_{L^2(\mathbb{T}^N)}^2+C\mathfrak{\eta}\int_0^t\big(1+\|v(s)\|_{L^2(\mathbb{T}^N)}^2\big)ds.
	\end{align*}
	The Gronwall's inequality implies
	$$\|v(t)\|_{L^2(\mathbb{T}^N)}^2\le\,C\,\e^{\mathfrak{\eta} C t}\big(\|\tilde{u}_0\|_{L^2(\mathbb{T}^N)}^2+1\big).$$
	Hence
	$$|p|([0,T]\times\mathbb{T}^N\times\mathbb{R})\le\,C\,\e^{\mathfrak{\eta} C T}\big(\|\tilde{u}_0\|_{L^2(\mathbb{T}^N)}^2+1\big).$$
	Since $\|\tilde{u}_0\|_L^2(\mathbb{T}^N)\,\le\,C\,k_0\,\varepsilon^{-\frac{N}{2}}$, it follows that 
	\begin{align}|p|\big([0,T]\times\mathbb{T}^N\times\mathbb{R}\big)\,\le\,C\,\e^{\mathfrak{\eta} C T}\big(k_0^2\varepsilon^{-N}+1\big).
	\end{align}

\noindent
	$\textbf{Step 6:}$ \textbf{Estimate on $P^m$:} It is easy to know that
	\begin{align*}
	\partial_{\mathfrak{\zeta}}\big(S_{A_{\theta,\delta}}^*(t-s)\varphi(x)\big)=-\big((t-s)F''(\mathfrak{\zeta})\cdot\nabla(S_{A_{\theta,\delta}}^*(t-s)\varphi(x))+(t-s)A''(\mathfrak{\zeta})(-\Delta)^\alpha(S_{A_{\theta,\delta}}^*(t-s)\varphi(x))\big).
	\end{align*}
	Using above identity we have
	\begin{align*}
	\langle P^\mathfrak{m}(t),\varphi\rangle&=-\int\limits_{\mathbb{T}^N\times\mathbb{R}\times[0,t]}\big((t-s)F''(\mathfrak{\zeta})\cdot\nabla(S_{A_{\theta,\delta}}^*(t-s)\varphi(x))+A''(\mathfrak{\zeta})(-\Delta)^\alpha(S_{A_{\theta,\delta}}(t-s)\varphi(x))\big)dp(x,\mathfrak{\zeta},s)\\
	&=:\langle I_1(t),\varphi\rangle+\langle I_2(t),\varphi\rangle.
	\end{align*}
We conclude that
\begin{align*}
	|\langle I_1(t),\varphi \rangle|\,&\le\,C\,\|\varphi\|_{L^\infty(\mathbb{T}^N)}\int_{\mathbb{T}^N\times\mathbb{R}\times[0,t]}(t-s)^{1-\frac{1}{2\beta}}\e^{-\delta(t-s)}\,\theta^{-\frac{1}{2\beta}}d|p|(x,\mathfrak{\zeta},s).
	\end{align*}
It implies that
	\begin{align}
	\int_0^T\|I_1(t)\|_{L^1(\mathbb{T}^N)}dt\,&\le C\,\theta^{-\frac{1}{2\beta}}\delta^{2-\frac{1}{2\beta}}\bigg(\int_0^\infty t^{1-\frac{1}{2\beta}}\e^{-t}dt\bigg)\int_{\mathbb{T}^N\times\mathbb{R}\times[0,T]}d|p|(x,\mathfrak{\zeta},s),\notag\\
	&\le\,C\,e^{\mathfrak{\eta} C T}\big(k_0^2\varepsilon^{-N}+1\big)\theta^{-\frac{1}{2\beta}}\delta^{2-\frac{1}{2\beta}}\bigg(\int_0^\infty t^{1-\frac{1}{2\beta}}\e^{-t}dt\bigg).
	\end{align}
	Similarly we can conclude that
	\begin{align}
	\int_0^T\|I_2(t)\|_{L^1(\mathbb{T}^N)}dt\,&\le C\,\theta^{-\frac{\alpha}{\beta}}\delta^{2-\frac{\alpha}{\beta}}\bigg(\int_0^\infty t^{1-\frac{\alpha}{\beta}}\e^{-t}dt\bigg)\int_{\mathbb{T}^N\times\mathbb{R}\times[0,T]}d|p|(x,\mathfrak{\zeta},s),\notag\\
	&\le\,C\,e^{\mathfrak{\eta} C T}\big(k_0^2\varepsilon^{-N}+1\big)\theta^{-\frac{\alpha}{\beta}}\delta^{2-\frac{\alpha}{\beta}}\bigg(\int_0^\infty t^{1-\frac{\alpha}{\beta}}\e^{-t}dt\bigg).
	\end{align}
	\textbf{Step 7:} \textbf{Conclusion:}
	We collect  all the estimates obtained in previous steps and choose $\alpha\,\le\,\beta$ then we deduce:
	\begin{align*}
	\int_0^T\|v(t)\|_{L^1(\mathbb{T}^N)}dt\,&\le\,C\,\bigg(T^{1/2}\big( \theta^{s-1}\kappa_1\big)^{1/2}+T\,\big(\theta^{s} + \frac{\mathfrak{\eta}}{\delta}\big) +\mathfrak{\eta}\,\theta^{-\frac{\alpha}{\beta}}\,\delta^{\frac{\alpha}{\beta}}+ \frac{1}{12}\int_0^T\|v(t)\|_{L^1(\mathbb{T}^N)}dt\\
	&\qquad+\big(\mathfrak{\eta}\,\theta^{-\frac{1}{2\beta}}\delta^{\frac{1}{2\beta}-1}\,+\frac{\mathfrak{\eta}}{\delta}+\mathfrak{\eta}+\theta^{-\frac{\alpha}{\beta}}\delta^{\frac{\alpha}{\beta}}\big)\int_0^T\|v(s)\|_{L^1(\mathbb{T}^N)}ds\\
	&\qquad+e^{\mathfrak{\eta} C T}\big(k_0^2\varepsilon^{-N}+1\big)(\theta^{-\frac{\alpha}{\beta}}\delta^{2-\frac{\alpha}{\beta}}+\theta^{-\frac{1}{2\beta}}\delta^{2-\frac{1}{2\beta}})\bigg),
	\end{align*}
	We choose $\theta,\delta,\,\mathfrak{\eta}\,\textgreater\,0$ such that
	$$\theta^s\,\le\,\frac{\varepsilon}{32\,C},\qquad\,\,\,\,\,\,\theta^{-\frac{1}{2\beta}}\delta^{\frac{1}{2\beta}-1} + \theta^{-\frac{\alpha}{\beta}}\delta^{\frac{\alpha}{\beta}}+e^{ C T}\big(k_0^2\varepsilon^{-N}+1\big)(\theta^{-\frac{\alpha}{\beta}}\delta^{2-\frac{\alpha}{\beta}}+\theta^{-\frac{1}{2\beta}}\delta^{2-\frac{1}{2\beta}})\le\,\min\bigg\{\frac{\varepsilon}{32\,C},\,\frac{1}{3}\bigg\},$$
	and
	$$\mathfrak{\eta}+\frac{\mathfrak{\eta}}{\delta}\le\,\min\big\{\frac{\varepsilon}{32\,C},\frac{1}{3},\,\,\frac{\varepsilon}{8\delta}\big\}.$$
	Choose T sufficiently large such that
	$$T^{-1/2}\big( \theta^{s-1}\kappa_1\big)^{1/2}\,\le\,\frac{\varepsilon}{32C}.$$
	Then
	$$\frac{1}{T}\int_0^T\|v(t)\|_{L^1(\mathbb{T}^N)}dt\,\le\,\frac{\varepsilon}{4},\,\,\qquad\frac{1}{T}\int_0^T\|\tilde{u}\|_{L^1(\mathbb{T}^N)}dt\,\le\,\frac{3\varepsilon}{8},$$
	and
	$$\frac{1}{T}\int_0^T\|u(t)\|_{L^1(\mathbb{T}^N)}dt\,\textless\,\varepsilon.$$
	This finishes the proof.
\end{proof}
\subsection{Sequence of increasing stopping times:}
In this subsection, first, we prove that estimate \eqref{final} implies that any solution enters a ball of some fixed radius at a finite time. With help of this property, we construct a sequence of increasing stopping times. Let $\tilde{u}_0^1,\,\tilde{u}_0^2$ in $L^3(\mathbb{T}^N)$ and denote by $\tilde{u}^1, \tilde{u}^2$ the corresponding solutions.
We can easily generalize \eqref{final} to two solutions on interval $[t,t+T]$ for $t,T\,\ge0.$ By repeating similar procedure to estimate for conditional expections, we have
\begin{align}
\mathbb{E}\bigg(\int_t^{t+T}\|\tilde{u}^1(s)\|_{L^1(\mathbb{T}^N)}+\|\tilde{u}^2(s)\|_{L^1(\mathbb{T}^N)}ds\big|\mathcal{F}_t\bigg)\le\,k\,\big(\|\tilde{u}^1(t)\|_{L^3(\mathbb{T}^N)}^3+\|\tilde{u}^2(t)\|_{L^3(\mathbb{T}^N)}^3+1+T\big)
\end{align}
The following property can be shown in the same way as proposed in \cite[Section 4]{vovelle 2} via a Borel-Cantelli Lemma.
Now we introduce for other term the sequence of deterministic times $(t_l)_{l\ge0}$ and $(r_l)_{l\ge0}$ by 
$$t_0=0,$$
$$t_{l+1}=t_l+r_l$$
and  we choose $(r_l)_{l\ge0}$ so that the following inequality holds for all $l\,\ge\,0:$
$$\frac{1}{2r_l}\big(\|\tilde{u}_0^1\|_{L^3(\mathbb{T}^N)}^3+\|\tilde{u}_0^2\|_{L^3(\mathbb{T}^N)}^3+1\big)\,\le\,\frac{1}{8},$$
then
$$k_0=\inf\{l\ge 0|\,\inf_{s\in[t_l,t_{l+1}]}\|\tilde{u}^1(t)\|_{L^1(\mathbb{T}^N)}+\|\tilde{u}^2(t)\|_{L^1(\mathbb{T}^N)}\,\le\,2k_1\}$$
is almost surely finite (for detials see \cite[Section 4]{vovelle 2}). Then we define the stopping time
$$\tau^{u_0^1, u_0^2}=\inf\{t\ge0,| \|\tilde{u}^1(t)\|_{L^1(\mathbb{T}^N)}+\|\tilde{u}^2(t)\|_{L^1(\mathbb{T}^N)}\,\le\,2k_1\}.$$
It is clear that $\tau^{\tilde{u}_0^1, \tilde{u}_0^2}\,\le\,t_{k_0+1}$ then it implies that $\tau^{\tilde{u}_0^1, \tilde{u}_0^2}\,\textless\,\infty$ almost surely. It shows that the following stopping times are also almost surely finite: for any $T\,\textgreater\,0$,
$$\tau_l=\inf\{t\,\ge\,\tau_{l-1}+T\big|\, \|\tilde{u}^1(t)\|_{L^1(\mathbb{T}^N)}+\|\tilde{u}^2(t)\|_{L^1(\mathbb{T}^N)}\,\le\,2k_1\},\,\,\tau_0=0.$$
\subsection{Conclusion of uniqueness}
Let $\varepsilon\,\textgreater\,0$. Let $u_0^1,\, u_0^2$ be in $L^1(\mathbb{T}^N)$. We take $\tilde{u}_0^1,\,\tilde{u}_0^2$ in $L^3(\mathbb{T}^N)$ such that $\|u_0^j-\tilde{u}_0^j\|_{L^1(\mathbb{T}^N)}\,\le\,\frac{\varepsilon}{4}$.
From Proposition \ref{smallness}, there exist $T\,\textgreater\,0$ and $\mathfrak{\eta}\,\textgreater\,0$ such that
$$\frac{1}{T}\int_{\tau_l}^{\tau_l+T}\|\tilde{u}^1(s)-\tilde{u}^2(s)\|_{L^1(\mathbb{T}^N)}\,\le\,\frac{\varepsilon}{2}$$
whenever
$$\sup_{[\tau_l,\tau_l+T]}\|\Psi W(t)-\Psi W(\tau_l)\|_{W^{2,\infty}(\mathbb{T}^N)}\,\le\,\mathfrak{\eta}.$$
We define
$$A_{W,l}=\big\{\omega\in\Omega;\,\sup_{[\tau_l,\tau_l+T]}\|\Psi W(t)-\Psi W(\tau_l)\|_{W^{2,\infty}(\mathbb{T}^N)}\,\le\,\mathfrak{\eta}\big\},$$
$$\tilde{A}_l=\big\{\omega\in\Omega;\, \frac{1}{T}\int_{\tau_l}^{\tau_l+T}\|\tilde{u}^1(s)-\tilde{u}^2(s)\|_{L^1(\mathbb{T}^N)}\,\le\,\frac{\varepsilon}{2}\big\},$$
and$$A_l=\big\{\omega\in\Omega;\,\frac{1}{T}\int_{\tau_l}^{\tau_l+T}\|u^1(s)-u^2(s)\|_{L^1(\mathbb{T}^N)}\,\le\,\varepsilon\big\}.$$
By $L^1$-contraction \eqref{contraction principal}, we have following inequality,
\begin{align*}
\mathbb{E}\big(\mathbbm{1}_{A_{W,l}}\,\big|\,\mathcal{F}_{\tau_l}\big)\,\le\,\mathbb{E}\big(\mathbbm{1}_{\tilde{A}_l}\,\big|\,\mathcal{F}_{\tau_l}\big)\,\le\,\mathbb{E}\big(\mathbbm{1}_{A_l}\,\big|\,\mathcal{F}_{\tau_l}\big).
\end{align*}
By the strong Markov property, $\mathbb{E}\big(\mathbbm{1}_{A_{W,l}}\,\big|\,\mathcal{F}_{\tau_l}\big)$ is positive constant $c$, then
\begin{align*}
\mathbb{E}\bigg[\mathbb{E}\big[\mathbbm{1}_{A_l^c\cap A_{l+1}^c}\,\big|\,\mathcal{F}_{\tau_{l+1}}\big]\,\bigg|\,\mathcal{F}_{\tau_l}\bigg]=\mathbb{E}\bigg[\mathbbm{1}_{A_l^c}\,\mathbb{E}\big[ \mathbbm{1}_{A_{l+1}^c}\,\big|\,\mathcal{F}_{\tau_{l+1}}\big]\,\bigg|\,\mathcal{F}_{\tau_l}\bigg]\,\textless\,(1-c)\,\mathbb{E}\big[\mathbbm{1}_{A_l^c}\,\big|\,\mathcal{F}_{\tau_l}\big]\,\textless\,(1-c)^2.
\end{align*}
It implies that
\begin{align*}
\mathbb{P}\big(A_l^c\cap A_{l+1}^c\big)\,\textless\,(1-c)^2.
\end{align*}
Inductively, for $l_0,\,k\in \mathbb{N}$, we have
\begin{align*}
\mathbb{P}\bigg(\cap_{j=0}^k A_{{l_j}}^c\bigg)\,\textless\,(1-c)^k,
\end{align*}
and therefore
\begin{align}\label{limit}
\mathbb{P}\bigg(\cap_{j=0}^{\infty}A_{j}^c\bigg)=\mathbb{P}\bigg(\exists\,l_0\in\mathbb{N};\,\cap_{j=0}^{\infty}A_{l_j}^c\bigg)=0.
\end{align}
Notice that limit exists, by \eqref{contraction principal}, $t\,\mapsto\,\|u^1(t)-u^2(t)\|_{L^1(\mathbb{T}^N)}$ is almost surely non-increasing. This latter property is again used to conclude from \eqref{limit} that
\begin{align*}
\mathbb{P}\big(\lim_{t\to\infty}\|u^1(t)-u^2(t)\|_{L^1(\mathbb{T}^N)}\,\textgreater\,\varepsilon\big)=0.
\end{align*}
It implies that $\mathbb{P}$-almost surely,
$$\lim_{t\to\infty}\|u^1(t)-u^2(t)\|_{L^1(\mathbb{T}^N)}=0.$$
It implies that $$\mathbb{P}\big(u^1(t)\neq u^2(t)\big)\,\to\,0\,\,\,\qquad\text{as}\,\qquad t\to\infty.$$
Let $\mu_{t,u_0^1},\,\mu_{t,u_0^2}$ be law of solutions $u^1(t)$, $u^2(t)$ respectively, then
$$\|\mu_{t,u_0^1}-\mu_{t,u_0^2}\|_{TV}\,\le\,2\,\mathbb{P}\big(u^1(t)\neq u^2(t)\big),$$
It shows that 
$$\lim_{t\to\infty}\langle \mu_{t,u_0^1},\kappa \rangle=\lim_{t\to\infty}\langle\mu_{t,u_0^2},\kappa\rangle=c_{\kappa}\,\,\,\,\,\,\forall\,\kappa\in C_b(L^1(\mathbb{T}^N)).$$
It implies that  for all $u_0\in L^1(\mathbb{T}^N)$,
\begin{align*}
	\lim_{t\to\infty}\mathcal{Q}_t\kappa(u_0)=\lim_{t\to\infty}\langle \mu_{t,u_0},\kappa\rangle= c_{\kappa}.
\end{align*}
Let $\lambda$ be any invariant measure and $\kappa\in C_b(L^1(\mathbb{T}^N))$, then by dominated convergence theorem, we get
\begin{align}\label{convergence of law}
	\lim_{t\to\infty}\langle \mathcal{Q}_t^*\lambda, \kappa\rangle=\lim_{t\to\infty}\langle\lambda,\mathcal{Q}_t\kappa\rangle\implies \lim_{t\to\infty}\langle \mu_{t,u_0},\kappa\rangle=\langle \lambda, \kappa\rangle, \,\,\forall\,\,u_0\in L^1(\mathbb{T}^N).
\end{align}
It shows that invariant measure $\lambda$ is unique \cite[Proposition 11.4]{prato} and law $\mu_{t,u_0}$ of solution $u(t)$ with initial data $u_0$ converge to an unique invariant measure $\lambda$ as $t\to\infty$.
\appendix

\section{Derivation of the kinetic formulation.}\label{A}
In this Appendix, we give brief details about the derivation of kinetic formulation of equation \eqref{1.1}. Here we are assuming that flux $F$ and diffusive function $A$ are smooth and Lipschitz continuous because we are only working with approximations of equation \eqref{1.1}. We prove that if u is a weak solution to \eqref{1.1} such that 
$$ u\in L^2(\Omega ; C([0,T];L^2(\mathbb{T}^N)))\cap L^2(\Omega;L^2(0,T; H^1(\mathbb{T}^N)),\text{and}\,A(u)\in L^2(\Omega;L^2(0,T; H^\alpha(\mathbb{T}^N))$$then $f(t)=\mathbbm{1}_{u(t)\textgreater\mathfrak{\zeta}}$ satisfies
$$df(t)+ F'\cdot\nabla f(t) dt + A'(\mathfrak{\zeta})(-\Delta)^\alpha( f(t)) dt = \delta_{u(t)=\mathfrak{\zeta}} \varphi dW(t) + \partial_{\mathfrak{\zeta}} (\mathfrak{\eta}-\frac{1}{2} H^2 \delta_{u(t)=\mathfrak{\zeta}}) dt$$
in the sense of $\mathcal{D}'(\mathbb{T}^N\times\mathbb{R})$ where $\eta\,\ge\,\eta_1$ with
$$d\mathfrak{\eta}_1(x,t,\mathfrak{\zeta})=\int_{\mathbb{R}^N}|A(u(x+z))-A(\mathfrak{\zeta})|\mathbbm{1}_{\mbox{Conv}\{u(x),u(x+z)\}}(\mathfrak{\zeta}) \mu(z) dz d\mathfrak{\zeta} dx dt .$$
Indeed, it follows from generalized It\^o formula \cite[AppendixA]{vovelle}, for $\varphi\in C_b^2(\mathbb{R})$ with $\varphi(-\infty)=0$, $\kappa\in C^2(\mathbb{T}^N),$ $\mathbb{P}$-almost surely,
\begin{align*}
\langle \varphi(u(t)), \kappa \rangle &= \langle \varphi(u_0), \kappa \rangle -\int_0^t \langle \varphi'(u) \mbox{div} (F(u)), \kappa \rangle ds-\int_0^t \langle \varphi'(u)(-\Delta)^\alpha(A(u)),\kappa \rangle ds\\
&\qquad+\sum_{k \ge 1} \int_0^t \langle \varphi'(u) h_k(x,u),\kappa \rangle d\beta_k(s)+\frac{1}{2}\int_0^t\langle \varphi''(u)H^2(x,u),\kappa \rangle ds.
\end{align*}
Since $(-\Delta)^{\alpha/2}$ is continuous linear operator from $H^1(\mathbb{T}^N)$ to $L^2(\mathbb{T}^N)$, it implies that 
\begin{align*}\langle \phi'(u)) (-\Delta)^\alpha\big(A(u)\big),\kappa \rangle&=\langle (-\Delta)^{\alpha/2}\big(\phi'(u)\kappa\big),(-\Delta)^{\alpha/2}\big(A(u)\big)\rangle\\&=\lim_{\delta\to 0}\langle (-\Delta)^{\alpha/2}(\phi'(u^\delta)\kappa),(-\Delta)^{\alpha/2}\big(A(u^\delta)\big)\rangle\\&=\lim_{\delta\to 0}\langle \phi'(u^\delta)) (-\Delta)^\alpha\big(A(u^\delta)\big),\kappa \rangle,
\end{align*}
Here $u^\delta$ is pathwise molification of $u$ in space variable. By using it,  we have 
\begin{align*}\langle \varphi'(u^\delta(x,t)) (-\Delta)^\alpha&(A(u^\delta(x,t))),\kappa(x) \rangle \\
&= -\langle \varphi'(u^\delta(x,t)) \int_{\mathbb{R}^N}(A(u^\delta(x+z,t))-A(u^\delta(x,t))) \mu(z) dz ,\kappa(x) \rangle, 
\end{align*}
By making use of the Taylor's identity \cite[identity (17)]{N-2}, we get
\begin{align*}
&\langle \varphi'(u^\delta(x,t))\int_{\mathbb{R}^N}(A(u^\delta(x+z,t))-A(u^\delta(x,t)))\mu(z) dz ,\kappa \rangle \\ 
&\qquad\qquad= \int_{\mathbb{T}^N}\int_{\mathbb{R}^N} \varphi'(u^\delta(x,t))(A(u^\delta(x+z,t))-A(u^\delta(x,t)))\kappa(x) \mu(z) dz dx\\
&\qquad\qquad=\int_{\mathbb{T}^N}\int_{\mathbb{R}^N}\kappa(x)\bigg(\int_{\mathbb{R}}(\varphi'(\mathfrak{\zeta})A'(\mathfrak{\zeta})\mathbbm{1}_{u^\delta(x+z,t)\textgreater\mathfrak{\zeta}}-\varphi'(\mathfrak{\zeta})A'(\mathfrak{\zeta})\mathbbm{1}_{u^\delta(x,t)\textgreater\mathfrak{\zeta}})d\mathfrak{\zeta}\\
&\qquad\qquad\qquad-\int_{\mathbb{R}}\varphi''(\mathfrak{\zeta})|A(u^\delta(x+z,t))-A(\mathfrak{\zeta})|\mathbbm{1}_{\mbox{Conv}\{u^\delta(x+z,t),u^\delta(x,t)\}} \bigg)\mu(z)dz dx\\
&\qquad\qquad=\int_{\mathbb{T}^N}\int_{\mathbb{R}^N}\int_{\mathbb{R}}A'(\mathfrak{\zeta}) \varphi(\mathfrak{\zeta}) \mathbbm{1}_{u^\delta(x,t)\textgreater\mathfrak{\zeta}}(\kappa(x+z)-\kappa(x)))\mu(z)\,d\mathfrak{\zeta}\,dz dx\\
&\qquad\qquad\qquad-\int_{\mathbb{T}^N} \int_{\mathbb{R}}\kappa(x) \varphi''(\mathfrak{\zeta})\int_{\mathbb{R}^N} |A(u^\delta(x+z,t))-A(\mathfrak{\zeta})|\mathbbm{1}_{\mbox{Conv}\{u^\delta(x+z,t),u^\delta(x,t)\}}(\mathfrak{\zeta})\mu(z)dz d\mathfrak{\zeta} dx,
\end{align*}
It implies that
\begin{align*}
&\int_0^t\langle \varphi'(u(x,s)) (-\Delta)^\alpha(A(u))(x,s),\kappa \rangle ds\\
&=-\lim_{\delta\to0}\int_0^t\int_{\mathbb{T}^N} \int_{\mathbb{R}}A'(\mathfrak{\zeta}) \varphi'(\mathfrak{\zeta}) \mathbbm{1}_{u^\delta(x,s)\textgreater\mathfrak{\zeta}}\int_{\mathbb{R}^N}(\kappa(x+z)-\kappa(x))\mu(z)d\mathfrak{\zeta} dz dx ds\\
&\quad+\lim_{\delta\to0}\int_0^t\int_{\mathbb{T}^N} \int_{\mathbb{R}}\kappa(x) \varphi''(\mathfrak{\zeta})\int_{\mathbb{R}^N} |A(u^\delta(x+z,s))-A(\mathfrak{\zeta})|\mathbbm{1}_{\mbox{Conv}\{u^\delta(x+z,s),u^\delta(x,s)\}}(\mathfrak{\zeta})\mu(z)dz d\mathfrak{\zeta} dx ds\\
&=-\int_{\mathbb{T}^N} \int_{\mathbb{R}}A'(\mathfrak{\zeta}) \varphi'(\mathfrak{\zeta}) \mathbbm{1}_{u(x,s)\textgreater\mathfrak{\zeta}}\int_{\mathbb{R}^N}(\kappa(x+z)-\kappa(x))\mu(z) dz d\mathfrak{\zeta}dx ds +\lim_{\delta\to0}\int_0^t\int_{\mathbb{T}^N} \int_{\mathbb{R}}\kappa(x) \varphi''(\mathfrak{\zeta})d\eta^\delta\\
&=\int_0^t\int_{\mathbb{T}^N}\int_{\mathbb{R}}A'(\mathfrak{\zeta}) \varphi'(\mathfrak{\zeta}) \mathbbm{1}_{u(x,s)\textgreater\mathfrak{\zeta}}(-\Delta)^\alpha(\kappa)(x)dxd\mathfrak{\zeta} +\int_0^t\int_{\mathbb{T}^N}\int_{\mathbb{R}}\kappa(x)\varphi''(\mathfrak{\zeta})d\eta(x,\zeta,s)\\
&=\int_0^t\langle \mathbbm{1}_{u(x,s)\textgreater \mathfrak{\zeta} } \varphi'(\mathfrak{\zeta}), A'(\mathfrak{\zeta}) (-\Delta)^\alpha(\kappa)(x)\rangle_{x,\mathfrak{\zeta} } -\langle \kappa(x) \varphi'(\mathfrak{\zeta}),\partial_{\mathfrak{\zeta}}\mathfrak{\eta}\rangle[0,t]
\end{align*}
where $\eta$ is weak-* limit of $\eta^\delta=\int_{\mathbb{R}^N} |A(u^\delta(x+z,t))-\xi|\mathbbm{1}_{\mbox{Conv}\{u^\delta(x+z,t),u^\delta(x,t)\}}(\xi)\mu(z)dz$ in $\mathcal{M}^+(\mathbb{T}^N\times[0,T]\times\mathbb{R})$. Note that above convergence holds for all $t\in[0,T]$ due to $u\in L^2(\Omega;C([0,T];L^2(\mathbb{T}^N))$. As an application of Fatou's lemma we have $\mathbb{P}$-almost surely, $\eta\,\ge\,\eta_1$.
Making use of the chain rule for functions from Sobolev spaces, we obtain the following equalities that hold true in $\mathcal{D}'(\mathbb{T}^N)$.
\begin{align*}\langle \mathbbm{1}_{u(x,t)\textgreater\mathfrak{\zeta}} , \varphi'\rangle_\mathfrak{\zeta}&= \int_{\mathbb{R}}\mathbbm{1}_{u(x,t)\textgreater\mathfrak{\zeta}}\varphi'(\mathfrak{\zeta})d\mathfrak{\zeta}= \varphi(u(x,t))\\
\varphi'(u(x,t))\mbox{div}(F(u(x,t)))&=\varphi'(u(x,t))F'(u(x,t))\cdot\nabla u(x,t)\notag\\
&=\mbox{div}(\int_{-\infty}^{u(x,t)}F'(\mathfrak{\zeta})\varphi'(\mathfrak{\zeta}))= \mbox{div}(\langle F'\mathbbm{1}_{u(x,t)\textgreater\mathfrak{\zeta}},\varphi'\rangle_\mathfrak{\zeta})\\
\varphi'(u(x,t))h_k(x,u(x,t))&=\langle h_k \delta_{u(x,t)=\mathfrak{\zeta}} ,\varphi'\rangle_{\mathfrak{\zeta}}\\
\varphi''(u(x,t))H^2(x,u(x,t))&=\langle H^2 \delta_{u(x,t)=\mathfrak{\zeta}} , \varphi''\rangle_{\mathfrak{\zeta}}=-\langle \partial_{\mathfrak{\zeta}}(H^2\delta_{u(x,t)=\mathfrak{\zeta}}), \varphi' \rangle_{\mathfrak{\zeta}}.
\end{align*}
Therefore we define $\varphi=\int_{-\infty}^\mathfrak{\zeta} \Upsilon(\mathfrak{\zeta}) d\mathfrak{\zeta}$ for some $\Upsilon\in C_c^\infty(\mathbb{R})$ to deduce the kinetic formulation.
\section*{Derivation of fomulation \eqref{kinetic formulation for v}}\label{app}
Here, we are giving only details about fractional term. For other terms, derivation is exactly same as in previous formulations. Let $\varphi\in C_c^2(\mathbb{R})$. By making use of the Taylor's identity \cite[identity (17)]{N-2}, we get
\begin{align*}
\varphi'(v)(A(\tilde{u}(x+z))&-A(\tilde{u}(x)))\\&=\int\limits_{R} \big(A'(\mathfrak{\zeta})\varphi'(\mathfrak{\zeta}+\Psi(x)W(x))\mathbbm{1}_{\tilde{u}(x+z)\,\textgreater\,\mathfrak{\zeta}}-A'(\mathfrak{\zeta})\varphi'(\mathfrak{\zeta}+\Psi(x)W(x))\mathbbm{1}_{\tilde{u}(x)\,\textgreater\,\mathfrak{\zeta}}\big)\d\mathfrak{\zeta}\\
&\qquad-\int_{\mathbb{R}}\varphi''(\mathfrak{\zeta}+\Psi(x)W(x))|A(\tilde{u}(x+z))-A(\mathfrak{\zeta})|\mathbbm{1}_{\mbox{Conv}\{\tilde{u}(x),\tilde{u}(x+z)\}}(\mathfrak{\zeta})d\mathfrak{\zeta}\\
&=\int_{\mathbb{R}}A'(\mathfrak{\zeta}+\Psi(x)W(x))\varphi'(\mathfrak{\zeta})\big(\mathbbm{1}_{\tilde{u}(x+z)\,\textgreater\,\mathfrak{\zeta}+\Psi(x)W(x)}-\mathbbm{1}_{\tilde{u}(x)\,\textgreater\,\mathfrak{\zeta}+\Psi(x)W(x)}\big)d\mathfrak{\zeta}\\
&\qquad-\int_{\mathbb{R}} \varphi''(\mathfrak{\zeta})|A(\tilde{u}(x+z))-A(\mathfrak{\zeta})|\mathbbm{1}_{\mbox{Conv}\{v,\tilde{u}(x+z)-\Psi(x)W(x)\}}d\mathfrak{\zeta},
\end{align*} 
This identity is enough to derive the kinetic formulation \eqref{kinetic formulation for v}.
\label{appB}
\section{Well-posedness for $L^1$-setting: Proof of Theorem \ref{main result 2}}\label{B}
In this section we present the proof of the main well-posedness result in $L^1$-setting, Theorem \ref{existance and uniqueness}. Here we closely follow the approach of \cite{vovelle 2}. The existence part of Theorem \ref{existance and uniqueness} will be discussed in subsection \ref{Section 3.1}, the uniqueness in subsection \ref{Section 3.2}.
\subsection{Existence}\label{Section 3.1}
In this subsection we prove the existence of kinetic solution. The proof of existence is based on a approximation procedure.\\
\textbf{Existence of predictable process u:}
First, we replace the initial condition $u_0$ by bounded approximation $u_0^k\in L^\infty(\mathbb{T}^N)$ such that $u_0^k\,\to\,u_0\,$in $L^1(\mathbb{T}^N)$.  This defines a sequence of solution $(u^k)$ and kinetic measures $(m^k)$ satisfy, $\mathbb{P}$-almost surely, for all $t\in[0,T]$
\begin{align}\label{formulation-1}
\langle f^k(t),\varphi \rangle &= \langle f_0^k, \varphi \rangle + \int_0^t\langle f^k(s),F'(\mathfrak{\zeta})\cdot\nabla\varphi\rangle ds -\int_0^t\langle f^k(s), A'(\mathfrak{\zeta})\,(-\Delta)^\alpha[\varphi]\rangle ds\notag\\
&\qquad+\sum_{k=1}^\infty \int_0^t \int_{\mathbb{T}^N} h_k(x)\varphi(x,u^k(x,s) dx d\beta_k(s)\notag\\
&\qquad+\frac{1}{2}\int_0^t\int_{\mathbb{T}^N}\partial_{\mathfrak{\zeta}}\varphi(x,u^k(x,s)) H^2 (x)dxds -m^k(\partial_{\mathfrak{\zeta}}\varphi)([0,t]).
\end{align}
where $f^k(s)=\mathbbm{1}_{u^k(s)\,\textgreater\,\mathfrak{\zeta}}$ and $m^k\,\ge\,\mathfrak{\eta}^k$ $\mathbb{P}$ almost surely with
$$\mathfrak{\eta}^k(x,t,\mathfrak{\zeta})=\int_{\mathbb{R}^N}|A(u^k(x+z,t))-A(\mathfrak{\zeta})|\mathbbm{1}_{\mbox{Conv}\{u^k(x,t),u^k(x+z,t)\}}(\mathfrak{\zeta})\mu(z)dz.$$
By the contraction property in $L^1$, sequence $(u^k)$ is cauchy sequence in $L^1_{\mathcal{P}}(\Omega\times\,\mathbb{T}^N\times[0,T])$. It implies that $u^k$ converges to $u$ in $L_{\mathcal{P}}^1(\Omega\times\mathbb{T}^N\times[0,T])$. 

\noindent
\textbf{Bound on approxmate kinetic measure:}
 $m^k$ is the kinetic measure associated with $u^k$. Now we are interested to find a convergence of $m^k$ in a suitable topological space. For that, we are trying to find some type of uniformly bound in k in the following Lemma.
\begin{lem}\label{bound on measure}
	Let $u^k$ be a kinetic solution to \eqref{1.1} with initial condition $u_0^k$. Then, for all $R\,\textgreater\,0$,
	$$\mathbb{E}|m^k([0,T]\times\mathbb{T}^N\times[-R,R])|^2\,\le\,C(T,R,\|u_0\|_{L^1(\mathbb{T}^N)}^2),$$
	and
	$$\mathbb{E}\big[\sup_{t\in[0,T]}\|u^k(t)\|_{L^1(\mathbb{T^N)}}\big]\le\,C,$$
	for some $C\,\textgreater\,0$, indepenent of k.
\end{lem}

\begin{proof}
	For $R\,\textgreater\,0,$ set
	$$\kappa_R(u)=\mathbbm{1}_{[-R,R]}(u)\,\,\,\,\,\,\kappa_R(u)=\int_{-R}^u\int_{-R}^r\kappa_R(s)ds dr.$$
	We take $\varphi(x,\mathfrak{\zeta})=\kappa_R'(\mathfrak{\zeta})$ in \eqref{formulation} to obtain
	\begin{align}
	\int_{\mathbb{T}^N}\kappa_R(u^k(x,t)) dx&=\int_{\mathbb{T}^N}\kappa_R(u_0^k(x))dx + \sum_{k\ge1}\int_0^t\int_{\mathbb{T}^N}h_k(x)\kappa_R'(u^k(x,s)) dx d\beta_j(s)\notag\\
	&\qquad+\frac{1}{2}\int_0^t\int_{\mathbb{T}^N} H^2(x)\kappa_R(u^k(x,s)) dx ds-\int_{\mathcal{B}_R}d\mathfrak{m}^k(x,s,t),
	\end{align}
	where $\mathcal{B}_R=\mathbb{T}^N\times[0,t]\times\{\mathfrak{\zeta}\in\mathbb{R},R\,\le\,\mathfrak{\zeta}\,\le\,R+1\}$. 
	\begin{align}
	&\mathbb{E}\bigg|\int_{\mathbb{T}^N}\kappa_R(u^k(x,t)) dx\bigg|^2+\mathbb{E}\bigg|\int_{\mathcal{B}_R}d\mathfrak{m}^k(x,t,\mathfrak{\zeta})\bigg|^2\le\mathbb{E}\bigg|\int_{\mathbb{T}^N}\kappa_R(u_0^k))dx\bigg|^2\notag\\&\qquad+\frac{1}{2}\mathbb{E}\bigg|\int_0^t\int_{\mathbb{T}^N}H^2(x) \kappa_R(u^k(x,s))dx\bigg|^2 +\mathbb{E}\bigg|\sum_{k\ge1}\int_0^t\int_{\mathbb{T}^N}h_k(x)\kappa_R'(u^k(x,s))dx d\beta_j(s)\bigg|^2\notag
	\\
	\end{align}
	Since $0\,\le\,\kappa_R(u)\,\le\,2R(R+|u|)$ and $H^2(x)\kappa_R(u^k(x,s))\,\le\,D$, we have
	$$\mathbb{E}\bigg|\int_{\mathbb{T}^N}\kappa_R(u_0^k))dx\bigg|^2 +\frac{1}{2}\mathbb{E}\bigg|\int_0^t\int_{\mathbb{T}^N}H^2(x) \kappa_R(u^k(x,s))dx\bigg|^2\,\le\,C(R^2,T^2,D,\|u_0\|_{L^1(\mathbb{T}^N)}^2)$$
	Since $0\,\le\,\kappa_R'(u)\,\le\,2R$, we estimate the stochastic integral using It\^o isometry as follows
	\begin{align*}
	\mathbb{E}\bigg|\sum_{k\ge1}\int_0^t\int_{\mathbb{T}^N} h_k(x) \kappa_R'(u^k(x,s))dx d\beta_j(s) \bigg|^2&=\mathbb{E}\bigg(\int_0^t\sum_{k\ge1}\bigg(\int_{\mathbb{T}^N}h_k(x)\kappa_R'(u^k(x,s))dx\bigg)^2ds\bigg)\\
	&\,\le\,4R^2\mathbb{E}\bigg(\int_0^T\bigg(\int_{\mathbb{T}^N}\big(\sum_{k\ge1}|h_k(x)|^2\big)^{\frac{1}{2}}dx\bigg)^2\bigg)\\
	&\,\le\,4R^2TC_0, 		
	\end{align*}
	This finishes proof of first part. For second, we can take constant $\varphi(\mathfrak{\zeta})=sign(\mathfrak{\zeta})$ as test function in \eqref{formulation} (justified by standard approximation), then we can easily prove second part as previous calculation.
\end{proof}	
\noindent	
\textbf{Limit of kinetic measures $m^k$:} Suppose that $\mathcal{M}(\mathcal{B}_R)$ is the collection of bounded Borel measure over $\mathcal{B}_R$ with norm given by the total variation of measures. It is dual space of $C(\mathcal{B}_R)$, the collection of continuous functions on $\mathcal{B}_R$. Since $\mathcal{M}(\mathcal{B}_R)$ is separable, the space $L^2(\Omega; \mathcal{M}(\mathcal{B}_R))$ is the dual space of $L^2(\Omega; C(\mathcal{B}_R))$. Let us discuss convergence of $m^k$. By Lemma \ref{bound on measure}, we have for  $R\in\mathbb{N}$
\begin{align}\label{uniform bound on measure}
	\sup_k\mathbb{E}|m^k(\mathcal{B}_R)|^2\,\le\,C_R,
\end{align}
where $\mathcal{B}_R=\mathbb{T}^N\times[0,T]\times[-R,R]$.
Bound \eqref{uniform bound on measure} gives a uniform bound on $(m^k)$ in $L^2(\Omega,\mathcal{M}(\mathcal{B}_R))$. There exists $m_R\in L^2(\Omega;\mathcal{M}(\mathcal{B}_R))$ such that
up to subsequence, $m^k\,\to\,m_R$ in $L^2(\Omega;\mathcal{M}(\mathcal{B}_R))$-weak star. By a diagonal process, we get, for $R\in\mathbb{N}$, $m_R=m_{R+1}$ in $L^2(\Omega; \mathcal{M}(\mathcal{B}_R))$ and the convergence of a single subsequence still denoted by $(m^k)$ in all the spaces $L^2(\Omega;\mathcal{M}(\mathcal{B}_R))$- weak-*. Let us define $\mathbb{P}$-almost surely, $\mathfrak{m}=m_R$ on $\mathcal{B}_R$. The condition (i) of Definition \ref{kinetic measure 2} is direct consequence of weak convergence, therefore satisfied by ${\mathfrak{m}}$. Now, condition (ii) of Definition \ref{kinetic measure 2} remains to prove. 

\noindent
\textbf{Kinetic solution:} Now, we have all in hand to apply the method as developed in subsection \ref{subsection 3.2} and pass to limit  in \eqref{formulation-1}, we get a $L^1(\mathbb{T}^N)$-valued process $(\tilde{u}(t))_{t\in[0,T]}$ such that $(\tilde{u}(t))_{t\in[0,T]}$ and $f(t)=\mathbbm{1}_{\tilde{u}(t)\textgreater\zeta}$ satisfy $\mathbb{P}$-almost surely, for all $t\in[0,T]$
\begin{align}\label{formulation-2}
\langle f(t),\varphi \rangle &= \langle f_0, \varphi \rangle + \int_0^t\langle f(s),F'(\mathfrak{\zeta})\cdot\nabla\varphi\rangle ds -\int_0^t\langle f(s), A'(\mathfrak{\zeta})\,(-\Delta)^\alpha[\varphi]\rangle ds\notag\\
&\qquad+\sum_{k=1}^\infty \int_0^t \int_{\mathbb{T}^N} h_k(x)\varphi(x,\tilde{u}(x,s) dx d\beta_k(s)\notag\\
&\qquad+\frac{1}{2}\int_0^t\int_{\mathbb{T}^N}\partial_{\mathfrak{\zeta}}\varphi(x,\tilde{u}(x,s)) H^2 (x)dxds -\mathfrak{m}(\partial_{\mathfrak{\zeta}}\varphi)([0,t]),
\end{align}
where $m\,\ge\,\mathfrak{\eta}_1$, $\mathbb{P}$-almost surely, where $\mathfrak{\eta}_1$ defined as in Definition \ref{definition kinetic solution in l1 setting}. For all $\varphi\in C_c^2(\mathbb{T}^N\times\mathbb{R}),$ $\mathbb{P}$-almost surely, $t\to \langle f(t),\varphi\rangle$ is c\'adl\'ag.

\noindent
\textbf{Decay of kinetic measure:}
In the existence part of the result, decay of kinetic measure and continuity of solution is remaining to prove. For that, we will prove the following proposition.
\begin{proposition}\label{22}
	Let $u_0\in L^1(\mathbb{T}^N)$. Let a measurable function $u:\mathbb{T}^N\times[0,T]\times\Omega\,\to\,\mathbb{R}$ be a solution to \eqref{1.1}. Let $T\,\textgreater\,0$,  There exists a decreasing function $\varepsilon:\mathbb{R}_+\,\to\,\mathbb{R}_+$ with $\lim_{n\to\infty}\varepsilon(n)=0$ depending on T and on the functions 
	$$n\mapsto\|(u_0-n)^+\|_{L^1(\mathbb{T}^N)},\,\,\,\,\,\,\, n\mapsto\|(u_0-n)^-\|_{L^1(\mathbb{T}^N)}$$
	only such that, for all $n\,\textgreater 1$,
	\begin{align}\label{tightness}
	\mathbb{E}\big(\sup_{t\in[0,T]}\|(u(t)-n)^{\pm}\|_{L^1(\mathbb{T}^N)}\big) + \mathbb{E}\mathfrak{m}(\mathcal{B}_n)\,\le\,C(T,C_0,\|u_0\|_{L^1(\mathbb{T}^N)})(\varepsilon(n-1)+\varepsilon^{1/2}(n)),
	\end{align}
	where $\mathcal{B}_n=\mathbb{T}^N\times[0,T]\times\{\mathfrak{\zeta}\in\mathbb{R},n\,\le\,|\mathfrak{\zeta}|\,\le n+1\}$.
\end{proposition}
\begin{proof}
	For $R\,\textgreater\,0$, set
	$$\kappa_R(u)=\mathbbm{1}_{R\,\textless\,u\,\textless R+1},\,\,\,\,\,\,\,\kappa_R(u)=\int_0^u\int_0^r\kappa_R(s)ds dr.$$
	We take $$\varphi(x,\mathfrak{\zeta})=\kappa_R'(\mathfrak{\zeta})$$ in \eqref{formulation} to obtain,$\mathbb{P}$-almost surely, for all $t\in[0,T]$
	\begin{align}
	\int_{\mathbb{T}^N}\kappa_R(u(x,t)) dx&=\int_{\mathbb{T}^N}\kappa_R(u_0(x))dx + \sum_{j\ge1}\int_0^t\int_{\mathbb{T}^N}g_j(x)\kappa_R'(u(x,s)) dx d\beta_j(s)\notag\\
	&\qquad+\frac{1}{2}\int_0^t\int_{\mathbb{T}^N} H^2(x)\kappa_R(u(x,s)) dx ds-\int_{\mathcal{B}_R}d\mathfrak{m}(x,s,t),
	\end{align}
	Taking  then expectation, we have, for all $t\in[0,T]$
	\begin{align}\label{m}
	\mathbb{E}\int_{\mathbb{T}^N}\kappa_R(u(x,t)) dx&=\int_{\mathbb{T}^N}\kappa_R(u_0(x))dx 
	+\frac{1}{2}\mathbb{E}\int_0^t\int_{\mathbb{T}^N} H^2(x)\kappa_R(u(x,s)) dx ds-\mathbb{E}\int_{\mathcal{B}_R}d\mathfrak{m}(x,s,t),
	\end{align}
	\begin{align}\label{r}
	\mathbb{E}\int_{\mathbb{T}^N}\kappa_R(u(x,t)) dx&\,\le\,\int_{\mathbb{T}^N}\kappa_R(u_0(x))dx 
	+\frac{1}{2}\mathbb{E}\int_0^t\int_{\mathbb{T}^N} H^2(x)\kappa_R(u(x,s)) dx ds
	\end{align}
	Note that
	$$(u-(R+1))^+\,\le\,\kappa_R(u)\,\le\,(u-R)^+,$$
	for all $R\,\ge\,0,\,u\in\mathbb{R}$. Note also
	\begin{align}\label{inequality}
	\kappa_R(u)\,\le\,\mathbbm{1}_{R\,\le\,u}\,\le\,\frac{(u-R+r)^+}{r},\,\,\,\,\,\,\, r\,\textgreater\,0.
	\end{align}
	\eqref{r} gives
	\begin{align*}
	\mathbb{E}\int_{\mathbb{T}^N}(u(x,t)-(R+1))^+dx\,&\le\,\frac{C_0}{2r}\mathbb{E}\int_0^t\int_{\mathbb{T}^N}(u(x,t)-R+r)^+ dx + \int_{\mathbb{T}^N}(u_0(x)-R)^+ dx.\\
	&\le\,\frac{C_0}{2r}\mathbb{E}\int_0^t\int_{\mathbb{T}^N}(u(x,t)-R+r)^+ dx+\int_{\mathbb{T}^N}(u_0(x)-R+r)^+ dx.
	\end{align*}
	Choose $r=C_0+1$,then
	\begin{align*}
	\mathbb{E}\int_{\mathbb{T}^N}(u(x,t)-(R+1))^+dx\,&\le\,\frac{1}{2}\mathbb{E}\int_0^t\int_{\mathbb{T}^N}(u(x,t)-R+C_0+1)^+ dx+\int_{\mathbb{T}^N}(u_0(x)-R+C_0+1)^+ dx,
	\end{align*}
	we take $R=nC_0-1$,\,\, then we have following inequality,
	\begin{align*}
	\mathbb{E}\int_{\mathbb{T}^N}(u(x,t)-nC_0)^+ dx\,\le\,\frac{1}{2}\mathbb{E}\int_0^t\int_{\mathbb{T}^N}(u(x,t)-(n-1)C_0)^+ dx+\int_{\mathbb{T}^N}(u_0(x)-(n-1)C_0)^+ dx.
	\end{align*}
	Denote by $h_n(t)$ the function
	$$h_n(t)=\mathbb{E}\int_{\mathbb{T}^N}(u(x,t)-nC_0)^+ dx,\,\,\,\forall t\in[0,T],$$
	then, we finally obtain
	$$h_n(t)\,\le\,\frac{1}{2}\int_0^t h_{n-1}(s)ds + h_{n-1}(0).$$
	Since $u^+\,\le\,(u-C_0)^+ + C_0,$ we have
	$$h_0(t)=\mathbb{E}\int_{\mathbb{T}^N}u^+(x)dx\,\le\,\mathbb{E}\int_{\mathbb{T}^N}(u(x,t)-C_0)^+ dx + C_0= h_1(t)+C_0$$
	After this, we can follow the proof of \cite[Proposition 15]{vovelle 2} to conclude the result.
\end{proof}

\noindent
\textbf{Kinetic measure:} By estimate \eqref{tightness}, it is clear that ${\mathfrak{m}}$ satiesfies decay codition $(ii)$ of Definition \ref{definition kinetic measure}. So ${\mathfrak{m}}$ is kinetic measure.
\subsection{Uniqueness:}\label{Section 3.2}
 Note that both Propositions \ref{Proposition 3.1}, \& \ref{th3.5} also hold in present $L^1$-setting. Here we state only the statement of the contraction principle. We are not giving proof of result, because proof of the uniqueness result is similar to the proof of Theorem \ref{th3.6}. 
\begin{thm}\label{new}
	Let $u_0\in L^1(\mathbb{T}^N)$. Let u be a kinetic solution to \eqref{1.1}, then there exists $L^1(\mathbb{T}^N)$-valued process $(u^{-}(t))_{t\in[0,T]}$ such that $\mathbb{P}$-almost surely, for all $t\in[0,T]$, $f^{-}(t)=\mathbbm{1}_{u^{-}(t)\textgreater\xi}$. Moreover, if $u_1$,\,\,$u_2$\, are kinetic solutions to \eqref{1.1} with initial data $u_{1,0} $ and $u_{2,0}$ respectively, then for all $t\in[0,T]$, we have $\mathbb{P}$-almost surely,
	\begin{align}
	\|u^1(t)-u^2(t)\|_{L^1(\mathbb{T}^N)}\,\le\,\|u_{1,0}-u_{2,0}\|_{L^1(\mathbb{T}^N)},
	\end{align}
\end{thm}

\noindent
\textbf{Continuity in time:}
Based on the equi-integrability estimate \eqref{tightness} we can deduce the kinetic solutions have almost surely continuous paths in $L^1(\mathbb{T}^N)$. 
\begin{cor}
	Let $u_0\in L^1(\mathbb{T}^N)$ and let $u$ be a kinetic solution to $\eqref{1.1}$. Then $u$ has almost surely continuous trajectories in $L^1(\mathbb{T}^N).$
\end{cor}
\begin{proof}
	Based on Proposition \ref{Proposition 3.1}, Theorem \ref{th3.6} and Proposition \ref{22}, we are in a position to apply \cite[Lemma 17]{vovelle 2}, which implies the continuity in $L^1(\mathbb{T}^N)$. Indeed, let us first show that $\mathbb{P}$-almost surely, $u$ is  right continuous in $L^1(\mathbb{T}^N)$. From the Proposition \ref{Proposition 3.1} , we know that $f(t+\delta_n)\to f(t)$ in $L^\infty(\mathbb{T}^N\times\mathbb{R})$ weak-* $\mathbb{P}$-almost surely as $\delta_n\to0$. Finally, \eqref{tightness} implies
$$\lim_{n\to\infty}\mathbb{E}\sup_{t\in[0,T]}\|(u(t)-n)^{\pm}\|_{L^1(\mathbb{T}^N)}=0.$$
Therefore, there exists a subsequence which converges $\mathbb{P}$-almost surely, that is, $\mathbb{P}$-almost surely,
$$\lim_{n\to\infty} \sup_{t\in[0,T]}\|(u(t)-n)^{\pm}\|_{L^1(\mathbb{T}^N)}=0.$$
Consequently, \cite[Lemma 17]{vovelle 2} applies and gives the following convergence $\mathbb{P}$-almost surely,
$$u(t+\delta_n)\,\to\,u(t)\,\,\,\,\,\text{in}\,\,\,\,L^1(\mathbb{T}^N)\,\,\,\,\text{as}\,\,\,\,\,\delta_n\,\to\,0.$$
By the same arguments we can show that $\mathbb{P}$-almost surely, $u^{-}$, constructed in Theorem \ref{new}, is left-continuous in $L^1(\mathbb{T}^N)$. 
\end{proof}
\subsection*{Acknowledgments}
The author wishes to thank Ujjwal Koley for many stimulating discussions and valuable suggestions. The author also thanks Martina Hofmanov\'a for enlightening discussions on this topic during her visit to TIFR-CAM.

\end{document}